%% file: sn-article.tex
\documentclass[pdflatex,sn-mathphys-num]{sn-jnl}

\usepackage{pgf}%
\usepackage{subcaption}%
\usepackage{graphicx}%
\usepackage{bm}%
\usepackage{multirow}%
\usepackage{amsmath,amssymb,amsfonts}%
\usepackage{amsthm}%
\usepackage{mathrsfs}%
\usepackage[title]{appendix}%
\usepackage{xcolor}%
\usepackage{textcomp}%
\usepackage{manyfoot}%
\usepackage{booktabs}%
\usepackage{algorithm}%
\usepackage{algorithmicx}%
\usepackage{algpseudocode}%
\usepackage{listings}%
\usepackage{cleveref}%

\newcommand{\mtx}[1]{\bm{#1}}%
\newcommand{\vct}[1]{\bm{#1}}%
\DeclareMathOperator{\Trace}{Tr}%

\theoremstyle{thmstyleone}%
\newtheorem{theorem}{Theorem}%
\newtheorem{lemma}[theorem]{Lemma}%
\crefname{equation}{}{}

\begin{document}

\title[Kernel-based linear system identification using augmented Krylov subspaces]{Kernel-based linear system identification using augmented Krylov subspaces}

\author*[1]{\fnm{Fabio} \sur{Matti}}\email{fabio.matti@epfl.ch}

\author[2]{\fnm{Martin Skovgaard} \sur{Andersen}}\email{mskan@dtu.dk}

\author[3]{\fnm{Tianshi} \sur{Chen}}\email{tschen@cuhk.edu.cn}

\author[1]{\fnm{Daniel} \sur{Kressner}}\email{daniel.kressner@epfl.ch}

\affil[1]{\orgdiv{Institute of Mathematics}, \orgname{EPFL}, \orgaddress{\city{Lausanne}, \postcode{1015}, \country{Switzerland}}}

\affil[2]{\orgdiv{Department of Applied Mathematics and Computer Science}, \orgname{Technical University of Denmark}, \orgaddress{\city{Kongens Lyngby}, \postcode{2800}, \country{Denmark}}}

\affil[3]{\orgdiv{School of Data Science and Shenzhen
Research Institute of Big Data}, \orgname{The Chinese University of Hong Kong}, \orgaddress{\city{Shenzhen}, \postcode{518172}, \country{China}}}

\abstract{We propose a novel Krylov subspace method for estimating the finite impulse response (FIR) of a one-dimensional linear time-invariant systems. The method approximates the system's FIR using a kernel-based formulation combined with hyperparameter selection based on maximum likelihood estimation (MLE), which requires repeated evaluation of two terms:
The data fit $\vct{y}^{\top} (\lambda \mtx{I} + \mtx{A})^{-1} \vct{y}$ and the model complexity $\log(\det (\lambda \mtx{I} + \mtx{A}))$, where $\mtx{A}$ is a certain positive semidefinite matrix that admits fast matrix--vector products and $\lambda > 0$ is a regularization parameter.
Instead of approximating these two quantities separately,
we jointly approximate them using a single augmented Krylov subspace for $\mtx{A}$. One major benefit of augmentation is that we obtain accelerated convergence when approximating the data fit quadratic form, through implicit preconditioning. Thanks to the shift invariance of Krylov subspaces, the extracted approximations can be used to evaluate the MLE objective for many values of $\lambda$ at little additional cost. We derive error bounds for the approximations, reflecting the benefits of augmentation demonstrated through multiple numerical experiments.}



\maketitle

\section{Introduction}

This paper is concerned with reconstructing signals from noisy data. In particular, we consider single-input-single-output, causal, and linear time-invariant systems as they appear in the context of audio processing~\cite{naylor-2010-speech-dereverberation}, aircraft design~\cite{ljung-1998-system-identification}, and industrial process control~\cite{andersson-1995-estimation-residence,ljung-1998-system-identification}. Those assume that the output $\vct{y} \in \mathbb{R}^{m}$ is related to the input $\vct{u} \in \mathbb{R}^{m}$ through a relation of the form
\begin{equation}
    \vct{y} = \mtx{\Phi} \vct{\theta} + \vct{e}.
    \label{equ:linear-model}
\end{equation}
Here, $\mtx{\Phi} \in \mathbb{R}^{m \times n}$ is a Toeplitz matrix with entries $\Phi_{ij} = u_{i-j}$ if $i > j$ and $\Phi_{ij} = 0$ otherwise, $\vct{\theta} \in \mathbb{R}^{n}$ is the unknown impulse response, and $\vct{e} \in \mathbb{R}^{m}$ represents additive noise, independent of the input. We assume that $m \geq n$.

Following~\cite{chen-2012-estimation-transfer,pillonetto-2010-new-kernelbased}, we use a Bayesian approach for reconstructing the impulse response $\vct{\theta}$, given an input-output pair $(\vct{u}, \vct{y})$. To this extent, we assume independent Gaussian priors
\begin{equation}
    \vct{\theta} \sim \mathcal{N}(\vct{0}, \nu \mtx{K}(\beta)) \quad \text{and} \quad \vct{e} \sim \mathcal{N}(\vct{0}, \sigma^2 \mtx{I})
    \label{equ:bayesian-model}
\end{equation}
with $\nu, \sigma > 0$, and a kernel matrix $\mtx{K}(\beta) \in \mathbb{R}^{n \times n}$ constructed from kernels characterized by one or a few parameters $\beta$. Throughout this work, $\mtx{I}$ denotes an identity matrix of appropriate size. The goal is to find the set of parameters $\nu$, $\beta$, and $\sigma^2$ which best fit the output. Commonly, this is done through Expectation-Maximization~\cite{suzuki-2025-direct-bayesian}, maximum a posteriori estimation (MAP)~\cite{chung-2025-efficient-hyperparameter,chung-2017-generalized-hybrid}, and maximum likelihood estimation, also referred to as empirical Bayes~\cite{chen-2025-fast-kernelbased, chen-2023-scalable-kernelbased,chen-2012-estimation-transfer}.
We will use the last approach: Under the model~\cref{equ:bayesian-model}, the output $\vct{y}$ behaves as a centered multivariate Gaussian with covariance matrix $\sigma^2 \mtx{I} + \nu \mtx{A}(\beta)$, where
\begin{equation}
  \mtx{A}(\beta) := \mtx{\Phi} \mtx{K}(\beta) \mtx{\Phi}^{\top} \in \mathbb{R}^{m \times m}.
  \label{equ:matrix-definition}
\end{equation}
To find the optimal values of the parameters $\nu$, $\beta$, and $\sigma^2$, we minimize the corresponding negative log-likelihood
\begin{equation}
  \frac{1}{2} \left( \vct{y}^{\top} \left(\sigma^2 \mtx{I} + \nu \mtx{A}(\beta)\right)^{-1} \vct{y} + \log\left(\det(\sigma^2 \mtx{I} + \nu \mtx{A}(\beta))\right) + m \log(2\pi) \right).
  \label{equ:pml-original}
\end{equation}
Following~\cite[Section II.B]{chen-2023-scalable-kernelbased}, we reparametrize $\sigma^2$ to $\lambda = \sigma^2 / \nu > 0$. As a consequence, the negative log-likelihood~\cref{equ:pml-original} becomes convex with respect to $\nu^{-1}$ and it admits the unique minimizer $\nu^{\ast} = \vct{y}^{\top} (\lambda \mtx{I} + \mtx{A}(\beta))^{-1} \vct{y} / m$. Thus, inserting $\nu^{\ast}$ into~\cref{equ:pml-original} and dropping constant terms, we get the equivalent profile marginal log-likelihood (PML) criterion
\begin{equation}
  \psi_{\text{PML}}(\lambda, \beta) =  \log\left( \vct{y}^{\top} (\lambda \mtx{I} + \mtx{A}(\beta))^{-1} \vct{y} \right) + \frac{1}{m} \Trace\left(\log(\lambda \mtx{I} + \mtx{A}(\beta))\right),
  \label{equ:pml-criterion}
\end{equation}
where the matrix identity $\log(\det) = \Trace(\log)$ is used. To find the minimum of the PML criterion~\cref{equ:pml-criterion}, we use Bayesian optimization. Bayesian optimizers are gradient-free optimizers which are well suited for approximating a black-box objective function that is expensive to evaluate and possibly non-convex and stochastic. To find a minimum, these optimizers often need to evaluate the objective function for dozens of parameter values $(\lambda, \beta)$, each involving the computation of a quadratic form with an inverted matrix and the trace of the logarithm of a matrix---two very costly operations, which may take seconds to compute even for moderately sized problems ($m = \mathcal{O}(10^4)$, $n = \mathcal{O}(10^3)$). Once the optimal parameters $\nu^{\ast}$, $\beta^{\ast}$, and $(\sigma^{\ast})^2$ have been found, the estimated impulse response $\vct{\theta}^{\ast}$ is taken to be the posterior mean given $\vct{y}$ and these optimal parameters~\cite[II.A]{chen-2025-fast-kernelbased}.

For the kernel matrix involved in the Gaussian prior~\cref{equ:bayesian-model} of $\vct{\theta}$,
the authors of~\cite{chen-2025-fast-kernelbased} consider semiseparable matrices in factorized form $\mtx{K}(\beta) = \mtx{L}(\beta) \mtx{L}(\beta)^{\top}$, where both
the Cholesky factor $\mtx{L}(\beta)$ and its transpose can be applied to a vector in $\mathcal{O}(n)$ operations. Examples of such kernels are the stable spline kernel~\cite{pillonetto-2010-new-kernelbased} and the TC and DC kernels~\cite{chen-2012-estimation-transfer}. The \emph{direct algorithm} proposed in~\cite[Section III.A]{chen-2025-fast-kernelbased} for computing the PML criterion~\cref{equ:pml-criterion} first computes a factorization of $\mtx{A}(\beta)$ defined in~\cref{equ:matrix-definition} for one value of $\beta$, exploiting its structure. After this preprocessing step, the evaluation of the PML criterion~\cref{equ:pml-criterion} becomes very cheap, even for many values of $\lambda$. This is
clearly beneficial for optimizing~\cref{equ:pml-criterion}. However, the SVD needed in the preprocessing step is costly and limits the problem size.
Therefore,~\cite[Section III.B]{chen-2025-fast-kernelbased} also propose a so-called \emph{indirect algorithm}, which approximates the PML criterion~\cref{equ:pml-criterion} using an iterative least-squares solver combined with the Girard--Hutchinson trace estimator, which both can in principle be applied to larger-sized problems. However, unlike the direct algorithm, this approach cannot effectively reuse computations when evaluating~\cref{equ:pml-criterion} for many values of $\lambda$ at a fixed $\beta$.

In this work, we propose a novel method that combines the advantages of the direct and indirect algorithms, as detailed in~\Cref{sec:exist}. It is as scalable as the indirect algorithm while still offering similar benefits as the direct algorithm in the context of optimization, allowing for the cheap repeated evaluations for multiple values of $\lambda$.
Key components of our approach are the shift-invariance and nestedness of certain Krylov subspaces. In~\Cref{sec:kryl}, we show how these properties are leveraged to produce fast and accurate evaluations of the PML criterion~\cref{equ:pml-criterion}. In~\Cref{sec:the}, we give theoretical guarantees for our method and compare its performance with the two existing methods in~\Cref{sec:num} with multiple numerical experiments.

\paragraph{Other related work} More general Bayesian inverse problems of the form~\cref{equ:linear-model} are well studied in the literature; see, e.g.,~\cite{dashti-2017-bayesian-approach}. For example, the work~\cite{chung-2017-generalized-hybrid} uses an iterative procedure to compute the MAP estimate of $\vct{\theta}$ associated with problem~\cref{equ:linear-model}.
In contrast to this work,~\cite{chung-2017-generalized-hybrid} considers a more general Gaussian prior for the noise $\vct{e}$ and a fixed covariance kernel matrix $\mtx{K}(\beta) \equiv \mtx{K}$ for $\vct\theta$ that is assumed to not be easily factorizable, e.g., arising from a Matérn kernel with vector-valued inputs or a dictionary collection. Moreover, it selects the regularization parameter $\lambda$ on-the-fly during the iterative procedure. More recent work~\cite{chung-2025-efficient-hyperparameter} suggests to determine the optimal parameters $\beta$ and $\nu$ in~\cref{equ:bayesian-model} as the MAP estimates of the marginalized posterior distribution of these parameters. In this context, the log-determinant term is approximated with a preconditioned variant of stochastic Lanczos quadrature~\cite[Section 2.2]{ubaru-2017-fast-estimation}. Common to Monte Carlo type methods, this estimator is potentially inhibited by slow convergence.

\paragraph{Reproducibility} The code for this paper can be found in the repository \url{https://github.com/FMatti/krylov-augmented}, which extends the \texttt{gprfire} MATLAB package described in~\cite{chen-2025-fast-kernelbased}.

\section{Existing approaches for evaluating PML criterion~\cref{equ:pml-criterion}}
\label{sec:exist}

To simplify notation, we suppress the dependence of $\mtx K$ and related matrices on the kernel parameter $\beta$ from now on.

In this section, we recall the direct and indirect algorithms from~\cite{chen-2025-fast-kernelbased}. A key requirement of these algorithms is the availability of a
factorization $\mtx{K} = \mtx{L} \mtx{L}^{\top}$
for which the Cholesky factor $\mtx{L}$ is highly structured (and implicitly represented), such that the application of $\mtx{L}$ and $\mtx{L}^{\top}$ to a vector can be carried out in $\mathcal{O}(n)$ operations. Examples for kernels admitting such a representation are stable spline kernels as well as TC and DC kernels~\cite{andersen-2020-smoothing-splines}.

\subsection{Direct algorithm}
\label{subsec:direct}

To evaluate the PML criterion~\cref{equ:pml-criterion}, the direct algorithm from~\cite[Section III.A]{chen-2025-fast-kernelbased}
first computes for $\tilde{\mtx{\Phi}} = \mtx{\Phi} \mtx{L} \in \mathbb R^{m\times n}$ an economy-size SVD $\tilde{\mtx{\Phi}} = \mtx{U} \mtx{S} \mtx{V}^{\top}$, where the diagonal matrix $\mtx{S} = \operatorname{diag}(s_1, \dots, s_n)$ contains the singular values $s_1 \ge  \dots \ge s_n \ge 0$.
Using that $\mtx{A} = \mtx{\Phi} \mtx{K} \mtx{\Phi}^\top = \tilde{\mtx{\Phi}} \tilde{\mtx{\Phi}}^{\top} = \mtx{U} S^2 \mtx{U}^\top$, one obtains the expressions
\begin{equation}
  \vct{y}^{\top} (\lambda \mtx{I} + \mtx{A})^{-1} \vct{y} = \frac{\lVert \vct{y} \rVert _2^2 - \lVert \tilde{\vct{y}} \rVert _2^2}{\lambda} + \sum_{i=1}^{n} \frac{\tilde{y}_i^2}{s_i^2 + \lambda},
  \label{equ:inv-quadratic-direct}
\end{equation}
where $\tilde{\vct{y}} = \mtx{U}^{\top} \vct{y}$, and
\begin{equation}
  \Trace(\log(\lambda \mtx{I} + \mtx{A})) = (m - n) \log(\lambda) + \sum_{i=1}^{n} \log(s_i^2 + \lambda).
  \label{equ:log-trace-direct}
\end{equation}
These expressions make the evaluation of the PML criterion~\cref{equ:pml-criterion} inexpensive for different values of the regularization parameter $\lambda$, requiring only $\mathcal{O}(n)$ operations per value. However, the economy-size SVD for $\tilde{\mtx{\Phi}}$ needed in the preprocessing step requires $\mathcal{O}(mn^2)$ operations and $\mathcal{O}(mn)$ memory, which significantly limits the computationally feasible problem size.

\subsection{Indirect algorithm}
\label{subsec:indirect}

The indirect algorithm for approximating the PML criterion~\cref{equ:pml-criterion} proceeds iteratively, indirectly accessing $\mtx{L}$ and $\mtx{\Phi}$ through matrix--vector products with these matrices and their transposes. This allows one to conveniently benefit from the structure of $\mtx{L}$ mentioned above as well as the Toeplitz structure of $\mtx{\Phi}$, reducing the complexity of matrix--vector products with $\mtx{\Phi}$ to $\mathcal{O}((n + m) \log(n + m))$ via the fast Fourier transform (FFT).

The indirect algorithm employs the Sherman--Morrison--Woodbury formula to rewrite
\begin{equation*}
  \vct{y}^{\top} (\lambda \mtx{I} + \mtx{A})^{-1} \vct{y} = \lambda^{-1} \vct{y}^{\top}(\mtx{I} - \tilde{\mtx{\Phi}} (\lambda \mtx{I} + \tilde{\mtx{\Phi}}^{\top} \tilde{\mtx{\Phi}})^{-1} \tilde{\mtx{\Phi}}^{\top}) \vct{y},
\end{equation*}
where we recall that $\tilde{\mtx{\Phi}} = \mtx{\Phi} \mtx{L}$. The expression on the right essentially comes down to
the regularized least-squares problem $(\lambda \mtx{I} + \tilde{\mtx{\Phi}}^{\top} \tilde{\mtx{\Phi}}) \vct{x} =  \tilde{\mtx{\Phi}}^{\top} \vct{y}$, which is solved
by the LSQR algorithm~\cite{paige-1982-lsqr-algorithm} accelerated by a randomized Nyström preconditioner $\mtx{P}$; see, e.g.,~\cite{frangella-2023-randomized-nystrom}.
To address the trace term in~\cref{equ:pml-criterion}, the indirect algorithm uses the Weinstein--Aronszajn identity and repurposes the Nyström preconditioner $\mtx{P}$ to write
\begin{align*}
  \Trace(\log(\lambda \mtx{I} + \mtx{A}))
  &= \Trace(\log(\lambda \mtx{I} + \tilde{\mtx{\Phi}}^{\top} \tilde{\mtx{\Phi}})) + (m - n) \log(\lambda) \\
  &= \Trace(\log(\mtx{P}^{-1/2} (\lambda \mtx{I} + \tilde{\mtx{\Phi}}^{\top} \tilde{\mtx{\Phi}}) \mtx{P}^{-1/2})) + \Trace(\log(\mtx{P})) + (m - n) \log(\lambda).
\end{align*}
While $\Trace(\log(\mtx{P}))$ can be computed in $\mathcal{O}(n)$ operations, the first term is approximated using an adaptively truncated Mercator series expansion $\log(1 - x) = - \sum_{k=1}^{\infty} x^k / k!$ and the Girard--Hutchinson trace estimator~\cite{hutchinson-1990-stochastic-estimator, girard-1989-fast-montecarlo}; see~\cite[Algorithm 1]{chen-2023-scalable-kernelbased} for details.

Based entirely on matrix--vector products, the indirect algorithm is well suited for larger problem sizes. However, in contrast to the direct algorithm, the approximation needs to be recomputed from scratch when $\lambda$ changes, a significant disadvantage during the optimization process that needs to evaluate~\cref{equ:pml-criterion} for many different values of $\lambda$ and $\beta$.

\section{Krylov-augmented algorithm for evaluating PML criterion~\cref{equ:pml-criterion}}
\label{sec:kryl}

In this section, we describe our new procedure for evaluating the PML criterion~\cref{equ:pml-criterion}.
Like the indirect algorithm, it scales well to larger problem sizes, and, like the direct algorithm, it is cheap to recompute~\cref{equ:pml-criterion}
for different values of $\lambda$ after a preprocessing step.

\subsection{Approximation by augmentation of Krylov subspaces}

Again for fixed $\beta$, we consider the matrix $\mtx{A} = \mtx{\Phi} \mtx{K} \mtx{\Phi}^{\top}$ from~\cref{equ:matrix-definition}. We will make use of
(block) Krylov subspaces of the form
\begin{equation}
  \mathcal{K}_k(\mtx{A}, \mtx{Z}) = \operatorname{span}\{ \mtx{Z}, \mtx{A} \mtx{Z}, \dots, \mtx{A}^{k-1} \mtx{Z}\},
  \label{equ:block-krylov-space}
\end{equation}
with particular vectors or matrices $\mtx{Z}$.
An orthonormal basis $\mtx{W}_k$ of $\mathcal{K}_k(\mtx{A}, \mtx{Z})$, along with the
block tridiagonal matrix $\mtx{T}_k = \mtx{W}_k^{\top} \mtx{A} \mtx{W}_k$,
can be computed with the block Lanczos method \cite[Section 9.2.6]{golub-2013-matrix-computations}. The usual precautions for block Lanczos methods are in order: Reorthogonalization and block size adaptation.
For the latter, we use the procedure described in~\cite[Section 2.1]{zhou-2008-block-krylov} that adaptively shrinks the block size to avoid linear dependencies; see also \Cref{alg:block-lanczos}.
While our implementation uses this procedure, which indeed results in reduced block sizes, we assume throughout the following that the block size remains constant, for the sake of simplifying the description.
\begin{algorithm}
\caption{Block Lanczos method}
\label{alg:block-lanczos}
\begin{algorithmic}[1]

\Require Symmetric matrix $\mtx{A}$, starting block $\mtx{Z}$
\Require Number of iterations $k$, truncation tolerance $\tau > 0$
\Ensure Orthonormal basis $\mtx{W}_k$ of $\mathcal{K}_k(\mtx{A}, \mtx{Z})$ and $\mtx{T}_k = \mtx{W}_k^{\top} \mtx{A} \mtx{W}_k$

\State Compute QR-factorization $\mtx{Z} = \mtx{Q} \mtx{R}$
\State $\mtx{\mathcal{W}}_0 \gets \mtx{0}$, $\mtx{N}_1 \gets \mtx{0}$, $\mtx{\mathcal{W}}_1 \gets \mtx{Q}$

\For{$i = 1, \dots, k$}
    \State $\mtx{Y} \gets \mtx{A} \mtx{\mathcal{W}}_i - \mtx{\mathcal{W}}_{i - 1} \mtx{N}_i^{\top}$
    \State $\mtx{M}_i \gets \mtx{\mathcal{W}}_i^{\top}\mtx{Y}$
    \State $\mtx{Y} \gets \mtx{Y} - \mtx{\mathcal{W}}_i \mtx{M}_i$

    \For{$j = 1, \dots, i - 1$}
        \State $\mtx{Y} \gets \mtx{Y} - \mtx{\mathcal{W}}_j (\mtx{\mathcal{W}}_j^{\top} \mtx{Y})$
    \EndFor

    \State Compute pivoted QR-factorization $\displaystyle
    \mtx{Y} =
    \begin{bmatrix} \mtx{Q}_{1} & \mtx{Q}_2 \end{bmatrix}
    \begin{bmatrix}
        \mtx{R}_{11} & \mtx{R}_{12} \\
        \mtx{0} & \mtx{R}_{22}
    \end{bmatrix}
    \mtx{\Pi}^{\top}
    $

    $\rhd$ diagonal elements of $\mtx{R}_{22}$ have magnitude $< \tau$

    \State $\mtx{\mathcal{W}}_{i+1} \gets \mtx{Q}_1$
    \State $\mtx{N}_{i+1} \gets \begin{bmatrix} \mtx{R}_{11} & \mtx{R}_{12} \end{bmatrix} \mtx{\Pi}^{\top}$
\EndFor

\State $\mtx{T}_k \gets \textsc{block-tridiagonal}([\mtx{N}_2, \dots, \mtx{N}_k], [\mtx{M}_1, \dots, \mtx{M}_k], [\mtx{N}_2^{\top}, \dots, \mtx{N}_k^{\top}])$
\State $\mtx{W}_k \gets [\mtx{\mathcal{W}}_1, \dots, \mtx{\mathcal{W}}_k]$

\end{algorithmic}
\end{algorithm}

It is well known that multiplying a matrix with a block of $b > 1$ vectors can have computational advantages compared to multiplying the same matrix with the $b$ vectors individually. For the type of matrices considered in this work, this benefit of bundling vectors is clearly visible in~\Cref{fig:mat-vec-time}, even for small values of $b$. In the block Lanczos method, this effect partly offsets the additional computational complexity caused when increasing the block size, that is, the number of columns of $\mtx{Z}$.
\begin{figure}[!ht]
  \centering
  \input{plots/matvec_runtimes.pgf}
  \caption{Average runtime per vector for multiplying the matrix $\mtx{A} = \mtx{\Phi} \mtx{K} \mtx{\Phi}^{\top} \in \mathbb{R}^{m \times m}$ with a block of $b$ vectors, as a function of the matrix size $m$. We use a FIR length of $n = m/5$ and a TC kernel to generate $\mtx{A}$ \cite{chen-2012-estimation-transfer}. We plot the mean of 1'000 repetitions and shade the range of one standard deviation from the mean. The runtime is non-monotonic in the matrix size $m$ because the FFT used for $\mtx{\Phi}$ behaves differently depending on how well the input length is factorizable.}
  \label{fig:mat-vec-time}
\end{figure}
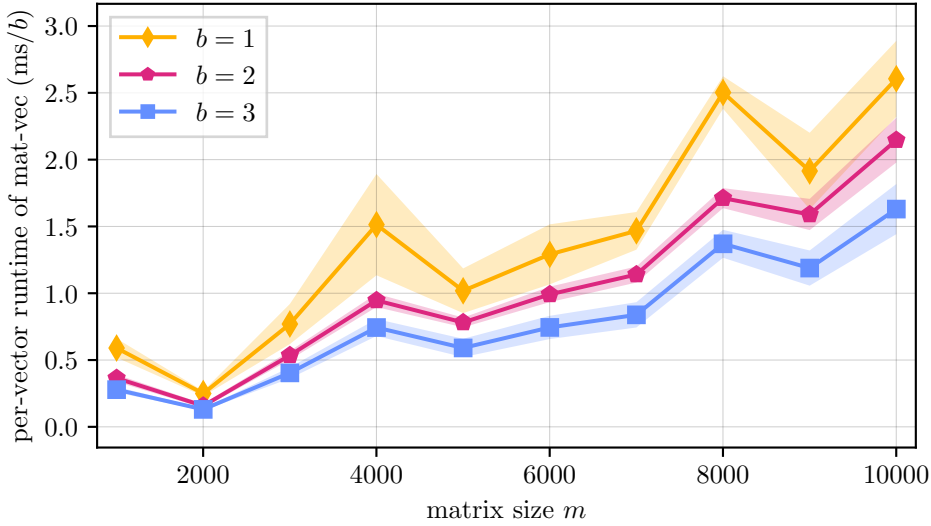

Given an orthonormal Krylov subspace basis $\mtx{W}_k$,
the quadratic form involved in the PML criterion~\cref{equ:pml-criterion} can be approximated by
\begin{equation}
  \vct{y}^{\top} (\lambda \mtx{I} + \mtx{A})^{-1} \vct{y} \approx \vct{y}^{\top} \mtx{W}_k (\lambda \mtx{I} + \mtx{T}_k)^{-1} \mtx{W}_k^{\top} \vct{y}.
  \label{equ:iqf-approximation}
\end{equation}
A typical choice for the Krylov subspace~\cref{equ:block-krylov-space} in this context is $\mtx{Z} = \vct{y}$~\cite[Section 7.2]{golub-2009-matrices-moments}. We consider the approximation
\begin{equation}
  \Trace(\log(\lambda \mtx{I} + \mtx{A})) \approx \Trace(\log(\lambda \mtx{I} + \mtx{W}_k \mtx{T}_k \mtx{W}_k^{\top}))
  \label{equ:trace-approximation}
\end{equation}
for the block Krylov subspace~\cref{equ:block-krylov-space} with $\mtx{Z} = \vct{\Omega}$ for some random matrix $\vct{\Omega}$; see, e.g.,~\cite{li-2021-randomized-block};
related techniques for approximating the trace of the matrix logarithm  in~\cref{equ:pml-criterion} have also been proposed in~\cite{cortinovis-2022-randomized-trace,li-2021-randomized-block,saibaba-2017-randomized-matrixfree,ubaru-2017-fast-estimation}.

Instead of performing the two approximations~\cref{equ:iqf-approximation,equ:trace-approximation} with two separate Krylov subspaces, we propose to combine them by setting $\mtx{Z} = [\vct{y}, \mtx{\Omega}]$. Hence, we compute an orthonormal basis $\widehat{\mtx{W}}_k$ and the corresponding compression $\widehat{\mtx{T}}_k = \widehat{\mtx{W}}_k^{\top} \mtx{A} \widehat{\mtx{W}}_k$ for the \emph{augmented} Krylov subspace
\begin{equation*}
  \mathcal{K}_{k}(\mtx{A}, [\vct{y}, \mtx{\Omega}]) = \operatorname{span}\{ \vct{y}, \mtx{\Omega}, \mtx{A} \vct{y}, \mtx{A} \mtx{\Omega}, \dots, \mtx{A}^{k-1} \vct{y}, \mtx{A}^{k-1} \mtx{\Omega} \} = \mathcal{K}_{k}(\mtx{A}, \vct{y}) + \mathcal{K}_{k}(\mtx{A}, \mtx{\Omega}).
\end{equation*}
We will see in~\Cref{sec:the,sec:num} that the augmented approximations~\cref{equ:iqf-approximation,,equ:trace-approximation} not only match the accuracy of the standard approximations, but usually yield significantly better results, both theoretically and in practice.
Moreover, by building the combined augmented Krylov subspace instead of running two separate block Lanczos algorithms (\Cref{alg:block-lanczos}), we achieve a noticeable speed up because of the benefits of bundling vectors discussed above and illustrated in~\Cref{fig:mat-vec-time}.

Further, the Krylov subspace~\cref{equ:block-krylov-space} is shift-invariant in the sense that $\mathcal{K}_k(\lambda \mtx{I} + \mtx{A}, \mtx{Z}) = \mathcal{K}_k(\mtx{A}, \mtx{Z})$ for every $\lambda \in \mathbb{R}$. In particular, this allows us to run~\Cref{alg:block-lanczos} only once to extract $\widehat{\mtx{W}}_k$ and $\widehat{\mtx{T}}_k$, which can then be used to cheaply form the approximations~\cref{equ:iqf-approximation,equ:trace-approximation} for several values of $\lambda$.

\subsection{Fast evaluation of approximations~\cref{equ:iqf-approximation,equ:trace-approximation}}
\label{subsec:pml-computation}

To speed up the evaluation of the approximations~\cref{equ:iqf-approximation,equ:trace-approximation} for multiple values of $\lambda$, we adapt the techniques developed for the direct algorithm; see~\Cref{subsec:direct}.
We first compute a spectral decomposition $\widehat{\mtx{T}}_k = \mtx{V} \mtx{\Theta} \mtx{V}^{\top}$, where the diagonal matrix $\mtx{\Theta} = \operatorname{diag}(\theta_1, \dots, \theta_{k (n_{\mtx{\Omega}} + 1)})$ contains the eigenvalues $\theta_1 \ge  \dots \ge \theta_{k (n_{\mtx{\Omega}} + 1)} \ge 0$. This allows us to rewrite the approximations \cref{equ:iqf-approximation,equ:trace-approximation} as
\begin{equation*}
  \vct{y}^{\top} \widehat{\mtx{W}}_k  (\lambda \mtx{I} + \widehat{\mtx{T}}_k)^{-1} \widehat{\mtx{W}}_k^{\top} \vct{y} = \sum_{i=1}^{k (n_{\mtx{\Omega}} + 1)} \widetilde{y}_i^2 (\lambda + \theta_i)^{-1},
\end{equation*}
with the elements $\widetilde{y}_i$ of $\widetilde{\vct{y}} = \mtx{V}^{\top} \widehat{\mtx{W}}_k^{\top} \vct{y}$, and
\begin{equation*}
  \Trace(\log(\lambda \mtx{I} + \widehat{\mtx{W}}_k \widehat{\mtx{T}}_k \widehat{\mtx{W}}_k^{\top})) = (m - k (n_{\mtx{\Omega}} + 1))\log(\lambda) + \sum_{i=1}^{k (n_{\mtx{\Omega}} + 1)} \log(\lambda + \theta_i).
\end{equation*}

\Cref{alg:krylov-augmented} summarizes the obtained procedure for approximating the PML criterion~\cref{equ:pml-criterion}.

\begin{algorithm}
\caption{Krylov-augmented algorithm for PML criterion}
\label{alg:krylov-augmented}
\begin{algorithmic}[1]

\Require Positive semidefinite matrix $\mtx{A}$, vector $\vct{y}$
\Require Block Lanczos iterations $k$, augmentation size $n_{\mtx{\Omega}}$
\Ensure Approximation of the PML criterion $\psi_{\mathrm{PML}}(\lambda)$ defined in~\cref{equ:pml-criterion}

\State Sample Gaussian random matrix $\mtx{\Omega} \in \mathbb{R}^{m \times n_{\mtx{\Omega}}}$
\State $(\mtx{T}_k, \mtx{W}_k) \gets \textsc{BlockLanczos}(\mtx{A}, [\vct{y}, \mtx{\Omega}], k)$ \Comment{using \Cref{alg:block-lanczos}}
\State Compute eigenvalue decomposition $\mtx{T}_k = \mtx{V} \mtx{\Theta} \mtx{V}^{\top}$
\State $\vct{\theta} \gets \operatorname{diag}(\mtx{\Theta})$
\State $\widetilde{\vct{y}} \gets \mtx{V}^{\top} \mtx{W}_k^{\top} \vct{y}$

\State \Return $(1 - \frac{k(n_{\mtx{\Omega}} + 1)}{m}) \log(\lambda) + \sum_{i=1}^{k (n_{\mtx{\Omega}} + 1)} \left( \widetilde{y}_i^2 (\lambda + \theta_i)^{-1} + \frac{\log(\lambda + \theta_i)}{m} \right)$

\end{algorithmic}
\end{algorithm}

\subsection{Residual trace estimation}
\label{subsec:residual-trace}

The decay of the eigenvalues of the kernel matrix $\mtx{K}$ depends on what kernel function is used and how it is parametrized. If the eigenvalues do not decay as rapidly, it may happen that the trace approximation~\cref{equ:trace-approximation} alone is not sufficiently accurate for reasonable values of $k$ and $n_{\mtx{\Omega}}$. To address this, we introduce an additional component to the algorithm, which aims to improve the original approximation~\cref{equ:trace-approximation} in a manner analogous to the Hutch++ estimator~\cite{lin-2017-randomized-estimation, meyer-2021-hutch-optimal}.

Specifically, we apply the Girard--Hutchinson trace estimator~\cite{girard-1989-fast-montecarlo,hutchinson-1990-stochastic-estimator} to the residual of the approximation~\cref{equ:trace-approximation}: We draw $n_{\mtx{\Psi}}$ additional standard Gaussian random vectors $\vct{\psi}_1, \dots, \vct{\psi}_{n_{\mtx{\Psi}}} \in \mathbb{R}^{m}$ and construct the estimator
\begin{equation}
  \Trace(\mtx{R}(\lambda)) \approx \frac{1}{n_{\mtx{\Psi}}} \sum_{i=1}^{n_{\mtx{\Psi}}} \vct{\psi}_i^{\top} \mtx{R}(\lambda) \vct{\psi}_i,
  \label{equ:residual-trace}
\end{equation}
with the symmetric positive semidefinite matrix $\mtx{R}(\lambda) = \log(\lambda \mtx{I} + \mtx{A}) - \log(\lambda \mtx{I} + \widehat{\mtx{W}}_k \widehat{\mtx{T}}_k \widehat{\mtx{W}}_k^{\top})$. We approximate the quadratic forms $\vct{\psi}_i^{\top} \log(\lambda \mtx{I} + \mtx{A}) \vct{\psi}_i$ and $\vct{\psi}_i^{\top} \log(\lambda \mtx{I} + \widehat{\mtx{W}}_k \widehat{\mtx{T}}_k \widehat{\mtx{W}}_k^{\top}) \vct{\psi}_i$, $i=1, \dots, n_{\mtx{\Psi}}$, using stochastic Lanczos quadrature~\cite{ubaru-2017-fast-estimation}. Although more direct methods are available for evaluating $\vct{\psi}_i^{\top} \log(\lambda \mtx{I} + \widehat{\mtx{W}}_k \widehat{\mtx{T}}_k \widehat{\mtx{W}}_k^{\top}) \vct{\psi}_i$~\cite[Theorem 1.35]{higham-2008-functions-matrices}, we observe numerically that approximating both terms with the same scheme is crucial for achieving higher accuracy. Motivated by \cite[Theorem 2.1]{matti-2025-stochastic-trace}, we reuse the same set of random vectors $\vct{\psi}_i$, $i=1, \dots, n_{\mtx{\Psi}}$ for every value of $\lambda$. A batched implementation of the Lanczos method, also called loop-interchange Lanczos method, allows us to exploit the computational benefits of bundled matrix--vector products shown in \Cref{fig:mat-vec-time}. Further, in an analogous manipulation as in \Cref{subsec:pml-computation}, the resulting approximations can again be expressed such that it is cheap to evaluate for different values of $\lambda$.

Once the residual estimate~\cref{equ:residual-trace} is computed, it is added as a correction to the trace estimate~\cref{equ:trace-approximation}.

\section{Theoretical analysis}
\label{sec:the}

The purpose of this section is to provide theoretical insight into the approximation returned by~\Cref{alg:krylov-augmented}. In particular,
we derive error bounds for the quadratic form approximation~\cref{equ:iqf-approximation} and the trace estimate~\cref{equ:trace-approximation}. For each approximation, we first analyze the case without augmentation and subsequently provide an argument that augmentation can only improve accuracy.

\subsection{Analysis of the quadratic form approximation~\cref{equ:iqf-approximation}}

In this section, to simplify notation, we absorb the regularization term $\lambda \mtx{I}$ into the matrix $\mtx{A}$ and, consequently, $\mtx A\gets \lambda \mtx{I} + \mtx A$ becomes positive definite.

\subsubsection{Without augmentation}

We first state and prove an upper bound on the approximation~\cref{equ:cg-qf-approximation} without augmentation.
Assuming that $\mtx{W}_k$ is the orthonormal basis of the Krylov subspace $\mathcal{K}_k(\mtx{A}, \vct{y})$ computed by the Lanczos method, the compressed matrix $\mtx{T}_k = \mtx{W}_k^{\top} \mtx{A} \mtx{W}_k$ is tridiagonal and the approximation~\cref{equ:iqf-approximation} can be expressed as
\begin{equation}
  \vct{y}^{\top} \mtx{A}^{-1} \vct{y} \approx \vct{y}^{\top} \mtx{W}_k \mtx{T}_k^{-1} \mtx{W}_k^{\top} \vct{y} = \lVert \vct{y} \rVert _2^2 (\mtx{T}_k^{-1})_{11}.
  \label{equ:cg-qf-approximation}
\end{equation}
Here, $(\mtx{T}_k^{-1})_{11} = \vct{e}_1^{\top} \mtx{T}_k^{-1} \vct{e}_1$ denotes the top left entry in the matrix $\mtx{T}_k^{-1}$.

\begin{lemma}
\label{lem:cg-bound}
With the notation introduced above, it holds that
\begin{equation*}
  0 \leq \frac{\vct{y}^{\top} \mtx{A}^{-1} \vct{y} - \lVert \vct{y} \rVert _2^2 (\mtx{T}_k^{-1})_{11}}{\vct{y}^{\top} \mtx{A}^{-1} \vct{y}} \leq 4 \left( \frac{\sqrt{\kappa} - 1}{\sqrt{\kappa} + 1} \right)^{2k},
\end{equation*}
where $\kappa$ is the condition number of $\mtx{A}$.
\end{lemma}
\begin{proof}
The vector $\vct{x}_k = \lVert \vct{y} \rVert _2 \mtx{W}_k \mtx{T}_k^{-1} \vct{e}_1$ coincides with the approximation
obtained by applying $k$ iterations of the conjugate gradient method to the linear system $\mtx{A} \vct{x} = \vct{y}$ with starting vector $\vct{x}_0 = \vct{0}$.
These iterations are well known to satisfy
\begin{equation}
  \lVert \vct{x} - \vct{x}_k \rVert _{\mtx{A}}^2 \leq 4 \lVert \vct{x} \rVert _{\mtx{A}}^2 \left( \frac{\sqrt{\kappa} - 1}{\sqrt{\kappa} + 1} \right)^{2k};
  \label{equ:cg-bound}
\end{equation}
see, e.g.,~\cite[Theorem 3.1.1]{greenbaum-1997-iterative-methods}. Using the definition of the $\mtx{A}$-norm, this error can be rewritten as
\begin{align*}
  \lVert \vct{x} - \vct{x}_k \rVert _{\mtx{A}}^2 &= \vct{x}^{\top} \mtx{A} \vct{x} - 2 \vct{x}^{\top} \mtx{A} \vct{x}_k + \vct{x}_k^{\top} \mtx{A} \vct{x}_k = \vct{y}^{\top} \mtx{A}^{-1} \vct{y} - 2 \vct{y}^{\top} \vct{x}_k + \vct{x}_k^{\top} \mtx{A} \vct{x}_k \\
  &= \vct{y}^{\top} \mtx{A}^{-1} \vct{y} - 2 \lVert \vct{y} \rVert _2 \underbrace{ \vct{y}^{\top}  \mtx{W}_k}_{=~\lVert \vct{y} \rVert _2\vct{e}_1^{\top}} \mtx{T}_k^{-1} \vct{e}_1 + \lVert \vct{y} \rVert _2^2  \vct{e}_1^{\top} \mtx{T}_k^{-1} \underbrace{\mtx{W}_k^{\top} \mtx{A} \mtx{W}_k}_{=~\mtx{T}_k} \mtx{T}_k^{-1} \vct{e}_1.
\end{align*}
Consequently,
\begin{equation}
  \lVert \vct{x} - \vct{x}_k \rVert _{\mtx{A}}^2 = \vct{y}^{\top} \mtx{A}^{-1} \vct{y} - \lVert \vct{y} \rVert _2^2 \vct{e}_1^{\top} \mtx{T}_k^{-1} \vct{e}_1 = \vct{y}^{\top} \mtx{A}^{-1} \vct{y} - \lVert \vct{y} \rVert _2^2 (\mtx{T}_k^{-1})_{11},
  \label{equ:cg-connection-to-qf}
\end{equation}
or in other words, the squared $\mtx{A}$-norm-error of the conjugate gradient iterates is the same as the residual of the quadratic form approximation. Hence, inserting~\cref{equ:cg-connection-to-qf} into the standard conjugate gradient result~\cref{equ:cg-bound} with $\lVert \vct{x} \rVert _{\mtx{A}}^2 = \vct{x}^{\top} \mtx{A} \vct{x} = \vct{y}^{\top} \mtx{A}^{-1} \vct{y}$ shows the bound.
\end{proof}

\subsubsection{With augmentation}

We now analyze the use of an augmented Krylov subspace $\mathcal{K}_k(\mtx{A}, [\vct{y}, \mtx{\Omega}])$ for some matrix $\mtx{\Omega} \in \mathbb{R}^{m \times n_{\mtx{\Omega}}}$
in the quadratic form approximation~\cref{equ:iqf-approximation}. For this purpose, we consider the orthonormal basis $\widehat{\mtx{W}}_k$ and
the block tridiagonal matrix $\widehat{\mtx{T}}_k$ computed with the block Lanczos method applied to $\mathcal{K}_k(\mtx{A}, [\vct{y}, \mtx{\Omega}])$. Then an expression analogous to~\cref{equ:cg-qf-approximation} holds when $\mtx{W}_k$ and $\mtx{T}_k$ are replaced with $\widehat{\mtx{W}}_k$ and $\widehat{\mtx{T}}_k$:
\begin{equation}
  \vct{y}^{\top} \mtx{A}^{-1} \vct{y} \approx \vct{y}^{\top} \widehat{\mtx{W}}_k \widehat{\mtx{T}}_k^{-1} \widehat{\mtx{W}}_k^{\top} \vct{y} = \lVert \vct{y} \rVert _2^2 (\widehat{\mtx{T}}_k^{-1})_{11}.
  \label{equ:cg-qf-augmented-approximation}
\end{equation}
The following result shows that this augmented approximation can only improve the approximation.

\begin{lemma}[Augmentation does not harm]
\label{lem:cg-bound-augmented}
With the notation introduced above, it holds that
\begin{equation*}
  0 \leq \vct{y}^{\top} \mtx{A}^{-1} \vct{y} - \lVert \vct{y} \rVert _2^2 (\widehat{\mtx{T}}_k^{-1})_{11} \leq \vct{y}^{\top} \mtx{A}^{-1} \vct{y} - \lVert \vct{y} \rVert _2^2 (\mtx{T}_k^{-1})_{11}.
\end{equation*}
\end{lemma}

\begin{proof}
The conjugate gradient iterates $\vct{x}_k = \lVert \vct{y} \rVert _2 \mtx{W}_k \mtx{T}_k^{-1} \vct{e}_1$ and $\widehat{\vct{x}}_k = \lVert \vct{y} \rVert _2 \widehat{\mtx{W}}_k \widehat{\mtx{T}}_k^{-1} \vct{e}_1$ satisfy the inequality
\begin{equation}
  \lVert \vct{x} - \widehat{\vct{x}}_k \rVert _{\mtx{A}}^2 = \min_{\vct{u} \in \mathcal{K}_k(\mtx{A}, [\vct{y}, \mtx{\Omega}])} \lVert \vct{x} - \vct{u} \rVert _{\mtx{A}}^2 \leq \min_{\vct{u} \in \mathcal{K}_k(\mtx{A}, \vct{y})} \lVert \vct{x} - \vct{u} \rVert _{\mtx{A}}^2 = \lVert \vct{x} - \vct{x}_k \rVert _{\mtx{A}}^2,
  \label{equ:cg-error-monotonicity}
\end{equation}
where we used $\mathcal{K}_k(\mtx{A}, \vct{y}) \subseteq \mathcal{K}_k(\mtx{A}, [\vct{y}, \mtx{\Omega}])$.
Following the proof of~\Cref{lem:cg-bound}, an analogous expression for the $\mtx{A}$-norm-error of the approximation~\cref{equ:cg-qf-augmented-approximation} can be established:
\begin{equation}
  \vct{y}^{\top} \mtx{A}^{-1} \vct{y} - \lVert \vct{y} \rVert _2^2 (\widehat{\mtx{T}}_k^{-1})_{11} = \lVert \vct{x} - \widehat{\vct{x}}_k \rVert _{\mtx{A}}^2.
  \label{equ:acg-connection-to-qf}
\end{equation}
Inserting~\cref{equ:acg-connection-to-qf} along with~\cref{equ:cg-connection-to-qf} into~\cref{equ:cg-error-monotonicity} shows the claim.
\end{proof}

\Cref{lem:cg-bound-augmented} merely establishes that augmentation does no harm. Theorem 5 in~\cite{oleary-1980-block-conjugate} shows that augmentation enjoys an improved  bound of the form~\Cref{lem:cg-bound}, with a reduced condition number $\kappa$. Following the developments in~\cite{chen-2025-preconditioning-preconditioner}, we provide a different interpretation, relating augmentation to a certain class of preconditioners.

\begin{theorem}[Implicit preconditioning by augmentation]
\label{thm:implicit-preconditioning}
With the notation introduced above, it holds that
  \begin{equation*}
    0 \leq \frac{\vct{y}^{\top} \mtx{A}^{-1} \vct{y} - \lVert \vct{y} \rVert _2^2 (\widehat{\mtx{T}}_k^{-1})_{11}}{\vct{y}^{\top} \mtx{A}^{-1} \vct{y}} \leq 4 \left( \frac{\sqrt{\widetilde{\kappa}} - 1}{\sqrt{\widetilde{\kappa}} + 1} \right)^{2(k-s)},
  \end{equation*}
  where $\widetilde{\kappa}$ is the condition number of the ``preconditioned'' matrix $\mtx{P}^{-1/2} \mtx{A} \mtx{P}^{-1/2}$ for any preconditioner $\mtx{P} = (\mtx{I} + \mtx{X})^{-1}$ with $\operatorname{range}(\mtx{X}) \subseteq \mathcal{K}_{s+1}(\mtx{A}, \mtx{\Omega})$ for some $s \leq k$.
\end{theorem}

\begin{proof}
By~\cite[Theorem 3.2]{chen-2025-preconditioning-preconditioner}, the properties of $\mtx{P} = (\mtx{I} + \mtx{X})^{-1}$ imply that
\begin{equation*}
  \mathcal{K}_{k-s}(\mtx{P}^{-1} \mtx{A}, \mtx{P}^{-1} \vct{y}) \subseteq \mathcal{K}_{k}(\mtx{A}, [\vct{y}, \mtx{\Omega}]).
\end{equation*}
As in the proof of~\Cref{lem:cg-bound-augmented}, we obtain the inequality
\begin{equation*}
  \lVert \vct{x} - \widehat{\vct{x}}_k \rVert _{\mtx{A}}^2 = \min_{\vct{u} \in \mathcal{K}_k(\mtx{A}, [\vct{y}, \mtx{\Omega}])} \lVert \vct{x} - \vct{u} \rVert _{\mtx{A}}^2 \leq \min_{\vct{u} \in \mathcal{K}_{k-s}(\mtx{P}^{-1} \mtx{A}, \mtx{P}^{-1} \vct{y})} \lVert \vct{x} - \vct{u} \rVert _{\mtx{A}}^2 = \lVert \vct{x} - \widetilde{\vct{x}}_{k-s} \rVert _{\mtx{A}}^2,
\end{equation*}
where $\widetilde{\vct{x}}_{k-s}$ is the $(k-s)$th iterate of the preconditioned conjugate gradient method with preconditioner $\mtx{P}$. Note that, equivalently,
$\mtx{P}^{1/2} \widetilde{\vct{x}}_{k-s}$ is the $(k-s)$th iterate
of the conjugate gradient method applied to $\mtx{P}^{-1/2} \mtx{A} \mtx{P}^{-1/2} \vct z =
\mtx{P}^{-1/2} \vct y$.
Applying~\cref{equ:cg-bound}, we obtain
\begin{equation*}
  \lVert \vct{x} - \widetilde{\vct{x}}_{k-s} \rVert _{\mtx{A}}^2 \leq 4 \lVert \vct{x} \rVert _{\mtx{A}}^2 \left( \frac{\sqrt{\widetilde{\kappa}} - 1}{\sqrt{\widetilde{\kappa}} + 1} \right)^{2(k-s)}.
  \label{equ:pcg-bound}
\end{equation*}
Together with~\cref{equ:acg-connection-to-qf}, the claim follows.
\end{proof}

There are many constructions of matrices $\mtx{P}$ that satisfy the requirements of~\Cref{thm:implicit-preconditioning}. For matrices $\mtx{A}$ whose eigenvalue decay rapidly---as observed for the matrices of interest in our application (cf.~\cref{equ:matrix-definition})---a particularly suitable choice are Nyström preconditioners
\begin{equation*}
  \mtx{P}^{-1} = \mtx{I} + \mtx{U}\left( C(\Lambda + c \mtx{I})^{-1} - \mtx{I} \right)\mtx{U}^{\top}, ~C>c>0,
\end{equation*}
where $\mtx{U} \mtx{\Lambda} \mtx{U}^{\top}$ is a (truncated) eigendecomposition of a Nyström approximation of $\mtx{A}$. The approximation is constructed using a sketching matrix whose columns span $\mathcal{K}_{s+1}(\mtx{A}, \mtx{\Omega})$; see~\cite{chen-2025-preconditioning-preconditioner,frangella-2023-randomized-nystrom}. By construction, the columns of $\mtx{U}$ lie in $\mathcal{K}_{s+1}(\mtx{A}, \mtx{\Omega})$, so indeed the conditions of \Cref{thm:implicit-preconditioning} are satisfied.
For a Gaussian random matrix $\mtx{\Omega} \in \mathbb{R}^{n \times n_{\mtx{\Omega}}}$,~\cite[Theorem 4.4]{chen-2025-preconditioning-preconditioner} shows that, for sufficiently large $n_{\mtx{\Omega}}$ and $s$, the condition number $\widetilde{\kappa}$ of the preconditioned matrix $\mtx{P}^{-1/2} \mtx{A} \mtx{P}^{-1/2}$ remains moderate with high probability.

\subsection{Analysis of the trace estimate~\cref{equ:trace-approximation}}

We return to the original definition~\cref{equ:matrix-definition}
of the symmetric positive semidefinite matrix $\mtx{A}$
and treat the regularization $\lambda \mtx{I}$ with $\lambda > 0$ separately instead of absorbing it into $\mtx A$.

\subsubsection{Without augmentation}
\label{subsubsec:log-trace-approx}

In this section, we consider the approximation~\cref{equ:trace-approximation}, that is,
\begin{equation*}
  \Trace(\log(\lambda \mtx{I} + \mtx{A})) \approx \Trace(\log(\lambda \mtx{I} + \mtx{W}_k \mtx{T}_k \mtx{W}_k^{\top})),
  \label{equ:log-trace-approximation}
\end{equation*}
with $\mtx{W}_k$ being an orthonormal basis of the Krylov subspace $\mathcal{K}_k(\mtx{A}, \mtx{\Omega})$ and $\mtx{T}_k = \mtx{W}_k^{\top} \mtx{A} \mtx{W}_k$. The following error bound follows from~\cite[Theorem 3.2]{li-2021-randomized-block} with some minor adaptations. It uses Chebyshev polynomials to construct a polynomial with controlled growth on the spectrum of $\mtx{A}$. This polynomial is then used to represent elements in the Krylov subspace and its properties help bound several quantities involving it in terms of the trailing eigenvalues of $\mtx{A}$.
\begin{lemma}
  \label{lem:log-trace-standard}
  Using the notation from above, let $n_{\mtx{\Omega}} = q + p$ for some $q,p \geq 2$. Then, with probability at least $1 - \delta$,
  \begin{align*}
    0 &\leq \Trace(\log(\lambda \mtx{I} + \mtx{A})) - \Trace(\log(\lambda \mtx{I} + \mtx{W}_k \mtx{T}_k \mtx{W}_k^{\top})) \\
    &\leq \Trace(\log(\mtx{I} + C\frac{\lambda_{q+1}}{\lambda_q} T_{k-2}^{-2}\left(\frac{2\lambda_q - \lambda_{q+1}}{\lambda_{q+1}}\right) \lambda^{-1} \mtx{\Lambda}_2)) + \Trace(\log(\mtx{I} + \lambda^{-1} \mtx{\Lambda}_2)),
  \end{align*}
  where $C = \left(\sqrt{n-q} + \sqrt{n_{\mtx{\Omega}}} + \sqrt{2\log(\frac{2}{\delta})}  \right)^2 (\frac{2}{\delta})^{\frac{2}{p+1}} \left( \frac{e \sqrt{n_{\mtx{\Omega}}}}{p + 1}\right)^2$, $T_{k-2}$ the $(k-2)$th Chebyshev polynomial, and $\mtx{\Lambda}_2 = \operatorname{diag}(\lambda_{q+1}, \dots, \lambda_m)$, where $\lambda_i$ denotes the $i$th largest eigenvalue of $\mtx{A}$.
\end{lemma}

\begin{proof}
  We have
  \begin{equation*}
    \Trace(\log(\lambda \mtx{I} + \mtx{A})) - \Trace(\log(\lambda \mtx{I} + \mtx{W}_k \mtx{T}_k \mtx{W}_k^{\top}))
    = \Trace(\log(\mtx{I} + \lambda^{-1} \mtx{A})) - \Trace(\log(\mtx{I} + \lambda^{-1} \mtx{T}_k ))
  \end{equation*}
  to which we can apply a version\footnote{Due to a different convention for the block Krylov subspace~\cref{equ:block-krylov-space} in \cite{li-2021-randomized-block}, we needed to replace $T_{k-1}$ with $T_{k-2}$ in the bound.} of~\cite[Theorem 3.2]{li-2021-randomized-block} along with the identity $\Trace(\log) = \log(\det)$ to show the claim.
\end{proof}

\subsubsection{With augmentation}
\label{subsubsec:log-trace-approx-aug}

In this section, we show that replacing $\mtx{W}_k$ and $\mtx{T}_k$ with their augmented counterparts $\widehat{\mtx{W}}_k$, the orthonormal basis of $\mathcal{K}_k(\mtx{A}, [\vct{y}, \mtx{\Omega}])$, and $\widehat{\mtx{T}}_k = \widehat{\mtx{W}}_k^{\top} \mtx{A} \widehat{\mtx{W}}_k$, can only improve the accuracy of the approximation~\cref{equ:trace-approximation}. For this purpose, we will make use of the following lemma.

\begin{lemma}
  \label{lem:eigval-order}
  Let $\mtx{A} \in \mathbb{R}^{m \times m}$ be symmetric positive semidefinite. Consider orthonormal bases $\mtx{Q} \in \mathbb{R}^{m \times \ell}$ and $\widehat{\mtx{Q}} \in \mathbb{R}^{m \times \widehat{\ell}}$ with
  $\ell \leq \widehat{\ell} \leq m$ and $\operatorname{range}(\mtx{Q}) \subseteq \operatorname{range}(\widehat{\mtx{Q}})$. Then
  \begin{equation*}
    \lambda_i(\mtx{Q} \mtx{Q}^{\top} \mtx{A} \mtx{Q} \mtx{Q}^{\top})
    \leq \lambda_i(\widehat{\mtx{Q}} \widehat{\mtx{Q}}^{\top} \mtx{A} \widehat{\mtx{Q}} \widehat{\mtx{Q}}^{\top}), \quad i = 1, \dots, m,
  \end{equation*}
  where $\lambda_i(\cdot)$ denotes the $i$th largest eigenvalue of a matrix.
\end{lemma}
\begin{proof}
  The minimax characterization of eigenvalues~\cite[Theorem 4.2.6]{horn-1985-matrix-analysis} immediately gives
  \begin{equation*}
    \lambda_i(\mtx{Q}^{\top} \mtx{A} \mtx{Q})
    = \max_{\substack{\operatorname{dim}(\mathcal{W}) = i \\ \mathcal{W} \subseteq \operatorname{range}(\mtx{Q})}} \min_{\vct{x} \in \mathcal{W}}~\vct{x}^{\top} \mtx{A} \vct{x}
    \leq \max_{\substack{\operatorname{dim}(\mathcal{W}) = i \\ \mathcal{W} \subseteq \operatorname{range}(\widehat{\mtx{Q}})}} \min_{\vct{x} \in \mathcal{W}}~\vct{x}^{\top} \mtx{A} \vct{x}
    = \lambda_i(\widehat{\mtx{Q}}^{\top} \mtx{A} \widehat{\mtx{Q}}).
  \end{equation*}
  This completes the proof because the eigenvalues of $\mtx{Q} \mtx{Q}^{\top} \mtx{A} \mtx{Q} \mtx{Q}^{\top}$ are the
eigenvalues of  $\mtx{Q}^{\top} \mtx{A} \mtx{Q}$ appended with $m-\ell$ zero eigenvalues (and an analogous statement holds for the eigenvalues of $\widehat{\mtx{Q}} \widehat{\mtx{Q}}^{\top} \mtx{A} \widehat{\mtx{Q}} \widehat{\mtx{Q}}^{\top}$).
\end{proof}

\begin{lemma}[Augmentation does not harm]
  \label{lem:log-trace-augmented}
  Using the notation introduced above, it holds that
  \begin{equation*}
    \Trace(\log(\lambda \mtx{I} + \mtx{W}_k \mtx{T}_k \mtx{W}_k^{\top} ))
    \leq \Trace(\log(\lambda \mtx{I} + \widehat{\mtx{W}}_k \widehat{\mtx{T}}_k \widehat{\mtx{W}}_k^{\top}))
    \leq \Trace(\log(\lambda \mtx{I} + \mtx{A})).
  \end{equation*}
  provided that $k ( n_{\mtx{\Omega}}+1) \leq m$.
\end{lemma}
\begin{proof}
Using $\operatorname{range}(\mtx{W}_k) \subseteq \operatorname{range}(\widehat{\mtx{W}}_k) \subseteq \operatorname{range}(\mtx{I})$
together with~\Cref{lem:eigval-order}, and the monotonicity of the function $x \mapsto \log(\lambda + x)$, we obtain the result:
  \begin{align*}
    & \Trace(\log(\lambda \mtx{I} + \mtx{W}_k \mtx{W}_k^{\top} \mtx{A} \mtx{W}_k \mtx{W}_k^{\top})) \\
    =&  \sum_{i=1}^{n} \log(\lambda + \lambda_i(\mtx{W}_k \mtx{W}_k^{\top} \mtx{A} \mtx{W}_k \mtx{W}_k^{\top}))
    \leq \sum_{i=1}^{n} \log(\lambda + \lambda_i(\widehat{\mtx{W}}_k \widehat{\mtx{W}}_k^{\top} \mtx{A} \widehat{\mtx{W}}_k \widehat{\mtx{W}}_k^{\top})) \\
    = & \Trace(\log(\lambda \mtx{I} + \widehat{\mtx{W}}_k \widehat{\mtx{W}}_k^{\top} \mtx{A} \widehat{\mtx{W}}_k \widehat{\mtx{W}}_k^{\top}))
    \leq  \sum_{i=1}^{n} \log(\lambda + \lambda_i(\mtx{I}^{\top} \mtx{A} \mtx{I})) \\
    = & \Trace(\log(\lambda \mtx{I} + \mtx{A})).
  \end{align*}

\end{proof}

\subsubsection{With residual trace estimation}

The trace estimator for the positive semidefinite residual $\mtx{R}(\lambda) = \log(\lambda \mtx{I} + \mtx{A}) - \log(\lambda \mtx{I} + {\mtx{W}}_k {\mtx{T}}_k {\mtx{W}}_k^{\top})$ described in~\Cref{subsec:residual-trace} is standard and satisfies a bound of the form
\begin{equation*}
  \left| \Trace(\mtx{R}(\lambda)) - \frac{1}{n_{\mtx{\Psi}}} \sum_{i=1}^{n_{\mtx{\Psi}}} \vct{\psi}_i^{\top} \mtx{R}(\lambda) \vct{\psi}_i \right| \leq C \sqrt{\frac{\log(1/\delta)}{n_{\mtx{\Psi}}}} \lVert \mtx{R}(\lambda) \rVert _F
  \label{equ:hutchinson-residual-bound}
\end{equation*}
with probability $1 - \delta$, if $n_{\mtx{\Psi}} > c \log(1/\delta)$ for fixed constants $C$ and $c$~\cite{cortinovis-2022-randomized-trace,meyer-2021-hutch-optimal}. Further, since $\mtx{R}(\lambda)$ is positive semidefinite, monotonicity of the Schatten norms yields
\begin{equation*}
\lVert \mtx{R}(\lambda) \rVert_F \le \Trace(\mtx{R}(\lambda)),
\end{equation*}
to which the discussion from~\Cref{subsubsec:log-trace-approx,subsubsec:log-trace-approx-aug} applies. As a consequence, incorporating the residual trace estimate improves---with high probability---the bound for the trace estimator (\Cref{lem:log-trace-standard}) by a factor proportional to $1/\sqrt{n_{\mtx{\Psi}}}$.

\subsection{Summary} In conclusion, we have shown that both the trace and the quadratic form in the PML criterion~\cref{equ:pml-criterion} can be approximated using a single augmented Krylov subspace, without any loss of accuracy compared to treating the two terms separately. Moreover, for the quadratic form approximation~\cref{equ:iqf-approximation}, the augmentation acts as an implicit form of preconditioning, leading to significantly faster convergence of the approximation~\cref{equ:iqf-approximation} when the singular values of the matrix $\mtx{A}$ have sufficient decay. This is the case for the matrices $\mtx{A}$ defined in~\cref{equ:matrix-definition}, which arise from kernel matrices $\mtx{K}$ with pronounced smoothing properties---a feature characteristic to all the kernels used in our application~\cite{chen-2012-estimation-transfer,andersen-2020-smoothing-splines,chen-2025-fast-kernelbased}.

Additionally, we have shown that the trace approximation~\cref{equ:trace-approximation} can be further refined through the residual trace estimation component, which we empirically found to be essential for achieving accurate approximation.

We give an overview of our results from this section in~\Cref{tab:overview}.

\begin{table}[ht]
  \caption{Overview of the bounds proved in this section.}
  \label{tab:overview}
  \centering
  \renewcommand{\arraystretch}{1.2}
  \begin{tabular}{@{}c|c|c|cc@{}}
  \toprule
  & \multirow{2}{*}{$\mathcal{K}(\mtx{A}, \vct{y})$} & \multirow{2}{*}{$\mathcal{K}(\mtx{A}, \mtx{\Omega})$} & \multicolumn{2}{c}{$\mathcal{K}(\mtx{A}, [\vct{y}, \mtx{\Omega}])$} \\
  Approximation of & & & Basic bound & Improved bound \\
  \midrule
  $\vct{y}^{\top} (\lambda \mtx{I} + \mtx{A})^{-1} \vct{y}$ & \Cref{lem:cg-bound} & --- & \Cref{lem:cg-bound-augmented} & \Cref{thm:implicit-preconditioning} \\
  $\Trace(\log(\lambda \mtx{I} + \mtx{A}))$ & --- & \Cref{lem:log-trace-standard} & \Cref{lem:log-trace-augmented} & \Cref{subsec:residual-trace} \\
  \bottomrule
  \end{tabular}
\end{table}

\section{Numerical experiments}
\label{sec:num}

In this section, we test the performance of the Krylov-augmented algorithm described in~\Cref{sec:kryl} on multiple impulse responses of artificial systems. To this extent, we have extended the GPR-FIRE (Gaussian process regression for finite impulse response estimation) MATLAB package introduced in~\cite{chen-2025-fast-kernelbased}. Our implementations are developed in MATLAB 2024b using the System Identification, Control Systems, and Optimization Toolboxes. Experiments are executed in Ubuntu 24.04 LTS running on a system equipped with a 13th Gen Intel Core i5-1335U CPU (4.6 GHz) and 16 GiB of LPDDR5 RAM.

The examples in this section are randomly generated systems as described in~\cite[Section V.A]{chen-2025-fast-kernelbased}. To generate their output, white Gaussian noise is filtered by a second-order transfer function $G(z) = (1 - a z^{-1})^{-2}$. White Gaussian noise, scaled to achieve a specific signal-to-noise ratio (SNR), is added to the output. Unless otherwise stated, we have $a = 0.2$, $\mathrm{SNR} = 10$, $n = 2 \times 10^3$, and $m = 10^4$.

First, we use the direct algorithm (\Cref{subsec:direct}), the indirect algorithm (\Cref{subsec:indirect}), and the Krylov-augmented algorithm (\Cref{alg:krylov-augmented}) to approximate the PML criterion~\cref{equ:pml-criterion} on a $(50 \times 50)$-grid of parameter pairs $(\beta, \lambda)$ taken logarithmically from $\beta \in [10^{-6}, 10^{-2}]$ and $\lambda \in [10^{-1}, 10^{6}]$. Unless otherwise mentioned, we use the TC kernel from~\cite{chen-2012-estimation-transfer} for these examples. The Krylov-augmented algorithm is run on parameters $n_{\mtx{\Omega}} = 1$, $n_{\mtx{\Psi}} = 3$, and $k = 40$. We time the methods and display the results in~\Cref{fig:likelihood-comparison}. The advantage of cheap evaluations with respect to many values of $\lambda$ becomes apparent in the significantly lower runtimes of the direct algorithm and the Krylov-augmented algorithm as opposed to the indirect algorithm. Further, the approximation produced by the Krylov-augmented algorithm is visibly very similar to the one of the direct algorithm, though computed in much less time.

\begin{figure}[!ht]
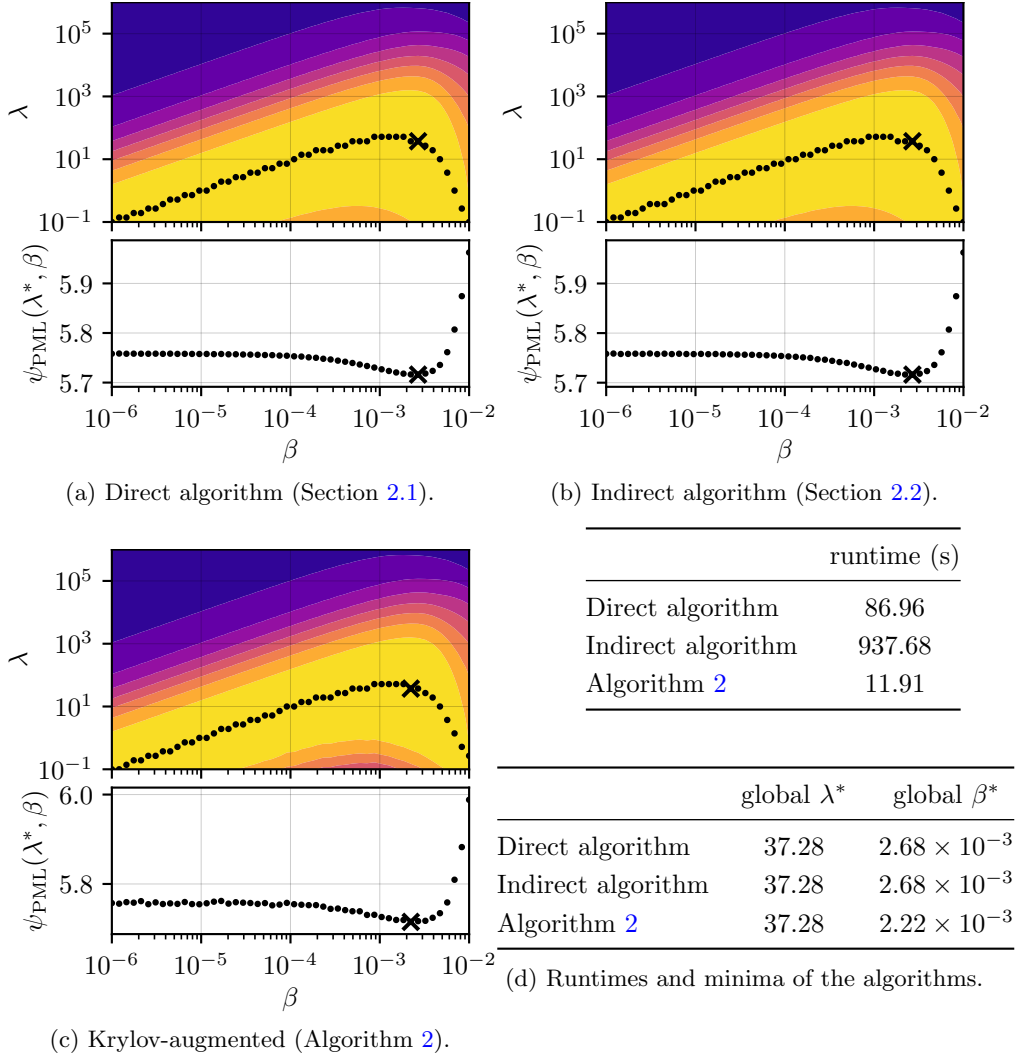

  \centering
    \begin{subfigure}[t]{0.5\textwidth}
      \centering
      \input{plots/likelihood_pml_direct.pgf}
      \caption{Direct algorithm (\Cref{subsec:direct}).}
      \label{fig:likelihood-direct}
    \end{subfigure}%
    \begin{subfigure}[t]{0.5\textwidth}
      \centering
      \input{plots/likelihood_pml_indirect.pgf}
      \caption{Indirect algorithm (\Cref{subsec:indirect}).}
      \label{fig:likelihood-indirect}
    \end{subfigure}

    \begin{subfigure}[t]{0.5\textwidth}
      \centering
      \input{plots/likelihood_pml_krylov.pgf}
      \caption{Krylov-augmented (\Cref{alg:krylov-augmented}).}
      \label{fig:likelihood-krylov}
    \end{subfigure}%
    \begin{subfigure}[b]{0.5\textwidth}
      \hspace*{15pt}
      \vspace*{20pt}
      \centering
      \renewcommand{\arraystretch}{1.2}
      \begin{tabular}{@{}lc@{}}
      \toprule
       & runtime (s)\\
      \midrule
      Direct algorithm & $86.96$ \\
      Indirect algorithm & $937.68$ \\
      \Cref{alg:krylov-augmented} & $11.91$ \\
      \bottomrule
      \end{tabular}
      \begin{tabular}{@{}lcc@{}}
      \toprule
       & global $\lambda^{\ast}$ & global $\beta^{\ast}$ \\
      \midrule
      Direct algorithm & $37.28$ & $2.68 \times 10^{-3}$ \\
      Indirect algorithm  & $37.28$ & $2.68 \times 10^{-3}$ \\
      \Cref{alg:krylov-augmented}  & $37.28$ & $2.22 \times 10^{-3}$ \\
      \bottomrule
      \end{tabular}
      \caption{Runtimes and minima of the algorithms.}
      \label{tab:likelihood-runtime}
    \end{subfigure}
    \caption{Approximation of the PML criterion $\psi_{\mathrm{PML}}(\lambda, \beta)$. For each value of $\beta$, the dots ($\boldsymbol{\cdot}$) identify the locations $\lambda^{\ast}$ of the minima with respect to $\lambda$. The cross ($\boldsymbol{\times}$) marks the global minimum value of $\psi_{\mathrm{PML}}$. In (d), the time each method took to compute the $2'500$ evaluations of $\psi_{\mathrm{PML}}$ is given.}
    \label{fig:likelihood-comparison}
\end{figure}

For~\Cref{fig:comparison-large}, we generate 20 different systems with $m=10^4$ observations as they are described above. We use a Bayesian optimizer limited to 40 evaluations of the PML criterion~\cref{equ:pml-criterion}. The PML criterion is approximated using one of the three proposed methods at a time. The quality of the approximated FIRs is measured with a metric called \emph{fit}, introduced in the System Identification Toolbox~\cite{ljung-1995-system-identification}. The fit is computed based on the root mean squared error (RMSE) of the estimated impulse response from the true impulse response. A fit of 100 is an exact match while a fit of 0 means that  the error is as large as the RMSE of the impulse response from its mean. We visualize the fits in a box-plot in~\Cref{fig:comparison-large-fits}. The runtimes of the different optimization procedures are visualized in~\Cref{fig:comparison-large-times}. Given the high number of observations, the direct algorithm cannot compete with the indirect method and Krylov method due to its inherently higher complexity.

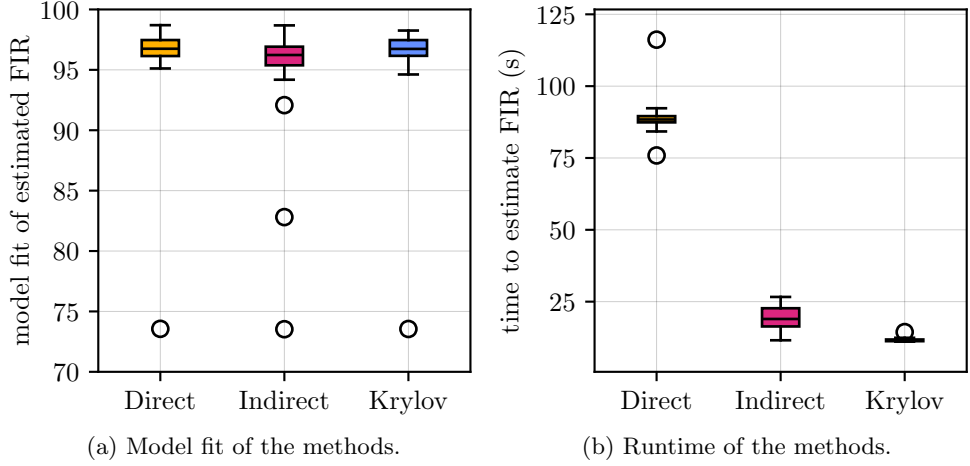
\begin{figure}[!ht]
    \begin{subfigure}[b]{0.49\textwidth}
      \input{plots/comparison_large_fits.pgf}
      \caption{Model fit of the methods.}
      \label{fig:comparison-large-fits}
    \end{subfigure}
    \begin{subfigure}[b]{0.49\textwidth}
      \input{plots/comparison_large_times.pgf}
      \caption{Runtime of the methods.}
      \label{fig:comparison-large-times}
    \end{subfigure}
    \caption{Box-plots of the model fit and runtime of 20 randomly generated systems with $m = 10^4$ observations and estimated FIRs of order $n = 2 \times 10^3$ with Bayesian optimization through PML objective evaluations based on the different methods.}
    \label{fig:comparison-large}
\end{figure}

In~\Cref{fig:comparison-large-snr-fits}, we demonstrate that our method remains robust under low signal-to-noise ratios (SNR) when estimating impulse responses via~\Cref{alg:krylov-augmented} for PML objective evaluations. In~\Cref{fig:comparison-large-kernels-fits} we demonstrate that our method also works for different choices of kernels $\mtx{K}(\beta)$, particularly the TC kernel as above, the DC kernel~\cite{chen-2012-estimation-transfer}, and the SS kernels, which are less commonly used in current practice~\cite{pillonetto-2010-new-kernelbased}.

\begin{figure}[!ht]
    \begin{subfigure}[b]{0.49\textwidth}
      \input{plots/comparison_large_snr_fits.pgf}
      \caption{Model fits for smaller SNR.}
      \label{fig:comparison-large-snr-fits}
    \end{subfigure}
    \begin{subfigure}[b]{0.49\textwidth}
      \input{plots/comparison_large_kernels_fits.pgf}
      \caption{Model fits for different kernels.}
      \label{fig:comparison-large-kernels-fits}
    \end{subfigure}
    \caption{Box-plots of the model fit of 20 randomly generated systems with $m = 10^4$ observations, and estimated FIRs of order $n = 2 \times 10^3$ with Bayesian optimization through objective evaluations with~\Cref{alg:krylov-augmented}.}
    \label{fig:comparison-large-snr}
\end{figure}
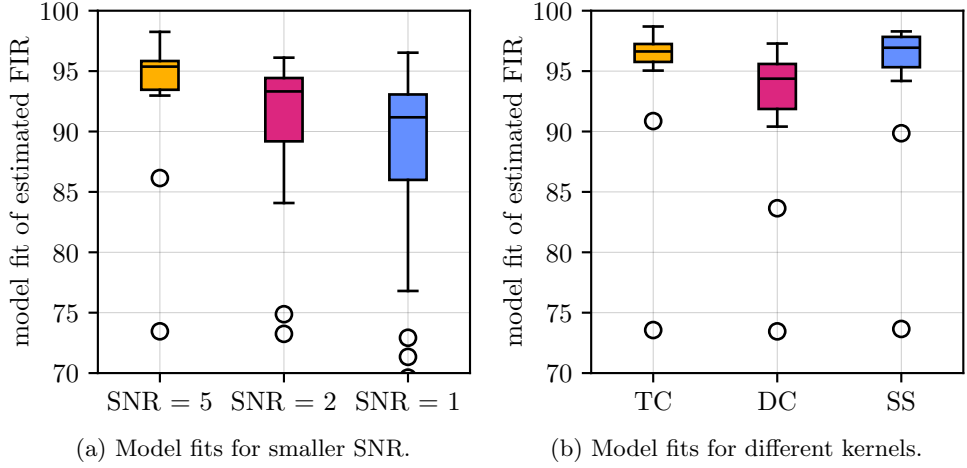

\section{Conclusion and outlook}
\label{sec:con}

We have introduced and analyzed a novel method for kernel-based regularized FIR estimation. This approach approximates both the quadratic form and the log-determinant appearing in the log-likelihood parameter optimization objective at the same time. To achieve this, it builds an augmented block Krylov subspace, which improves the accuracy of both approximations and reduces the number of matrix loads---noticeably lowering computational cost. Moreover, thanks to the shift-invariance of the Krylov subspace, the optimization objective can be evaluated for many values of the regularization parameter at little additional cost, which speeds up the optimization process significantly.

In numerical experiments, we demonstrate that our Krylov-based method achieves comparable results to existing methods in less time.

In future work, we aim at making the Krylov-augmented algorithm adaptive, such that the choice of parameters is no longer delegated to the user and only as many iterations of the Lanczos method are performed as necessary. Further, one could explore how the Krylov-augmented approach can be accelerated by incorporating gradient information into the Bayesian optimization process, similarly to the developments in~\cite{chung-2025-efficient-hyperparameter}. So-called \emph{gradient-enhanced} Bayesian optimization has the potential to land at a minimum of $\psi_{\mathrm{PML}}$ in fewer function evaluations than when the optimizer has no access to gradient information.

\backmatter

\bmhead{Acknowledgements} M.A. is supported by the Novo Nordisk Foundation (no. NNF20OC0061894). T.C. is supported by the National Natural Science Foundation of China (no. 62273287) and the Shenzhen Science and Technology Innovation Council (no. JCYJ20220530143418040).

\bibliography{sn-bibliography}

\end{document}

%% file: plots/matvec_runtimes.pgf
\begingroup%
\makeatletter%
\begin{pgfpicture}%
\pgfpathrectangle{\pgfpointorigin}{\pgfqpoint{5.003309in}{2.909691in}}%
\pgfusepath{use as bounding box, clip}%
\begin{pgfscope}%
\pgfsetbuttcap%
\pgfsetmiterjoin%
\definecolor{currentfill}{rgb}{1.000000,1.000000,1.000000}%
\pgfsetfillcolor{currentfill}%
\pgfsetlinewidth{0.000000pt}%
\definecolor{currentstroke}{rgb}{1.000000,1.000000,1.000000}%
\pgfsetstrokecolor{currentstroke}%
\pgfsetdash{}{0pt}%
\pgfpathmoveto{\pgfqpoint{0.000000in}{0.000000in}}%
\pgfpathlineto{\pgfqpoint{5.003309in}{0.000000in}}%
\pgfpathlineto{\pgfqpoint{5.003309in}{2.909691in}}%
\pgfpathlineto{\pgfqpoint{0.000000in}{2.909691in}}%
\pgfpathlineto{\pgfqpoint{0.000000in}{0.000000in}}%
\pgfpathclose%
\pgfusepath{fill}%
\end{pgfscope}%
\begin{pgfscope}%
\pgfsetbuttcap%
\pgfsetmiterjoin%
\definecolor{currentfill}{rgb}{1.000000,1.000000,1.000000}%
\pgfsetfillcolor{currentfill}%
\pgfsetlinewidth{0.000000pt}%
\definecolor{currentstroke}{rgb}{0.000000,0.000000,0.000000}%
\pgfsetstrokecolor{currentstroke}%
\pgfsetstrokeopacity{0.000000}%
\pgfsetdash{}{0pt}%
\pgfpathmoveto{\pgfqpoint{0.569136in}{0.499691in}}%
\pgfpathlineto{\pgfqpoint{4.831636in}{0.499691in}}%
\pgfpathlineto{\pgfqpoint{4.831636in}{2.809691in}}%
\pgfpathlineto{\pgfqpoint{0.569136in}{2.809691in}}%
\pgfpathlineto{\pgfqpoint{0.569136in}{0.499691in}}%
\pgfpathclose%
\pgfusepath{fill}%
\end{pgfscope}%
\begin{pgfscope}%
\pgfpathrectangle{\pgfqpoint{0.569136in}{0.499691in}}{\pgfqpoint{4.262500in}{2.310000in}}%
\pgfusepath{clip}%
\pgfsetbuttcap%
\pgfsetroundjoin%
\definecolor{currentfill}{rgb}{1.000000,0.690196,0.000000}%
\pgfsetfillcolor{currentfill}%
\pgfsetfillopacity{0.250000}%
\pgfsetlinewidth{0.000000pt}%
\definecolor{currentstroke}{rgb}{0.000000,0.000000,0.000000}%
\pgfsetstrokecolor{currentstroke}%
\pgfsetdash{}{0pt}%
\pgfpathmoveto{\pgfqpoint{0.670624in}{1.067116in}}%
\pgfpathlineto{\pgfqpoint{0.670624in}{0.968290in}}%
\pgfpathlineto{\pgfqpoint{1.121683in}{0.767312in}}%
\pgfpathlineto{\pgfqpoint{1.572741in}{1.038187in}}%
\pgfpathlineto{\pgfqpoint{2.023799in}{1.396431in}}%
\pgfpathlineto{\pgfqpoint{2.474857in}{1.199329in}}%
\pgfpathlineto{\pgfqpoint{2.925915in}{1.349171in}}%
\pgfpathlineto{\pgfqpoint{3.376974in}{1.529733in}}%
\pgfpathlineto{\pgfqpoint{3.828032in}{2.263364in}}%
\pgfpathlineto{\pgfqpoint{4.279090in}{1.740170in}}%
\pgfpathlineto{\pgfqpoint{4.730148in}{2.221701in}}%
\pgfpathlineto{\pgfqpoint{4.730148in}{2.617191in}}%
\pgfpathlineto{\pgfqpoint{4.730148in}{2.617191in}}%
\pgfpathlineto{\pgfqpoint{4.279090in}{2.138620in}}%
\pgfpathlineto{\pgfqpoint{3.828032in}{2.432687in}}%
\pgfpathlineto{\pgfqpoint{3.376974in}{1.726005in}}%
\pgfpathlineto{\pgfqpoint{2.925915in}{1.661674in}}%
\pgfpathlineto{\pgfqpoint{2.474857in}{1.432106in}}%
\pgfpathlineto{\pgfqpoint{2.023799in}{1.923508in}}%
\pgfpathlineto{\pgfqpoint{1.572741in}{1.247263in}}%
\pgfpathlineto{\pgfqpoint{1.121683in}{0.794214in}}%
\pgfpathlineto{\pgfqpoint{0.670624in}{1.067116in}}%
\pgfpathlineto{\pgfqpoint{0.670624in}{1.067116in}}%
\pgfpathclose%
\pgfusepath{fill}%
\end{pgfscope}%
\begin{pgfscope}%
\pgfpathrectangle{\pgfqpoint{0.569136in}{0.499691in}}{\pgfqpoint{4.262500in}{2.310000in}}%
\pgfusepath{clip}%
\pgfsetbuttcap%
\pgfsetroundjoin%
\definecolor{currentfill}{rgb}{0.862745,0.149020,0.498039}%
\pgfsetfillcolor{currentfill}%
\pgfsetfillopacity{0.250000}%
\pgfsetlinewidth{0.000000pt}%
\definecolor{currentstroke}{rgb}{0.000000,0.000000,0.000000}%
\pgfsetstrokecolor{currentstroke}%
\pgfsetdash{}{0pt}%
\pgfpathmoveto{\pgfqpoint{0.670624in}{0.879685in}}%
\pgfpathlineto{\pgfqpoint{0.670624in}{0.842017in}}%
\pgfpathlineto{\pgfqpoint{1.121683in}{0.705783in}}%
\pgfpathlineto{\pgfqpoint{1.572741in}{0.956834in}}%
\pgfpathlineto{\pgfqpoint{2.023799in}{1.231039in}}%
\pgfpathlineto{\pgfqpoint{2.474857in}{1.125342in}}%
\pgfpathlineto{\pgfqpoint{2.925915in}{1.256706in}}%
\pgfpathlineto{\pgfqpoint{3.376974in}{1.358540in}}%
\pgfpathlineto{\pgfqpoint{3.828032in}{1.745937in}}%
\pgfpathlineto{\pgfqpoint{4.279090in}{1.631718in}}%
\pgfpathlineto{\pgfqpoint{4.730148in}{1.983522in}}%
\pgfpathlineto{\pgfqpoint{4.730148in}{2.216531in}}%
\pgfpathlineto{\pgfqpoint{4.730148in}{2.216531in}}%
\pgfpathlineto{\pgfqpoint{4.279090in}{1.794280in}}%
\pgfpathlineto{\pgfqpoint{3.828032in}{1.850175in}}%
\pgfpathlineto{\pgfqpoint{3.376974in}{1.442342in}}%
\pgfpathlineto{\pgfqpoint{2.925915in}{1.337269in}}%
\pgfpathlineto{\pgfqpoint{2.474857in}{1.176227in}}%
\pgfpathlineto{\pgfqpoint{2.023799in}{1.300250in}}%
\pgfpathlineto{\pgfqpoint{1.572741in}{1.002485in}}%
\pgfpathlineto{\pgfqpoint{1.121683in}{0.728003in}}%
\pgfpathlineto{\pgfqpoint{0.670624in}{0.879685in}}%
\pgfpathlineto{\pgfqpoint{0.670624in}{0.879685in}}%
\pgfpathclose%
\pgfusepath{fill}%
\end{pgfscope}%
\begin{pgfscope}%
\pgfpathrectangle{\pgfqpoint{0.569136in}{0.499691in}}{\pgfqpoint{4.262500in}{2.310000in}}%
\pgfusepath{clip}%
\pgfsetbuttcap%
\pgfsetroundjoin%
\definecolor{currentfill}{rgb}{0.392157,0.560784,1.000000}%
\pgfsetfillcolor{currentfill}%
\pgfsetfillopacity{0.250000}%
\pgfsetlinewidth{0.000000pt}%
\definecolor{currentstroke}{rgb}{0.000000,0.000000,0.000000}%
\pgfsetstrokecolor{currentstroke}%
\pgfsetdash{}{0pt}%
\pgfpathmoveto{\pgfqpoint{0.670624in}{0.815779in}}%
\pgfpathlineto{\pgfqpoint{0.670624in}{0.786358in}}%
\pgfpathlineto{\pgfqpoint{1.121683in}{0.692191in}}%
\pgfpathlineto{\pgfqpoint{1.572741in}{0.859376in}}%
\pgfpathlineto{\pgfqpoint{2.023799in}{1.080268in}}%
\pgfpathlineto{\pgfqpoint{2.474857in}{0.970540in}}%
\pgfpathlineto{\pgfqpoint{2.925915in}{1.064419in}}%
\pgfpathlineto{\pgfqpoint{3.376974in}{1.124665in}}%
\pgfpathlineto{\pgfqpoint{3.828032in}{1.487639in}}%
\pgfpathlineto{\pgfqpoint{4.279090in}{1.342413in}}%
\pgfpathlineto{\pgfqpoint{4.730148in}{1.610582in}}%
\pgfpathlineto{\pgfqpoint{4.730148in}{1.871557in}}%
\pgfpathlineto{\pgfqpoint{4.730148in}{1.871557in}}%
\pgfpathlineto{\pgfqpoint{4.279090in}{1.525719in}}%
\pgfpathlineto{\pgfqpoint{3.828032in}{1.632989in}}%
\pgfpathlineto{\pgfqpoint{3.376974in}{1.256358in}}%
\pgfpathlineto{\pgfqpoint{2.925915in}{1.185520in}}%
\pgfpathlineto{\pgfqpoint{2.474857in}{1.065817in}}%
\pgfpathlineto{\pgfqpoint{2.023799in}{1.168614in}}%
\pgfpathlineto{\pgfqpoint{1.572741in}{0.917331in}}%
\pgfpathlineto{\pgfqpoint{1.121683in}{0.704349in}}%
\pgfpathlineto{\pgfqpoint{0.670624in}{0.815779in}}%
\pgfpathlineto{\pgfqpoint{0.670624in}{0.815779in}}%
\pgfpathclose%
\pgfusepath{fill}%
\end{pgfscope}%
\begin{pgfscope}%
\pgfpathrectangle{\pgfqpoint{0.569136in}{0.499691in}}{\pgfqpoint{4.262500in}{2.310000in}}%
\pgfusepath{clip}%
\pgfsetrectcap%
\pgfsetroundjoin%
\pgfsetlinewidth{0.250937pt}%
\definecolor{currentstroke}{rgb}{0.000000,0.000000,0.000000}%
\pgfsetstrokecolor{currentstroke}%
\pgfsetstrokeopacity{0.200000}%
\pgfsetdash{}{0pt}%
\pgfpathmoveto{\pgfqpoint{1.121232in}{0.499691in}}%
\pgfpathlineto{\pgfqpoint{1.121232in}{2.809691in}}%
\pgfusepath{stroke}%
\end{pgfscope}%
\begin{pgfscope}%
\pgfsetbuttcap%
\pgfsetroundjoin%
\definecolor{currentfill}{rgb}{0.000000,0.000000,0.000000}%
\pgfsetfillcolor{currentfill}%
\pgfsetlinewidth{0.803000pt}%
\definecolor{currentstroke}{rgb}{0.000000,0.000000,0.000000}%
\pgfsetstrokecolor{currentstroke}%
\pgfsetdash{}{0pt}%
\pgfsys@defobject{currentmarker}{\pgfqpoint{0.000000in}{-0.048611in}}{\pgfqpoint{0.000000in}{0.000000in}}{%
\pgfpathmoveto{\pgfqpoint{0.000000in}{0.000000in}}%
\pgfpathlineto{\pgfqpoint{0.000000in}{-0.048611in}}%
\pgfusepath{stroke,fill}%
}%
\begin{pgfscope}%
\pgfsys@transformshift{1.121232in}{0.499691in}%
\pgfsys@useobject{currentmarker}{}%
\end{pgfscope}%
\end{pgfscope}%
\begin{pgfscope}%
\definecolor{textcolor}{rgb}{0.000000,0.000000,0.000000}%
\pgfsetstrokecolor{textcolor}%
\pgfsetfillcolor{textcolor}%
\pgftext[x=1.121232in,y=0.402469in,,top]{\color{textcolor}{\rmfamily\fontsize{10.000000}{12.000000}\selectfont\catcode`\^=\active\def^{\ifmmode\sp\else\^{}\fi}\catcode`\%=\active\def
\end{pgfscope}%
\begin{pgfscope}%
\pgfpathrectangle{\pgfqpoint{0.569136in}{0.499691in}}{\pgfqpoint{4.262500in}{2.310000in}}%
\pgfusepath{clip}%
\pgfsetrectcap%
\pgfsetroundjoin%
\pgfsetlinewidth{0.250937pt}%
\definecolor{currentstroke}{rgb}{0.000000,0.000000,0.000000}%
\pgfsetstrokecolor{currentstroke}%
\pgfsetstrokeopacity{0.200000}%
\pgfsetdash{}{0pt}%
\pgfpathmoveto{\pgfqpoint{2.023348in}{0.499691in}}%
\pgfpathlineto{\pgfqpoint{2.023348in}{2.809691in}}%
\pgfusepath{stroke}%
\end{pgfscope}%
\begin{pgfscope}%
\pgfsetbuttcap%
\pgfsetroundjoin%
\definecolor{currentfill}{rgb}{0.000000,0.000000,0.000000}%
\pgfsetfillcolor{currentfill}%
\pgfsetlinewidth{0.803000pt}%
\definecolor{currentstroke}{rgb}{0.000000,0.000000,0.000000}%
\pgfsetstrokecolor{currentstroke}%
\pgfsetdash{}{0pt}%
\pgfsys@defobject{currentmarker}{\pgfqpoint{0.000000in}{-0.048611in}}{\pgfqpoint{0.000000in}{0.000000in}}{%
\pgfpathmoveto{\pgfqpoint{0.000000in}{0.000000in}}%
\pgfpathlineto{\pgfqpoint{0.000000in}{-0.048611in}}%
\pgfusepath{stroke,fill}%
}%
\begin{pgfscope}%
\pgfsys@transformshift{2.023348in}{0.499691in}%
\pgfsys@useobject{currentmarker}{}%
\end{pgfscope}%
\end{pgfscope}%
\begin{pgfscope}%
\definecolor{textcolor}{rgb}{0.000000,0.000000,0.000000}%
\pgfsetstrokecolor{textcolor}%
\pgfsetfillcolor{textcolor}%
\pgftext[x=2.023348in,y=0.402469in,,top]{\color{textcolor}{\rmfamily\fontsize{10.000000}{12.000000}\selectfont\catcode`\^=\active\def^{\ifmmode\sp\else\^{}\fi}\catcode`\%=\active\def
\end{pgfscope}%
\begin{pgfscope}%
\pgfpathrectangle{\pgfqpoint{0.569136in}{0.499691in}}{\pgfqpoint{4.262500in}{2.310000in}}%
\pgfusepath{clip}%
\pgfsetrectcap%
\pgfsetroundjoin%
\pgfsetlinewidth{0.250937pt}%
\definecolor{currentstroke}{rgb}{0.000000,0.000000,0.000000}%
\pgfsetstrokecolor{currentstroke}%
\pgfsetstrokeopacity{0.200000}%
\pgfsetdash{}{0pt}%
\pgfpathmoveto{\pgfqpoint{2.925464in}{0.499691in}}%
\pgfpathlineto{\pgfqpoint{2.925464in}{2.809691in}}%
\pgfusepath{stroke}%
\end{pgfscope}%
\begin{pgfscope}%
\pgfsetbuttcap%
\pgfsetroundjoin%
\definecolor{currentfill}{rgb}{0.000000,0.000000,0.000000}%
\pgfsetfillcolor{currentfill}%
\pgfsetlinewidth{0.803000pt}%
\definecolor{currentstroke}{rgb}{0.000000,0.000000,0.000000}%
\pgfsetstrokecolor{currentstroke}%
\pgfsetdash{}{0pt}%
\pgfsys@defobject{currentmarker}{\pgfqpoint{0.000000in}{-0.048611in}}{\pgfqpoint{0.000000in}{0.000000in}}{%
\pgfpathmoveto{\pgfqpoint{0.000000in}{0.000000in}}%
\pgfpathlineto{\pgfqpoint{0.000000in}{-0.048611in}}%
\pgfusepath{stroke,fill}%
}%
\begin{pgfscope}%
\pgfsys@transformshift{2.925464in}{0.499691in}%
\pgfsys@useobject{currentmarker}{}%
\end{pgfscope}%
\end{pgfscope}%
\begin{pgfscope}%
\definecolor{textcolor}{rgb}{0.000000,0.000000,0.000000}%
\pgfsetstrokecolor{textcolor}%
\pgfsetfillcolor{textcolor}%
\pgftext[x=2.925464in,y=0.402469in,,top]{\color{textcolor}{\rmfamily\fontsize{10.000000}{12.000000}\selectfont\catcode`\^=\active\def^{\ifmmode\sp\else\^{}\fi}\catcode`\%=\active\def
\end{pgfscope}%
\begin{pgfscope}%
\pgfpathrectangle{\pgfqpoint{0.569136in}{0.499691in}}{\pgfqpoint{4.262500in}{2.310000in}}%
\pgfusepath{clip}%
\pgfsetrectcap%
\pgfsetroundjoin%
\pgfsetlinewidth{0.250937pt}%
\definecolor{currentstroke}{rgb}{0.000000,0.000000,0.000000}%
\pgfsetstrokecolor{currentstroke}%
\pgfsetstrokeopacity{0.200000}%
\pgfsetdash{}{0pt}%
\pgfpathmoveto{\pgfqpoint{3.827581in}{0.499691in}}%
\pgfpathlineto{\pgfqpoint{3.827581in}{2.809691in}}%
\pgfusepath{stroke}%
\end{pgfscope}%
\begin{pgfscope}%
\pgfsetbuttcap%
\pgfsetroundjoin%
\definecolor{currentfill}{rgb}{0.000000,0.000000,0.000000}%
\pgfsetfillcolor{currentfill}%
\pgfsetlinewidth{0.803000pt}%
\definecolor{currentstroke}{rgb}{0.000000,0.000000,0.000000}%
\pgfsetstrokecolor{currentstroke}%
\pgfsetdash{}{0pt}%
\pgfsys@defobject{currentmarker}{\pgfqpoint{0.000000in}{-0.048611in}}{\pgfqpoint{0.000000in}{0.000000in}}{%
\pgfpathmoveto{\pgfqpoint{0.000000in}{0.000000in}}%
\pgfpathlineto{\pgfqpoint{0.000000in}{-0.048611in}}%
\pgfusepath{stroke,fill}%
}%
\begin{pgfscope}%
\pgfsys@transformshift{3.827581in}{0.499691in}%
\pgfsys@useobject{currentmarker}{}%
\end{pgfscope}%
\end{pgfscope}%
\begin{pgfscope}%
\definecolor{textcolor}{rgb}{0.000000,0.000000,0.000000}%
\pgfsetstrokecolor{textcolor}%
\pgfsetfillcolor{textcolor}%
\pgftext[x=3.827581in,y=0.402469in,,top]{\color{textcolor}{\rmfamily\fontsize{10.000000}{12.000000}\selectfont\catcode`\^=\active\def^{\ifmmode\sp\else\^{}\fi}\catcode`\%=\active\def
\end{pgfscope}%
\begin{pgfscope}%
\pgfpathrectangle{\pgfqpoint{0.569136in}{0.499691in}}{\pgfqpoint{4.262500in}{2.310000in}}%
\pgfusepath{clip}%
\pgfsetrectcap%
\pgfsetroundjoin%
\pgfsetlinewidth{0.250937pt}%
\definecolor{currentstroke}{rgb}{0.000000,0.000000,0.000000}%
\pgfsetstrokecolor{currentstroke}%
\pgfsetstrokeopacity{0.200000}%
\pgfsetdash{}{0pt}%
\pgfpathmoveto{\pgfqpoint{4.729697in}{0.499691in}}%
\pgfpathlineto{\pgfqpoint{4.729697in}{2.809691in}}%
\pgfusepath{stroke}%
\end{pgfscope}%
\begin{pgfscope}%
\pgfsetbuttcap%
\pgfsetroundjoin%
\definecolor{currentfill}{rgb}{0.000000,0.000000,0.000000}%
\pgfsetfillcolor{currentfill}%
\pgfsetlinewidth{0.803000pt}%
\definecolor{currentstroke}{rgb}{0.000000,0.000000,0.000000}%
\pgfsetstrokecolor{currentstroke}%
\pgfsetdash{}{0pt}%
\pgfsys@defobject{currentmarker}{\pgfqpoint{0.000000in}{-0.048611in}}{\pgfqpoint{0.000000in}{0.000000in}}{%
\pgfpathmoveto{\pgfqpoint{0.000000in}{0.000000in}}%
\pgfpathlineto{\pgfqpoint{0.000000in}{-0.048611in}}%
\pgfusepath{stroke,fill}%
}%
\begin{pgfscope}%
\pgfsys@transformshift{4.729697in}{0.499691in}%
\pgfsys@useobject{currentmarker}{}%
\end{pgfscope}%
\end{pgfscope}%
\begin{pgfscope}%
\definecolor{textcolor}{rgb}{0.000000,0.000000,0.000000}%
\pgfsetstrokecolor{textcolor}%
\pgfsetfillcolor{textcolor}%
\pgftext[x=4.729697in,y=0.402469in,,top]{\color{textcolor}{\rmfamily\fontsize{10.000000}{12.000000}\selectfont\catcode`\^=\active\def^{\ifmmode\sp\else\^{}\fi}\catcode`\%=\active\def
\end{pgfscope}%
\begin{pgfscope}%
\definecolor{textcolor}{rgb}{0.000000,0.000000,0.000000}%
\pgfsetstrokecolor{textcolor}%
\pgfsetfillcolor{textcolor}%
\pgftext[x=2.700386in,y=0.223457in,,top]{\color{textcolor}{\rmfamily\fontsize{10.000000}{12.000000}\selectfont\catcode`\^=\active\def^{\ifmmode\sp\else\^{}\fi}\catcode`\%=\active\def
\end{pgfscope}%
\begin{pgfscope}%
\pgfpathrectangle{\pgfqpoint{0.569136in}{0.499691in}}{\pgfqpoint{4.262500in}{2.310000in}}%
\pgfusepath{clip}%
\pgfsetrectcap%
\pgfsetroundjoin%
\pgfsetlinewidth{0.250937pt}%
\definecolor{currentstroke}{rgb}{0.000000,0.000000,0.000000}%
\pgfsetstrokecolor{currentstroke}%
\pgfsetstrokeopacity{0.200000}%
\pgfsetdash{}{0pt}%
\pgfpathmoveto{\pgfqpoint{0.569136in}{0.607438in}}%
\pgfpathlineto{\pgfqpoint{4.831636in}{0.607438in}}%
\pgfusepath{stroke}%
\end{pgfscope}%
\begin{pgfscope}%
\pgfsetbuttcap%
\pgfsetroundjoin%
\definecolor{currentfill}{rgb}{0.000000,0.000000,0.000000}%
\pgfsetfillcolor{currentfill}%
\pgfsetlinewidth{0.803000pt}%
\definecolor{currentstroke}{rgb}{0.000000,0.000000,0.000000}%
\pgfsetstrokecolor{currentstroke}%
\pgfsetdash{}{0pt}%
\pgfsys@defobject{currentmarker}{\pgfqpoint{-0.048611in}{0.000000in}}{\pgfqpoint{-0.000000in}{0.000000in}}{%
\pgfpathmoveto{\pgfqpoint{-0.000000in}{0.000000in}}%
\pgfpathlineto{\pgfqpoint{-0.048611in}{0.000000in}}%
\pgfusepath{stroke,fill}%
}%
\begin{pgfscope}%
\pgfsys@transformshift{0.569136in}{0.607438in}%
\pgfsys@useobject{currentmarker}{}%
\end{pgfscope}%
\end{pgfscope}%
\begin{pgfscope}%
\definecolor{textcolor}{rgb}{0.000000,0.000000,0.000000}%
\pgfsetstrokecolor{textcolor}%
\pgfsetfillcolor{textcolor}%
\pgftext[x=0.294444in, y=0.559213in, left, base]{\color{textcolor}{\rmfamily\fontsize{10.000000}{12.000000}\selectfont\catcode`\^=\active\def^{\ifmmode\sp\else\^{}\fi}\catcode`\%=\active\def
\end{pgfscope}%
\begin{pgfscope}%
\pgfpathrectangle{\pgfqpoint{0.569136in}{0.499691in}}{\pgfqpoint{4.262500in}{2.310000in}}%
\pgfusepath{clip}%
\pgfsetrectcap%
\pgfsetroundjoin%
\pgfsetlinewidth{0.250937pt}%
\definecolor{currentstroke}{rgb}{0.000000,0.000000,0.000000}%
\pgfsetstrokecolor{currentstroke}%
\pgfsetstrokeopacity{0.200000}%
\pgfsetdash{}{0pt}%
\pgfpathmoveto{\pgfqpoint{0.569136in}{0.955202in}}%
\pgfpathlineto{\pgfqpoint{4.831636in}{0.955202in}}%
\pgfusepath{stroke}%
\end{pgfscope}%
\begin{pgfscope}%
\pgfsetbuttcap%
\pgfsetroundjoin%
\definecolor{currentfill}{rgb}{0.000000,0.000000,0.000000}%
\pgfsetfillcolor{currentfill}%
\pgfsetlinewidth{0.803000pt}%
\definecolor{currentstroke}{rgb}{0.000000,0.000000,0.000000}%
\pgfsetstrokecolor{currentstroke}%
\pgfsetdash{}{0pt}%
\pgfsys@defobject{currentmarker}{\pgfqpoint{-0.048611in}{0.000000in}}{\pgfqpoint{-0.000000in}{0.000000in}}{%
\pgfpathmoveto{\pgfqpoint{-0.000000in}{0.000000in}}%
\pgfpathlineto{\pgfqpoint{-0.048611in}{0.000000in}}%
\pgfusepath{stroke,fill}%
}%
\begin{pgfscope}%
\pgfsys@transformshift{0.569136in}{0.955202in}%
\pgfsys@useobject{currentmarker}{}%
\end{pgfscope}%
\end{pgfscope}%
\begin{pgfscope}%
\definecolor{textcolor}{rgb}{0.000000,0.000000,0.000000}%
\pgfsetstrokecolor{textcolor}%
\pgfsetfillcolor{textcolor}%
\pgftext[x=0.294444in, y=0.906977in, left, base]{\color{textcolor}{\rmfamily\fontsize{10.000000}{12.000000}\selectfont\catcode`\^=\active\def^{\ifmmode\sp\else\^{}\fi}\catcode`\%=\active\def
\end{pgfscope}%
\begin{pgfscope}%
\pgfpathrectangle{\pgfqpoint{0.569136in}{0.499691in}}{\pgfqpoint{4.262500in}{2.310000in}}%
\pgfusepath{clip}%
\pgfsetrectcap%
\pgfsetroundjoin%
\pgfsetlinewidth{0.250937pt}%
\definecolor{currentstroke}{rgb}{0.000000,0.000000,0.000000}%
\pgfsetstrokecolor{currentstroke}%
\pgfsetstrokeopacity{0.200000}%
\pgfsetdash{}{0pt}%
\pgfpathmoveto{\pgfqpoint{0.569136in}{1.302966in}}%
\pgfpathlineto{\pgfqpoint{4.831636in}{1.302966in}}%
\pgfusepath{stroke}%
\end{pgfscope}%
\begin{pgfscope}%
\pgfsetbuttcap%
\pgfsetroundjoin%
\definecolor{currentfill}{rgb}{0.000000,0.000000,0.000000}%
\pgfsetfillcolor{currentfill}%
\pgfsetlinewidth{0.803000pt}%
\definecolor{currentstroke}{rgb}{0.000000,0.000000,0.000000}%
\pgfsetstrokecolor{currentstroke}%
\pgfsetdash{}{0pt}%
\pgfsys@defobject{currentmarker}{\pgfqpoint{-0.048611in}{0.000000in}}{\pgfqpoint{-0.000000in}{0.000000in}}{%
\pgfpathmoveto{\pgfqpoint{-0.000000in}{0.000000in}}%
\pgfpathlineto{\pgfqpoint{-0.048611in}{0.000000in}}%
\pgfusepath{stroke,fill}%
}%
\begin{pgfscope}%
\pgfsys@transformshift{0.569136in}{1.302966in}%
\pgfsys@useobject{currentmarker}{}%
\end{pgfscope}%
\end{pgfscope}%
\begin{pgfscope}%
\definecolor{textcolor}{rgb}{0.000000,0.000000,0.000000}%
\pgfsetstrokecolor{textcolor}%
\pgfsetfillcolor{textcolor}%
\pgftext[x=0.294444in, y=1.254741in, left, base]{\color{textcolor}{\rmfamily\fontsize{10.000000}{12.000000}\selectfont\catcode`\^=\active\def^{\ifmmode\sp\else\^{}\fi}\catcode`\%=\active\def
\end{pgfscope}%
\begin{pgfscope}%
\pgfpathrectangle{\pgfqpoint{0.569136in}{0.499691in}}{\pgfqpoint{4.262500in}{2.310000in}}%
\pgfusepath{clip}%
\pgfsetrectcap%
\pgfsetroundjoin%
\pgfsetlinewidth{0.250937pt}%
\definecolor{currentstroke}{rgb}{0.000000,0.000000,0.000000}%
\pgfsetstrokecolor{currentstroke}%
\pgfsetstrokeopacity{0.200000}%
\pgfsetdash{}{0pt}%
\pgfpathmoveto{\pgfqpoint{0.569136in}{1.650731in}}%
\pgfpathlineto{\pgfqpoint{4.831636in}{1.650731in}}%
\pgfusepath{stroke}%
\end{pgfscope}%
\begin{pgfscope}%
\pgfsetbuttcap%
\pgfsetroundjoin%
\definecolor{currentfill}{rgb}{0.000000,0.000000,0.000000}%
\pgfsetfillcolor{currentfill}%
\pgfsetlinewidth{0.803000pt}%
\definecolor{currentstroke}{rgb}{0.000000,0.000000,0.000000}%
\pgfsetstrokecolor{currentstroke}%
\pgfsetdash{}{0pt}%
\pgfsys@defobject{currentmarker}{\pgfqpoint{-0.048611in}{0.000000in}}{\pgfqpoint{-0.000000in}{0.000000in}}{%
\pgfpathmoveto{\pgfqpoint{-0.000000in}{0.000000in}}%
\pgfpathlineto{\pgfqpoint{-0.048611in}{0.000000in}}%
\pgfusepath{stroke,fill}%
}%
\begin{pgfscope}%
\pgfsys@transformshift{0.569136in}{1.650731in}%
\pgfsys@useobject{currentmarker}{}%
\end{pgfscope}%
\end{pgfscope}%
\begin{pgfscope}%
\definecolor{textcolor}{rgb}{0.000000,0.000000,0.000000}%
\pgfsetstrokecolor{textcolor}%
\pgfsetfillcolor{textcolor}%
\pgftext[x=0.294444in, y=1.602505in, left, base]{\color{textcolor}{\rmfamily\fontsize{10.000000}{12.000000}\selectfont\catcode`\^=\active\def^{\ifmmode\sp\else\^{}\fi}\catcode`\%=\active\def
\end{pgfscope}%
\begin{pgfscope}%
\pgfpathrectangle{\pgfqpoint{0.569136in}{0.499691in}}{\pgfqpoint{4.262500in}{2.310000in}}%
\pgfusepath{clip}%
\pgfsetrectcap%
\pgfsetroundjoin%
\pgfsetlinewidth{0.250937pt}%
\definecolor{currentstroke}{rgb}{0.000000,0.000000,0.000000}%
\pgfsetstrokecolor{currentstroke}%
\pgfsetstrokeopacity{0.200000}%
\pgfsetdash{}{0pt}%
\pgfpathmoveto{\pgfqpoint{0.569136in}{1.998495in}}%
\pgfpathlineto{\pgfqpoint{4.831636in}{1.998495in}}%
\pgfusepath{stroke}%
\end{pgfscope}%
\begin{pgfscope}%
\pgfsetbuttcap%
\pgfsetroundjoin%
\definecolor{currentfill}{rgb}{0.000000,0.000000,0.000000}%
\pgfsetfillcolor{currentfill}%
\pgfsetlinewidth{0.803000pt}%
\definecolor{currentstroke}{rgb}{0.000000,0.000000,0.000000}%
\pgfsetstrokecolor{currentstroke}%
\pgfsetdash{}{0pt}%
\pgfsys@defobject{currentmarker}{\pgfqpoint{-0.048611in}{0.000000in}}{\pgfqpoint{-0.000000in}{0.000000in}}{%
\pgfpathmoveto{\pgfqpoint{-0.000000in}{0.000000in}}%
\pgfpathlineto{\pgfqpoint{-0.048611in}{0.000000in}}%
\pgfusepath{stroke,fill}%
}%
\begin{pgfscope}%
\pgfsys@transformshift{0.569136in}{1.998495in}%
\pgfsys@useobject{currentmarker}{}%
\end{pgfscope}%
\end{pgfscope}%
\begin{pgfscope}%
\definecolor{textcolor}{rgb}{0.000000,0.000000,0.000000}%
\pgfsetstrokecolor{textcolor}%
\pgfsetfillcolor{textcolor}%
\pgftext[x=0.294444in, y=1.950270in, left, base]{\color{textcolor}{\rmfamily\fontsize{10.000000}{12.000000}\selectfont\catcode`\^=\active\def^{\ifmmode\sp\else\^{}\fi}\catcode`\%=\active\def
\end{pgfscope}%
\begin{pgfscope}%
\pgfpathrectangle{\pgfqpoint{0.569136in}{0.499691in}}{\pgfqpoint{4.262500in}{2.310000in}}%
\pgfusepath{clip}%
\pgfsetrectcap%
\pgfsetroundjoin%
\pgfsetlinewidth{0.250937pt}%
\definecolor{currentstroke}{rgb}{0.000000,0.000000,0.000000}%
\pgfsetstrokecolor{currentstroke}%
\pgfsetstrokeopacity{0.200000}%
\pgfsetdash{}{0pt}%
\pgfpathmoveto{\pgfqpoint{0.569136in}{2.346259in}}%
\pgfpathlineto{\pgfqpoint{4.831636in}{2.346259in}}%
\pgfusepath{stroke}%
\end{pgfscope}%
\begin{pgfscope}%
\pgfsetbuttcap%
\pgfsetroundjoin%
\definecolor{currentfill}{rgb}{0.000000,0.000000,0.000000}%
\pgfsetfillcolor{currentfill}%
\pgfsetlinewidth{0.803000pt}%
\definecolor{currentstroke}{rgb}{0.000000,0.000000,0.000000}%
\pgfsetstrokecolor{currentstroke}%
\pgfsetdash{}{0pt}%
\pgfsys@defobject{currentmarker}{\pgfqpoint{-0.048611in}{0.000000in}}{\pgfqpoint{-0.000000in}{0.000000in}}{%
\pgfpathmoveto{\pgfqpoint{-0.000000in}{0.000000in}}%
\pgfpathlineto{\pgfqpoint{-0.048611in}{0.000000in}}%
\pgfusepath{stroke,fill}%
}%
\begin{pgfscope}%
\pgfsys@transformshift{0.569136in}{2.346259in}%
\pgfsys@useobject{currentmarker}{}%
\end{pgfscope}%
\end{pgfscope}%
\begin{pgfscope}%
\definecolor{textcolor}{rgb}{0.000000,0.000000,0.000000}%
\pgfsetstrokecolor{textcolor}%
\pgfsetfillcolor{textcolor}%
\pgftext[x=0.294444in, y=2.298034in, left, base]{\color{textcolor}{\rmfamily\fontsize{10.000000}{12.000000}\selectfont\catcode`\^=\active\def^{\ifmmode\sp\else\^{}\fi}\catcode`\%=\active\def
\end{pgfscope}%
\begin{pgfscope}%
\pgfpathrectangle{\pgfqpoint{0.569136in}{0.499691in}}{\pgfqpoint{4.262500in}{2.310000in}}%
\pgfusepath{clip}%
\pgfsetrectcap%
\pgfsetroundjoin%
\pgfsetlinewidth{0.250937pt}%
\definecolor{currentstroke}{rgb}{0.000000,0.000000,0.000000}%
\pgfsetstrokecolor{currentstroke}%
\pgfsetstrokeopacity{0.200000}%
\pgfsetdash{}{0pt}%
\pgfpathmoveto{\pgfqpoint{0.569136in}{2.694024in}}%
\pgfpathlineto{\pgfqpoint{4.831636in}{2.694024in}}%
\pgfusepath{stroke}%
\end{pgfscope}%
\begin{pgfscope}%
\pgfsetbuttcap%
\pgfsetroundjoin%
\definecolor{currentfill}{rgb}{0.000000,0.000000,0.000000}%
\pgfsetfillcolor{currentfill}%
\pgfsetlinewidth{0.803000pt}%
\definecolor{currentstroke}{rgb}{0.000000,0.000000,0.000000}%
\pgfsetstrokecolor{currentstroke}%
\pgfsetdash{}{0pt}%
\pgfsys@defobject{currentmarker}{\pgfqpoint{-0.048611in}{0.000000in}}{\pgfqpoint{-0.000000in}{0.000000in}}{%
\pgfpathmoveto{\pgfqpoint{-0.000000in}{0.000000in}}%
\pgfpathlineto{\pgfqpoint{-0.048611in}{0.000000in}}%
\pgfusepath{stroke,fill}%
}%
\begin{pgfscope}%
\pgfsys@transformshift{0.569136in}{2.694024in}%
\pgfsys@useobject{currentmarker}{}%
\end{pgfscope}%
\end{pgfscope}%
\begin{pgfscope}%
\definecolor{textcolor}{rgb}{0.000000,0.000000,0.000000}%
\pgfsetstrokecolor{textcolor}%
\pgfsetfillcolor{textcolor}%
\pgftext[x=0.294444in, y=2.645798in, left, base]{\color{textcolor}{\rmfamily\fontsize{10.000000}{12.000000}\selectfont\catcode`\^=\active\def^{\ifmmode\sp\else\^{}\fi}\catcode`\%=\active\def
\end{pgfscope}%
\begin{pgfscope}%
\definecolor{textcolor}{rgb}{0.000000,0.000000,0.000000}%
\pgfsetstrokecolor{textcolor}%
\pgfsetfillcolor{textcolor}%
\pgftext[x=0.238889in,y=1.654691in,,bottom,rotate=90.000000]{\color{textcolor}{\rmfamily\fontsize{10.000000}{12.000000}\selectfont\catcode`\^=\active\def^{\ifmmode\sp\else\^{}\fi}\catcode`\%=\active\def
\end{pgfscope}%
\begin{pgfscope}%
\pgfpathrectangle{\pgfqpoint{0.569136in}{0.499691in}}{\pgfqpoint{4.262500in}{2.310000in}}%
\pgfusepath{clip}%
\pgfsetrectcap%
\pgfsetroundjoin%
\pgfsetlinewidth{1.505625pt}%
\definecolor{currentstroke}{rgb}{1.000000,0.690196,0.000000}%
\pgfsetstrokecolor{currentstroke}%
\pgfsetdash{}{0pt}%
\pgfpathmoveto{\pgfqpoint{0.670624in}{1.017703in}}%
\pgfpathlineto{\pgfqpoint{1.121683in}{0.780763in}}%
\pgfpathlineto{\pgfqpoint{1.572741in}{1.142725in}}%
\pgfpathlineto{\pgfqpoint{2.023799in}{1.659969in}}%
\pgfpathlineto{\pgfqpoint{2.474857in}{1.315717in}}%
\pgfpathlineto{\pgfqpoint{2.925915in}{1.505422in}}%
\pgfpathlineto{\pgfqpoint{3.376974in}{1.627869in}}%
\pgfpathlineto{\pgfqpoint{3.828032in}{2.348025in}}%
\pgfpathlineto{\pgfqpoint{4.279090in}{1.939395in}}%
\pgfpathlineto{\pgfqpoint{4.730148in}{2.419446in}}%
\pgfusepath{stroke}%
\end{pgfscope}%
\begin{pgfscope}%
\pgfpathrectangle{\pgfqpoint{0.569136in}{0.499691in}}{\pgfqpoint{4.262500in}{2.310000in}}%
\pgfusepath{clip}%
\pgfsetbuttcap%
\pgfsetmiterjoin%
\definecolor{currentfill}{rgb}{1.000000,0.690196,0.000000}%
\pgfsetfillcolor{currentfill}%
\pgfsetlinewidth{1.003750pt}%
\definecolor{currentstroke}{rgb}{1.000000,0.690196,0.000000}%
\pgfsetstrokecolor{currentstroke}%
\pgfsetdash{}{0pt}%
\pgfsys@defobject{currentmarker}{\pgfqpoint{-0.035355in}{-0.058926in}}{\pgfqpoint{0.035355in}{0.058926in}}{%
\pgfpathmoveto{\pgfqpoint{-0.000000in}{-0.058926in}}%
\pgfpathlineto{\pgfqpoint{0.035355in}{0.000000in}}%
\pgfpathlineto{\pgfqpoint{0.000000in}{0.058926in}}%
\pgfpathlineto{\pgfqpoint{-0.035355in}{0.000000in}}%
\pgfpathlineto{\pgfqpoint{-0.000000in}{-0.058926in}}%
\pgfpathclose%
\pgfusepath{stroke,fill}%
}%
\begin{pgfscope}%
\pgfsys@transformshift{0.670624in}{1.017703in}%
\pgfsys@useobject{currentmarker}{}%
\end{pgfscope}%
\begin{pgfscope}%
\pgfsys@transformshift{1.121683in}{0.780763in}%
\pgfsys@useobject{currentmarker}{}%
\end{pgfscope}%
\begin{pgfscope}%
\pgfsys@transformshift{1.572741in}{1.142725in}%
\pgfsys@useobject{currentmarker}{}%
\end{pgfscope}%
\begin{pgfscope}%
\pgfsys@transformshift{2.023799in}{1.659969in}%
\pgfsys@useobject{currentmarker}{}%
\end{pgfscope}%
\begin{pgfscope}%
\pgfsys@transformshift{2.474857in}{1.315717in}%
\pgfsys@useobject{currentmarker}{}%
\end{pgfscope}%
\begin{pgfscope}%
\pgfsys@transformshift{2.925915in}{1.505422in}%
\pgfsys@useobject{currentmarker}{}%
\end{pgfscope}%
\begin{pgfscope}%
\pgfsys@transformshift{3.376974in}{1.627869in}%
\pgfsys@useobject{currentmarker}{}%
\end{pgfscope}%
\begin{pgfscope}%
\pgfsys@transformshift{3.828032in}{2.348025in}%
\pgfsys@useobject{currentmarker}{}%
\end{pgfscope}%
\begin{pgfscope}%
\pgfsys@transformshift{4.279090in}{1.939395in}%
\pgfsys@useobject{currentmarker}{}%
\end{pgfscope}%
\begin{pgfscope}%
\pgfsys@transformshift{4.730148in}{2.419446in}%
\pgfsys@useobject{currentmarker}{}%
\end{pgfscope}%
\end{pgfscope}%
\begin{pgfscope}%
\pgfpathrectangle{\pgfqpoint{0.569136in}{0.499691in}}{\pgfqpoint{4.262500in}{2.310000in}}%
\pgfusepath{clip}%
\pgfsetrectcap%
\pgfsetroundjoin%
\pgfsetlinewidth{1.505625pt}%
\definecolor{currentstroke}{rgb}{0.862745,0.149020,0.498039}%
\pgfsetstrokecolor{currentstroke}%
\pgfsetdash{}{0pt}%
\pgfpathmoveto{\pgfqpoint{0.670624in}{0.860851in}}%
\pgfpathlineto{\pgfqpoint{1.121683in}{0.716893in}}%
\pgfpathlineto{\pgfqpoint{1.572741in}{0.979660in}}%
\pgfpathlineto{\pgfqpoint{2.023799in}{1.265644in}}%
\pgfpathlineto{\pgfqpoint{2.474857in}{1.150784in}}%
\pgfpathlineto{\pgfqpoint{2.925915in}{1.296988in}}%
\pgfpathlineto{\pgfqpoint{3.376974in}{1.400441in}}%
\pgfpathlineto{\pgfqpoint{3.828032in}{1.798056in}}%
\pgfpathlineto{\pgfqpoint{4.279090in}{1.712999in}}%
\pgfpathlineto{\pgfqpoint{4.730148in}{2.100027in}}%
\pgfusepath{stroke}%
\end{pgfscope}%
\begin{pgfscope}%
\pgfpathrectangle{\pgfqpoint{0.569136in}{0.499691in}}{\pgfqpoint{4.262500in}{2.310000in}}%
\pgfusepath{clip}%
\pgfsetbuttcap%
\pgfsetmiterjoin%
\definecolor{currentfill}{rgb}{0.862745,0.149020,0.498039}%
\pgfsetfillcolor{currentfill}%
\pgfsetlinewidth{1.003750pt}%
\definecolor{currentstroke}{rgb}{0.862745,0.149020,0.498039}%
\pgfsetstrokecolor{currentstroke}%
\pgfsetdash{}{0pt}%
\pgfsys@defobject{currentmarker}{\pgfqpoint{-0.039627in}{-0.033709in}}{\pgfqpoint{0.039627in}{0.041667in}}{%
\pgfpathmoveto{\pgfqpoint{0.000000in}{0.041667in}}%
\pgfpathlineto{\pgfqpoint{-0.039627in}{0.012876in}}%
\pgfpathlineto{\pgfqpoint{-0.024491in}{-0.033709in}}%
\pgfpathlineto{\pgfqpoint{0.024491in}{-0.033709in}}%
\pgfpathlineto{\pgfqpoint{0.039627in}{0.012876in}}%
\pgfpathlineto{\pgfqpoint{0.000000in}{0.041667in}}%
\pgfpathclose%
\pgfusepath{stroke,fill}%
}%
\begin{pgfscope}%
\pgfsys@transformshift{0.670624in}{0.860851in}%
\pgfsys@useobject{currentmarker}{}%
\end{pgfscope}%
\begin{pgfscope}%
\pgfsys@transformshift{1.121683in}{0.716893in}%
\pgfsys@useobject{currentmarker}{}%
\end{pgfscope}%
\begin{pgfscope}%
\pgfsys@transformshift{1.572741in}{0.979660in}%
\pgfsys@useobject{currentmarker}{}%
\end{pgfscope}%
\begin{pgfscope}%
\pgfsys@transformshift{2.023799in}{1.265644in}%
\pgfsys@useobject{currentmarker}{}%
\end{pgfscope}%
\begin{pgfscope}%
\pgfsys@transformshift{2.474857in}{1.150784in}%
\pgfsys@useobject{currentmarker}{}%
\end{pgfscope}%
\begin{pgfscope}%
\pgfsys@transformshift{2.925915in}{1.296988in}%
\pgfsys@useobject{currentmarker}{}%
\end{pgfscope}%
\begin{pgfscope}%
\pgfsys@transformshift{3.376974in}{1.400441in}%
\pgfsys@useobject{currentmarker}{}%
\end{pgfscope}%
\begin{pgfscope}%
\pgfsys@transformshift{3.828032in}{1.798056in}%
\pgfsys@useobject{currentmarker}{}%
\end{pgfscope}%
\begin{pgfscope}%
\pgfsys@transformshift{4.279090in}{1.712999in}%
\pgfsys@useobject{currentmarker}{}%
\end{pgfscope}%
\begin{pgfscope}%
\pgfsys@transformshift{4.730148in}{2.100027in}%
\pgfsys@useobject{currentmarker}{}%
\end{pgfscope}%
\end{pgfscope}%
\begin{pgfscope}%
\pgfpathrectangle{\pgfqpoint{0.569136in}{0.499691in}}{\pgfqpoint{4.262500in}{2.310000in}}%
\pgfusepath{clip}%
\pgfsetrectcap%
\pgfsetroundjoin%
\pgfsetlinewidth{1.505625pt}%
\definecolor{currentstroke}{rgb}{0.392157,0.560784,1.000000}%
\pgfsetstrokecolor{currentstroke}%
\pgfsetdash{}{0pt}%
\pgfpathmoveto{\pgfqpoint{0.670624in}{0.801069in}}%
\pgfpathlineto{\pgfqpoint{1.121683in}{0.698270in}}%
\pgfpathlineto{\pgfqpoint{1.572741in}{0.888353in}}%
\pgfpathlineto{\pgfqpoint{2.023799in}{1.124441in}}%
\pgfpathlineto{\pgfqpoint{2.474857in}{1.018179in}}%
\pgfpathlineto{\pgfqpoint{2.925915in}{1.124970in}}%
\pgfpathlineto{\pgfqpoint{3.376974in}{1.190512in}}%
\pgfpathlineto{\pgfqpoint{3.828032in}{1.560314in}}%
\pgfpathlineto{\pgfqpoint{4.279090in}{1.434066in}}%
\pgfpathlineto{\pgfqpoint{4.730148in}{1.741069in}}%
\pgfusepath{stroke}%
\end{pgfscope}%
\begin{pgfscope}%
\pgfpathrectangle{\pgfqpoint{0.569136in}{0.499691in}}{\pgfqpoint{4.262500in}{2.310000in}}%
\pgfusepath{clip}%
\pgfsetbuttcap%
\pgfsetmiterjoin%
\definecolor{currentfill}{rgb}{0.392157,0.560784,1.000000}%
\pgfsetfillcolor{currentfill}%
\pgfsetlinewidth{1.003750pt}%
\definecolor{currentstroke}{rgb}{0.392157,0.560784,1.000000}%
\pgfsetstrokecolor{currentstroke}%
\pgfsetdash{}{0pt}%
\pgfsys@defobject{currentmarker}{\pgfqpoint{-0.041667in}{-0.041667in}}{\pgfqpoint{0.041667in}{0.041667in}}{%
\pgfpathmoveto{\pgfqpoint{-0.041667in}{-0.041667in}}%
\pgfpathlineto{\pgfqpoint{0.041667in}{-0.041667in}}%
\pgfpathlineto{\pgfqpoint{0.041667in}{0.041667in}}%
\pgfpathlineto{\pgfqpoint{-0.041667in}{0.041667in}}%
\pgfpathlineto{\pgfqpoint{-0.041667in}{-0.041667in}}%
\pgfpathclose%
\pgfusepath{stroke,fill}%
}%
\begin{pgfscope}%
\pgfsys@transformshift{0.670624in}{0.801069in}%
\pgfsys@useobject{currentmarker}{}%
\end{pgfscope}%
\begin{pgfscope}%
\pgfsys@transformshift{1.121683in}{0.698270in}%
\pgfsys@useobject{currentmarker}{}%
\end{pgfscope}%
\begin{pgfscope}%
\pgfsys@transformshift{1.572741in}{0.888353in}%
\pgfsys@useobject{currentmarker}{}%
\end{pgfscope}%
\begin{pgfscope}%
\pgfsys@transformshift{2.023799in}{1.124441in}%
\pgfsys@useobject{currentmarker}{}%
\end{pgfscope}%
\begin{pgfscope}%
\pgfsys@transformshift{2.474857in}{1.018179in}%
\pgfsys@useobject{currentmarker}{}%
\end{pgfscope}%
\begin{pgfscope}%
\pgfsys@transformshift{2.925915in}{1.124970in}%
\pgfsys@useobject{currentmarker}{}%
\end{pgfscope}%
\begin{pgfscope}%
\pgfsys@transformshift{3.376974in}{1.190512in}%
\pgfsys@useobject{currentmarker}{}%
\end{pgfscope}%
\begin{pgfscope}%
\pgfsys@transformshift{3.828032in}{1.560314in}%
\pgfsys@useobject{currentmarker}{}%
\end{pgfscope}%
\begin{pgfscope}%
\pgfsys@transformshift{4.279090in}{1.434066in}%
\pgfsys@useobject{currentmarker}{}%
\end{pgfscope}%
\begin{pgfscope}%
\pgfsys@transformshift{4.730148in}{1.741069in}%
\pgfsys@useobject{currentmarker}{}%
\end{pgfscope}%
\end{pgfscope}%
\begin{pgfscope}%
\pgfsetrectcap%
\pgfsetmiterjoin%
\pgfsetlinewidth{0.803000pt}%
\definecolor{currentstroke}{rgb}{0.000000,0.000000,0.000000}%
\pgfsetstrokecolor{currentstroke}%
\pgfsetdash{}{0pt}%
\pgfpathmoveto{\pgfqpoint{0.569136in}{0.499691in}}%
\pgfpathlineto{\pgfqpoint{0.569136in}{2.809691in}}%
\pgfusepath{stroke}%
\end{pgfscope}%
\begin{pgfscope}%
\pgfsetrectcap%
\pgfsetmiterjoin%
\pgfsetlinewidth{0.803000pt}%
\definecolor{currentstroke}{rgb}{0.000000,0.000000,0.000000}%
\pgfsetstrokecolor{currentstroke}%
\pgfsetdash{}{0pt}%
\pgfpathmoveto{\pgfqpoint{4.831636in}{0.499691in}}%
\pgfpathlineto{\pgfqpoint{4.831636in}{2.809691in}}%
\pgfusepath{stroke}%
\end{pgfscope}%
\begin{pgfscope}%
\pgfsetrectcap%
\pgfsetmiterjoin%
\pgfsetlinewidth{0.803000pt}%
\definecolor{currentstroke}{rgb}{0.000000,0.000000,0.000000}%
\pgfsetstrokecolor{currentstroke}%
\pgfsetdash{}{0pt}%
\pgfpathmoveto{\pgfqpoint{0.569136in}{0.499691in}}%
\pgfpathlineto{\pgfqpoint{4.831636in}{0.499691in}}%
\pgfusepath{stroke}%
\end{pgfscope}%
\begin{pgfscope}%
\pgfsetrectcap%
\pgfsetmiterjoin%
\pgfsetlinewidth{0.803000pt}%
\definecolor{currentstroke}{rgb}{0.000000,0.000000,0.000000}%
\pgfsetstrokecolor{currentstroke}%
\pgfsetdash{}{0pt}%
\pgfpathmoveto{\pgfqpoint{0.569136in}{2.809691in}}%
\pgfpathlineto{\pgfqpoint{4.831636in}{2.809691in}}%
\pgfusepath{stroke}%
\end{pgfscope}%
\begin{pgfscope}%
\pgfsetbuttcap%
\pgfsetmiterjoin%
\definecolor{currentfill}{rgb}{1.000000,1.000000,1.000000}%
\pgfsetfillcolor{currentfill}%
\pgfsetfillopacity{0.800000}%
\pgfsetlinewidth{1.003750pt}%
\definecolor{currentstroke}{rgb}{0.800000,0.800000,0.800000}%
\pgfsetstrokecolor{currentstroke}%
\pgfsetstrokeopacity{0.800000}%
\pgfsetdash{}{0pt}%
\pgfpathmoveto{\pgfqpoint{0.638581in}{2.117562in}}%
\pgfpathlineto{\pgfqpoint{1.452816in}{2.117562in}}%
\pgfpathlineto{\pgfqpoint{1.452816in}{2.740247in}}%
\pgfpathlineto{\pgfqpoint{0.638581in}{2.740247in}}%
\pgfpathlineto{\pgfqpoint{0.638581in}{2.117562in}}%
\pgfpathclose%
\pgfusepath{stroke,fill}%
\end{pgfscope}%
\begin{pgfscope}%
\pgfsetrectcap%
\pgfsetroundjoin%
\pgfsetlinewidth{1.505625pt}%
\definecolor{currentstroke}{rgb}{1.000000,0.690196,0.000000}%
\pgfsetstrokecolor{currentstroke}%
\pgfsetdash{}{0pt}%
\pgfpathmoveto{\pgfqpoint{0.694136in}{2.636080in}}%
\pgfpathlineto{\pgfqpoint{0.833025in}{2.636080in}}%
\pgfpathlineto{\pgfqpoint{0.971914in}{2.636080in}}%
\pgfusepath{stroke}%
\end{pgfscope}%
\begin{pgfscope}%
\pgfsetbuttcap%
\pgfsetmiterjoin%
\definecolor{currentfill}{rgb}{1.000000,0.690196,0.000000}%
\pgfsetfillcolor{currentfill}%
\pgfsetlinewidth{1.003750pt}%
\definecolor{currentstroke}{rgb}{1.000000,0.690196,0.000000}%
\pgfsetstrokecolor{currentstroke}%
\pgfsetdash{}{0pt}%
\pgfsys@defobject{currentmarker}{\pgfqpoint{-0.026517in}{-0.044194in}}{\pgfqpoint{0.026517in}{0.044194in}}{%
\pgfpathmoveto{\pgfqpoint{-0.000000in}{-0.044194in}}%
\pgfpathlineto{\pgfqpoint{0.026517in}{0.000000in}}%
\pgfpathlineto{\pgfqpoint{0.000000in}{0.044194in}}%
\pgfpathlineto{\pgfqpoint{-0.026517in}{0.000000in}}%
\pgfpathlineto{\pgfqpoint{-0.000000in}{-0.044194in}}%
\pgfpathclose%
\pgfusepath{stroke,fill}%
}%
\begin{pgfscope}%
\pgfsys@transformshift{0.833025in}{2.636080in}%
\pgfsys@useobject{currentmarker}{}%
\end{pgfscope}%
\end{pgfscope}%
\begin{pgfscope}%
\definecolor{textcolor}{rgb}{0.000000,0.000000,0.000000}%
\pgfsetstrokecolor{textcolor}%
\pgfsetfillcolor{textcolor}%
\pgftext[x=1.083025in,y=2.587469in,left,base]{\color{textcolor}{\rmfamily\fontsize{10.000000}{12.000000}\selectfont\catcode`\^=\active\def^{\ifmmode\sp\else\^{}\fi}\catcode`\%=\active\def
\end{pgfscope}%
\begin{pgfscope}%
\pgfsetrectcap%
\pgfsetroundjoin%
\pgfsetlinewidth{1.505625pt}%
\definecolor{currentstroke}{rgb}{0.862745,0.149020,0.498039}%
\pgfsetstrokecolor{currentstroke}%
\pgfsetdash{}{0pt}%
\pgfpathmoveto{\pgfqpoint{0.694136in}{2.442407in}}%
\pgfpathlineto{\pgfqpoint{0.833025in}{2.442407in}}%
\pgfpathlineto{\pgfqpoint{0.971914in}{2.442407in}}%
\pgfusepath{stroke}%
\end{pgfscope}%
\begin{pgfscope}%
\pgfsetbuttcap%
\pgfsetmiterjoin%
\definecolor{currentfill}{rgb}{0.862745,0.149020,0.498039}%
\pgfsetfillcolor{currentfill}%
\pgfsetlinewidth{1.003750pt}%
\definecolor{currentstroke}{rgb}{0.862745,0.149020,0.498039}%
\pgfsetstrokecolor{currentstroke}%
\pgfsetdash{}{0pt}%
\pgfsys@defobject{currentmarker}{\pgfqpoint{-0.029721in}{-0.025282in}}{\pgfqpoint{0.029721in}{0.031250in}}{%
\pgfpathmoveto{\pgfqpoint{0.000000in}{0.031250in}}%
\pgfpathlineto{\pgfqpoint{-0.029721in}{0.009657in}}%
\pgfpathlineto{\pgfqpoint{-0.018368in}{-0.025282in}}%
\pgfpathlineto{\pgfqpoint{0.018368in}{-0.025282in}}%
\pgfpathlineto{\pgfqpoint{0.029721in}{0.009657in}}%
\pgfpathlineto{\pgfqpoint{0.000000in}{0.031250in}}%
\pgfpathclose%
\pgfusepath{stroke,fill}%
}%
\begin{pgfscope}%
\pgfsys@transformshift{0.833025in}{2.442407in}%
\pgfsys@useobject{currentmarker}{}%
\end{pgfscope}%
\end{pgfscope}%
\begin{pgfscope}%
\definecolor{textcolor}{rgb}{0.000000,0.000000,0.000000}%
\pgfsetstrokecolor{textcolor}%
\pgfsetfillcolor{textcolor}%
\pgftext[x=1.083025in,y=2.393796in,left,base]{\color{textcolor}{\rmfamily\fontsize{10.000000}{12.000000}\selectfont\catcode`\^=\active\def^{\ifmmode\sp\else\^{}\fi}\catcode`\%=\active\def
\end{pgfscope}%
\begin{pgfscope}%
\pgfsetrectcap%
\pgfsetroundjoin%
\pgfsetlinewidth{1.505625pt}%
\definecolor{currentstroke}{rgb}{0.392157,0.560784,1.000000}%
\pgfsetstrokecolor{currentstroke}%
\pgfsetdash{}{0pt}%
\pgfpathmoveto{\pgfqpoint{0.694136in}{2.248734in}}%
\pgfpathlineto{\pgfqpoint{0.833025in}{2.248734in}}%
\pgfpathlineto{\pgfqpoint{0.971914in}{2.248734in}}%
\pgfusepath{stroke}%
\end{pgfscope}%
\begin{pgfscope}%
\pgfsetbuttcap%
\pgfsetmiterjoin%
\definecolor{currentfill}{rgb}{0.392157,0.560784,1.000000}%
\pgfsetfillcolor{currentfill}%
\pgfsetlinewidth{1.003750pt}%
\definecolor{currentstroke}{rgb}{0.392157,0.560784,1.000000}%
\pgfsetstrokecolor{currentstroke}%
\pgfsetdash{}{0pt}%
\pgfsys@defobject{currentmarker}{\pgfqpoint{-0.031250in}{-0.031250in}}{\pgfqpoint{0.031250in}{0.031250in}}{%
\pgfpathmoveto{\pgfqpoint{-0.031250in}{-0.031250in}}%
\pgfpathlineto{\pgfqpoint{0.031250in}{-0.031250in}}%
\pgfpathlineto{\pgfqpoint{0.031250in}{0.031250in}}%
\pgfpathlineto{\pgfqpoint{-0.031250in}{0.031250in}}%
\pgfpathlineto{\pgfqpoint{-0.031250in}{-0.031250in}}%
\pgfpathclose%
\pgfusepath{stroke,fill}%
}%
\begin{pgfscope}%
\pgfsys@transformshift{0.833025in}{2.248734in}%
\pgfsys@useobject{currentmarker}{}%
\end{pgfscope}%
\end{pgfscope}%
\begin{pgfscope}%
\definecolor{textcolor}{rgb}{0.000000,0.000000,0.000000}%
\pgfsetstrokecolor{textcolor}%
\pgfsetfillcolor{textcolor}%
\pgftext[x=1.083025in,y=2.200123in,left,base]{\color{textcolor}{\rmfamily\fontsize{10.000000}{12.000000}\selectfont\catcode`\^=\active\def^{\ifmmode\sp\else\^{}\fi}\catcode`\%=\active\def
\end{pgfscope}%
\end{pgfpicture}%
\makeatother%
\endgroup%

%% file: plots/comparison_large_fits.pgf
\begingroup%
\makeatletter%
\begin{pgfpicture}%
\pgfpathrectangle{\pgfpointorigin}{\pgfqpoint{2.462068in}{2.195404in}}%
\pgfusepath{use as bounding box, clip}%
\begin{pgfscope}%
\pgfsetbuttcap%
\pgfsetmiterjoin%
\definecolor{currentfill}{rgb}{1.000000,1.000000,1.000000}%
\pgfsetfillcolor{currentfill}%
\pgfsetlinewidth{0.000000pt}%
\definecolor{currentstroke}{rgb}{1.000000,1.000000,1.000000}%
\pgfsetstrokecolor{currentstroke}%
\pgfsetdash{}{0pt}%
\pgfpathmoveto{\pgfqpoint{0.000000in}{0.000000in}}%
\pgfpathlineto{\pgfqpoint{2.462068in}{0.000000in}}%
\pgfpathlineto{\pgfqpoint{2.462068in}{2.195404in}}%
\pgfpathlineto{\pgfqpoint{0.000000in}{2.195404in}}%
\pgfpathlineto{\pgfqpoint{0.000000in}{0.000000in}}%
\pgfpathclose%
\pgfusepath{fill}%
\end{pgfscope}%
\begin{pgfscope}%
\pgfsetbuttcap%
\pgfsetmiterjoin%
\definecolor{currentfill}{rgb}{1.000000,1.000000,1.000000}%
\pgfsetfillcolor{currentfill}%
\pgfsetlinewidth{0.000000pt}%
\definecolor{currentstroke}{rgb}{0.000000,0.000000,0.000000}%
\pgfsetstrokecolor{currentstroke}%
\pgfsetstrokeopacity{0.000000}%
\pgfsetdash{}{0pt}%
\pgfpathmoveto{\pgfqpoint{0.504568in}{0.240679in}}%
\pgfpathlineto{\pgfqpoint{2.442068in}{0.240679in}}%
\pgfpathlineto{\pgfqpoint{2.442068in}{2.127179in}}%
\pgfpathlineto{\pgfqpoint{0.504568in}{2.127179in}}%
\pgfpathlineto{\pgfqpoint{0.504568in}{0.240679in}}%
\pgfpathclose%
\pgfusepath{fill}%
\end{pgfscope}%
\begin{pgfscope}%
\pgfpathrectangle{\pgfqpoint{0.504568in}{0.240679in}}{\pgfqpoint{1.937500in}{1.886500in}}%
\pgfusepath{clip}%
\pgfsetrectcap%
\pgfsetroundjoin%
\pgfsetlinewidth{0.250937pt}%
\definecolor{currentstroke}{rgb}{0.000000,0.000000,0.000000}%
\pgfsetstrokecolor{currentstroke}%
\pgfsetstrokeopacity{0.200000}%
\pgfsetdash{}{0pt}%
\pgfpathmoveto{\pgfqpoint{0.827485in}{0.240679in}}%
\pgfpathlineto{\pgfqpoint{0.827485in}{2.127179in}}%
\pgfusepath{stroke}%
\end{pgfscope}%
\begin{pgfscope}%
\pgfsetbuttcap%
\pgfsetroundjoin%
\definecolor{currentfill}{rgb}{0.000000,0.000000,0.000000}%
\pgfsetfillcolor{currentfill}%
\pgfsetlinewidth{0.803000pt}%
\definecolor{currentstroke}{rgb}{0.000000,0.000000,0.000000}%
\pgfsetstrokecolor{currentstroke}%
\pgfsetdash{}{0pt}%
\pgfsys@defobject{currentmarker}{\pgfqpoint{0.000000in}{-0.048611in}}{\pgfqpoint{0.000000in}{0.000000in}}{%
\pgfpathmoveto{\pgfqpoint{0.000000in}{0.000000in}}%
\pgfpathlineto{\pgfqpoint{0.000000in}{-0.048611in}}%
\pgfusepath{stroke,fill}%
}%
\begin{pgfscope}%
\pgfsys@transformshift{0.827485in}{0.240679in}%
\pgfsys@useobject{currentmarker}{}%
\end{pgfscope}%
\end{pgfscope}%
\begin{pgfscope}%
\definecolor{textcolor}{rgb}{0.000000,0.000000,0.000000}%
\pgfsetstrokecolor{textcolor}%
\pgfsetfillcolor{textcolor}%
\pgftext[x=0.827485in,y=0.143457in,,top]{\color{textcolor}{\rmfamily\fontsize{10.000000}{12.000000}\selectfont\catcode`\^=\active\def^{\ifmmode\sp\else\^{}\fi}\catcode`\%=\active\def
\end{pgfscope}%
\begin{pgfscope}%
\pgfpathrectangle{\pgfqpoint{0.504568in}{0.240679in}}{\pgfqpoint{1.937500in}{1.886500in}}%
\pgfusepath{clip}%
\pgfsetrectcap%
\pgfsetroundjoin%
\pgfsetlinewidth{0.250937pt}%
\definecolor{currentstroke}{rgb}{0.000000,0.000000,0.000000}%
\pgfsetstrokecolor{currentstroke}%
\pgfsetstrokeopacity{0.200000}%
\pgfsetdash{}{0pt}%
\pgfpathmoveto{\pgfqpoint{1.473318in}{0.240679in}}%
\pgfpathlineto{\pgfqpoint{1.473318in}{2.127179in}}%
\pgfusepath{stroke}%
\end{pgfscope}%
\begin{pgfscope}%
\pgfsetbuttcap%
\pgfsetroundjoin%
\definecolor{currentfill}{rgb}{0.000000,0.000000,0.000000}%
\pgfsetfillcolor{currentfill}%
\pgfsetlinewidth{0.803000pt}%
\definecolor{currentstroke}{rgb}{0.000000,0.000000,0.000000}%
\pgfsetstrokecolor{currentstroke}%
\pgfsetdash{}{0pt}%
\pgfsys@defobject{currentmarker}{\pgfqpoint{0.000000in}{-0.048611in}}{\pgfqpoint{0.000000in}{0.000000in}}{%
\pgfpathmoveto{\pgfqpoint{0.000000in}{0.000000in}}%
\pgfpathlineto{\pgfqpoint{0.000000in}{-0.048611in}}%
\pgfusepath{stroke,fill}%
}%
\begin{pgfscope}%
\pgfsys@transformshift{1.473318in}{0.240679in}%
\pgfsys@useobject{currentmarker}{}%
\end{pgfscope}%
\end{pgfscope}%
\begin{pgfscope}%
\definecolor{textcolor}{rgb}{0.000000,0.000000,0.000000}%
\pgfsetstrokecolor{textcolor}%
\pgfsetfillcolor{textcolor}%
\pgftext[x=1.473318in,y=0.143457in,,top]{\color{textcolor}{\rmfamily\fontsize{10.000000}{12.000000}\selectfont\catcode`\^=\active\def^{\ifmmode\sp\else\^{}\fi}\catcode`\%=\active\def
\end{pgfscope}%
\begin{pgfscope}%
\pgfpathrectangle{\pgfqpoint{0.504568in}{0.240679in}}{\pgfqpoint{1.937500in}{1.886500in}}%
\pgfusepath{clip}%
\pgfsetrectcap%
\pgfsetroundjoin%
\pgfsetlinewidth{0.250937pt}%
\definecolor{currentstroke}{rgb}{0.000000,0.000000,0.000000}%
\pgfsetstrokecolor{currentstroke}%
\pgfsetstrokeopacity{0.200000}%
\pgfsetdash{}{0pt}%
\pgfpathmoveto{\pgfqpoint{2.119152in}{0.240679in}}%
\pgfpathlineto{\pgfqpoint{2.119152in}{2.127179in}}%
\pgfusepath{stroke}%
\end{pgfscope}%
\begin{pgfscope}%
\pgfsetbuttcap%
\pgfsetroundjoin%
\definecolor{currentfill}{rgb}{0.000000,0.000000,0.000000}%
\pgfsetfillcolor{currentfill}%
\pgfsetlinewidth{0.803000pt}%
\definecolor{currentstroke}{rgb}{0.000000,0.000000,0.000000}%
\pgfsetstrokecolor{currentstroke}%
\pgfsetdash{}{0pt}%
\pgfsys@defobject{currentmarker}{\pgfqpoint{0.000000in}{-0.048611in}}{\pgfqpoint{0.000000in}{0.000000in}}{%
\pgfpathmoveto{\pgfqpoint{0.000000in}{0.000000in}}%
\pgfpathlineto{\pgfqpoint{0.000000in}{-0.048611in}}%
\pgfusepath{stroke,fill}%
}%
\begin{pgfscope}%
\pgfsys@transformshift{2.119152in}{0.240679in}%
\pgfsys@useobject{currentmarker}{}%
\end{pgfscope}%
\end{pgfscope}%
\begin{pgfscope}%
\definecolor{textcolor}{rgb}{0.000000,0.000000,0.000000}%
\pgfsetstrokecolor{textcolor}%
\pgfsetfillcolor{textcolor}%
\pgftext[x=2.119152in,y=0.143457in,,top]{\color{textcolor}{\rmfamily\fontsize{10.000000}{12.000000}\selectfont\catcode`\^=\active\def^{\ifmmode\sp\else\^{}\fi}\catcode`\%=\active\def
\end{pgfscope}%
\begin{pgfscope}%
\pgfpathrectangle{\pgfqpoint{0.504568in}{0.240679in}}{\pgfqpoint{1.937500in}{1.886500in}}%
\pgfusepath{clip}%
\pgfsetrectcap%
\pgfsetroundjoin%
\pgfsetlinewidth{0.250937pt}%
\definecolor{currentstroke}{rgb}{0.000000,0.000000,0.000000}%
\pgfsetstrokecolor{currentstroke}%
\pgfsetstrokeopacity{0.200000}%
\pgfsetdash{}{0pt}%
\pgfpathmoveto{\pgfqpoint{0.504568in}{0.240679in}}%
\pgfpathlineto{\pgfqpoint{2.442068in}{0.240679in}}%
\pgfusepath{stroke}%
\end{pgfscope}%
\begin{pgfscope}%
\pgfsetbuttcap%
\pgfsetroundjoin%
\definecolor{currentfill}{rgb}{0.000000,0.000000,0.000000}%
\pgfsetfillcolor{currentfill}%
\pgfsetlinewidth{0.803000pt}%
\definecolor{currentstroke}{rgb}{0.000000,0.000000,0.000000}%
\pgfsetstrokecolor{currentstroke}%
\pgfsetdash{}{0pt}%
\pgfsys@defobject{currentmarker}{\pgfqpoint{-0.048611in}{0.000000in}}{\pgfqpoint{-0.000000in}{0.000000in}}{%
\pgfpathmoveto{\pgfqpoint{-0.000000in}{0.000000in}}%
\pgfpathlineto{\pgfqpoint{-0.048611in}{0.000000in}}%
\pgfusepath{stroke,fill}%
}%
\begin{pgfscope}%
\pgfsys@transformshift{0.504568in}{0.240679in}%
\pgfsys@useobject{currentmarker}{}%
\end{pgfscope}%
\end{pgfscope}%
\begin{pgfscope}%
\definecolor{textcolor}{rgb}{0.000000,0.000000,0.000000}%
\pgfsetstrokecolor{textcolor}%
\pgfsetfillcolor{textcolor}%
\pgftext[x=0.268457in, y=0.192454in, left, base]{\color{textcolor}{\rmfamily\fontsize{10.000000}{12.000000}\selectfont\catcode`\^=\active\def^{\ifmmode\sp\else\^{}\fi}\catcode`\%=\active\def
\end{pgfscope}%
\begin{pgfscope}%
\pgfpathrectangle{\pgfqpoint{0.504568in}{0.240679in}}{\pgfqpoint{1.937500in}{1.886500in}}%
\pgfusepath{clip}%
\pgfsetrectcap%
\pgfsetroundjoin%
\pgfsetlinewidth{0.250937pt}%
\definecolor{currentstroke}{rgb}{0.000000,0.000000,0.000000}%
\pgfsetstrokecolor{currentstroke}%
\pgfsetstrokeopacity{0.200000}%
\pgfsetdash{}{0pt}%
\pgfpathmoveto{\pgfqpoint{0.504568in}{0.555096in}}%
\pgfpathlineto{\pgfqpoint{2.442068in}{0.555096in}}%
\pgfusepath{stroke}%
\end{pgfscope}%
\begin{pgfscope}%
\pgfsetbuttcap%
\pgfsetroundjoin%
\definecolor{currentfill}{rgb}{0.000000,0.000000,0.000000}%
\pgfsetfillcolor{currentfill}%
\pgfsetlinewidth{0.803000pt}%
\definecolor{currentstroke}{rgb}{0.000000,0.000000,0.000000}%
\pgfsetstrokecolor{currentstroke}%
\pgfsetdash{}{0pt}%
\pgfsys@defobject{currentmarker}{\pgfqpoint{-0.048611in}{0.000000in}}{\pgfqpoint{-0.000000in}{0.000000in}}{%
\pgfpathmoveto{\pgfqpoint{-0.000000in}{0.000000in}}%
\pgfpathlineto{\pgfqpoint{-0.048611in}{0.000000in}}%
\pgfusepath{stroke,fill}%
}%
\begin{pgfscope}%
\pgfsys@transformshift{0.504568in}{0.555096in}%
\pgfsys@useobject{currentmarker}{}%
\end{pgfscope}%
\end{pgfscope}%
\begin{pgfscope}%
\definecolor{textcolor}{rgb}{0.000000,0.000000,0.000000}%
\pgfsetstrokecolor{textcolor}%
\pgfsetfillcolor{textcolor}%
\pgftext[x=0.268457in, y=0.506870in, left, base]{\color{textcolor}{\rmfamily\fontsize{10.000000}{12.000000}\selectfont\catcode`\^=\active\def^{\ifmmode\sp\else\^{}\fi}\catcode`\%=\active\def
\end{pgfscope}%
\begin{pgfscope}%
\pgfpathrectangle{\pgfqpoint{0.504568in}{0.240679in}}{\pgfqpoint{1.937500in}{1.886500in}}%
\pgfusepath{clip}%
\pgfsetrectcap%
\pgfsetroundjoin%
\pgfsetlinewidth{0.250937pt}%
\definecolor{currentstroke}{rgb}{0.000000,0.000000,0.000000}%
\pgfsetstrokecolor{currentstroke}%
\pgfsetstrokeopacity{0.200000}%
\pgfsetdash{}{0pt}%
\pgfpathmoveto{\pgfqpoint{0.504568in}{0.869512in}}%
\pgfpathlineto{\pgfqpoint{2.442068in}{0.869512in}}%
\pgfusepath{stroke}%
\end{pgfscope}%
\begin{pgfscope}%
\pgfsetbuttcap%
\pgfsetroundjoin%
\definecolor{currentfill}{rgb}{0.000000,0.000000,0.000000}%
\pgfsetfillcolor{currentfill}%
\pgfsetlinewidth{0.803000pt}%
\definecolor{currentstroke}{rgb}{0.000000,0.000000,0.000000}%
\pgfsetstrokecolor{currentstroke}%
\pgfsetdash{}{0pt}%
\pgfsys@defobject{currentmarker}{\pgfqpoint{-0.048611in}{0.000000in}}{\pgfqpoint{-0.000000in}{0.000000in}}{%
\pgfpathmoveto{\pgfqpoint{-0.000000in}{0.000000in}}%
\pgfpathlineto{\pgfqpoint{-0.048611in}{0.000000in}}%
\pgfusepath{stroke,fill}%
}%
\begin{pgfscope}%
\pgfsys@transformshift{0.504568in}{0.869512in}%
\pgfsys@useobject{currentmarker}{}%
\end{pgfscope}%
\end{pgfscope}%
\begin{pgfscope}%
\definecolor{textcolor}{rgb}{0.000000,0.000000,0.000000}%
\pgfsetstrokecolor{textcolor}%
\pgfsetfillcolor{textcolor}%
\pgftext[x=0.268457in, y=0.821287in, left, base]{\color{textcolor}{\rmfamily\fontsize{10.000000}{12.000000}\selectfont\catcode`\^=\active\def^{\ifmmode\sp\else\^{}\fi}\catcode`\%=\active\def
\end{pgfscope}%
\begin{pgfscope}%
\pgfpathrectangle{\pgfqpoint{0.504568in}{0.240679in}}{\pgfqpoint{1.937500in}{1.886500in}}%
\pgfusepath{clip}%
\pgfsetrectcap%
\pgfsetroundjoin%
\pgfsetlinewidth{0.250937pt}%
\definecolor{currentstroke}{rgb}{0.000000,0.000000,0.000000}%
\pgfsetstrokecolor{currentstroke}%
\pgfsetstrokeopacity{0.200000}%
\pgfsetdash{}{0pt}%
\pgfpathmoveto{\pgfqpoint{0.504568in}{1.183929in}}%
\pgfpathlineto{\pgfqpoint{2.442068in}{1.183929in}}%
\pgfusepath{stroke}%
\end{pgfscope}%
\begin{pgfscope}%
\pgfsetbuttcap%
\pgfsetroundjoin%
\definecolor{currentfill}{rgb}{0.000000,0.000000,0.000000}%
\pgfsetfillcolor{currentfill}%
\pgfsetlinewidth{0.803000pt}%
\definecolor{currentstroke}{rgb}{0.000000,0.000000,0.000000}%
\pgfsetstrokecolor{currentstroke}%
\pgfsetdash{}{0pt}%
\pgfsys@defobject{currentmarker}{\pgfqpoint{-0.048611in}{0.000000in}}{\pgfqpoint{-0.000000in}{0.000000in}}{%
\pgfpathmoveto{\pgfqpoint{-0.000000in}{0.000000in}}%
\pgfpathlineto{\pgfqpoint{-0.048611in}{0.000000in}}%
\pgfusepath{stroke,fill}%
}%
\begin{pgfscope}%
\pgfsys@transformshift{0.504568in}{1.183929in}%
\pgfsys@useobject{currentmarker}{}%
\end{pgfscope}%
\end{pgfscope}%
\begin{pgfscope}%
\definecolor{textcolor}{rgb}{0.000000,0.000000,0.000000}%
\pgfsetstrokecolor{textcolor}%
\pgfsetfillcolor{textcolor}%
\pgftext[x=0.268457in, y=1.135704in, left, base]{\color{textcolor}{\rmfamily\fontsize{10.000000}{12.000000}\selectfont\catcode`\^=\active\def^{\ifmmode\sp\else\^{}\fi}\catcode`\%=\active\def
\end{pgfscope}%
\begin{pgfscope}%
\pgfpathrectangle{\pgfqpoint{0.504568in}{0.240679in}}{\pgfqpoint{1.937500in}{1.886500in}}%
\pgfusepath{clip}%
\pgfsetrectcap%
\pgfsetroundjoin%
\pgfsetlinewidth{0.250937pt}%
\definecolor{currentstroke}{rgb}{0.000000,0.000000,0.000000}%
\pgfsetstrokecolor{currentstroke}%
\pgfsetstrokeopacity{0.200000}%
\pgfsetdash{}{0pt}%
\pgfpathmoveto{\pgfqpoint{0.504568in}{1.498346in}}%
\pgfpathlineto{\pgfqpoint{2.442068in}{1.498346in}}%
\pgfusepath{stroke}%
\end{pgfscope}%
\begin{pgfscope}%
\pgfsetbuttcap%
\pgfsetroundjoin%
\definecolor{currentfill}{rgb}{0.000000,0.000000,0.000000}%
\pgfsetfillcolor{currentfill}%
\pgfsetlinewidth{0.803000pt}%
\definecolor{currentstroke}{rgb}{0.000000,0.000000,0.000000}%
\pgfsetstrokecolor{currentstroke}%
\pgfsetdash{}{0pt}%
\pgfsys@defobject{currentmarker}{\pgfqpoint{-0.048611in}{0.000000in}}{\pgfqpoint{-0.000000in}{0.000000in}}{%
\pgfpathmoveto{\pgfqpoint{-0.000000in}{0.000000in}}%
\pgfpathlineto{\pgfqpoint{-0.048611in}{0.000000in}}%
\pgfusepath{stroke,fill}%
}%
\begin{pgfscope}%
\pgfsys@transformshift{0.504568in}{1.498346in}%
\pgfsys@useobject{currentmarker}{}%
\end{pgfscope}%
\end{pgfscope}%
\begin{pgfscope}%
\definecolor{textcolor}{rgb}{0.000000,0.000000,0.000000}%
\pgfsetstrokecolor{textcolor}%
\pgfsetfillcolor{textcolor}%
\pgftext[x=0.268457in, y=1.450120in, left, base]{\color{textcolor}{\rmfamily\fontsize{10.000000}{12.000000}\selectfont\catcode`\^=\active\def^{\ifmmode\sp\else\^{}\fi}\catcode`\%=\active\def
\end{pgfscope}%
\begin{pgfscope}%
\pgfpathrectangle{\pgfqpoint{0.504568in}{0.240679in}}{\pgfqpoint{1.937500in}{1.886500in}}%
\pgfusepath{clip}%
\pgfsetrectcap%
\pgfsetroundjoin%
\pgfsetlinewidth{0.250937pt}%
\definecolor{currentstroke}{rgb}{0.000000,0.000000,0.000000}%
\pgfsetstrokecolor{currentstroke}%
\pgfsetstrokeopacity{0.200000}%
\pgfsetdash{}{0pt}%
\pgfpathmoveto{\pgfqpoint{0.504568in}{1.812762in}}%
\pgfpathlineto{\pgfqpoint{2.442068in}{1.812762in}}%
\pgfusepath{stroke}%
\end{pgfscope}%
\begin{pgfscope}%
\pgfsetbuttcap%
\pgfsetroundjoin%
\definecolor{currentfill}{rgb}{0.000000,0.000000,0.000000}%
\pgfsetfillcolor{currentfill}%
\pgfsetlinewidth{0.803000pt}%
\definecolor{currentstroke}{rgb}{0.000000,0.000000,0.000000}%
\pgfsetstrokecolor{currentstroke}%
\pgfsetdash{}{0pt}%
\pgfsys@defobject{currentmarker}{\pgfqpoint{-0.048611in}{0.000000in}}{\pgfqpoint{-0.000000in}{0.000000in}}{%
\pgfpathmoveto{\pgfqpoint{-0.000000in}{0.000000in}}%
\pgfpathlineto{\pgfqpoint{-0.048611in}{0.000000in}}%
\pgfusepath{stroke,fill}%
}%
\begin{pgfscope}%
\pgfsys@transformshift{0.504568in}{1.812762in}%
\pgfsys@useobject{currentmarker}{}%
\end{pgfscope}%
\end{pgfscope}%
\begin{pgfscope}%
\definecolor{textcolor}{rgb}{0.000000,0.000000,0.000000}%
\pgfsetstrokecolor{textcolor}%
\pgfsetfillcolor{textcolor}%
\pgftext[x=0.268457in, y=1.764537in, left, base]{\color{textcolor}{\rmfamily\fontsize{10.000000}{12.000000}\selectfont\catcode`\^=\active\def^{\ifmmode\sp\else\^{}\fi}\catcode`\%=\active\def
\end{pgfscope}%
\begin{pgfscope}%
\pgfpathrectangle{\pgfqpoint{0.504568in}{0.240679in}}{\pgfqpoint{1.937500in}{1.886500in}}%
\pgfusepath{clip}%
\pgfsetrectcap%
\pgfsetroundjoin%
\pgfsetlinewidth{0.250937pt}%
\definecolor{currentstroke}{rgb}{0.000000,0.000000,0.000000}%
\pgfsetstrokecolor{currentstroke}%
\pgfsetstrokeopacity{0.200000}%
\pgfsetdash{}{0pt}%
\pgfpathmoveto{\pgfqpoint{0.504568in}{2.127179in}}%
\pgfpathlineto{\pgfqpoint{2.442068in}{2.127179in}}%
\pgfusepath{stroke}%
\end{pgfscope}%
\begin{pgfscope}%
\pgfsetbuttcap%
\pgfsetroundjoin%
\definecolor{currentfill}{rgb}{0.000000,0.000000,0.000000}%
\pgfsetfillcolor{currentfill}%
\pgfsetlinewidth{0.803000pt}%
\definecolor{currentstroke}{rgb}{0.000000,0.000000,0.000000}%
\pgfsetstrokecolor{currentstroke}%
\pgfsetdash{}{0pt}%
\pgfsys@defobject{currentmarker}{\pgfqpoint{-0.048611in}{0.000000in}}{\pgfqpoint{-0.000000in}{0.000000in}}{%
\pgfpathmoveto{\pgfqpoint{-0.000000in}{0.000000in}}%
\pgfpathlineto{\pgfqpoint{-0.048611in}{0.000000in}}%
\pgfusepath{stroke,fill}%
}%
\begin{pgfscope}%
\pgfsys@transformshift{0.504568in}{2.127179in}%
\pgfsys@useobject{currentmarker}{}%
\end{pgfscope}%
\end{pgfscope}%
\begin{pgfscope}%
\definecolor{textcolor}{rgb}{0.000000,0.000000,0.000000}%
\pgfsetstrokecolor{textcolor}%
\pgfsetfillcolor{textcolor}%
\pgftext[x=0.199012in, y=2.078954in, left, base]{\color{textcolor}{\rmfamily\fontsize{10.000000}{12.000000}\selectfont\catcode`\^=\active\def^{\ifmmode\sp\else\^{}\fi}\catcode`\%=\active\def
\end{pgfscope}%
\begin{pgfscope}%
\definecolor{textcolor}{rgb}{0.000000,0.000000,0.000000}%
\pgfsetstrokecolor{textcolor}%
\pgfsetfillcolor{textcolor}%
\pgftext[x=0.143457in,y=1.183929in,,bottom,rotate=90.000000]{\color{textcolor}{\rmfamily\fontsize{10.000000}{12.000000}\selectfont\catcode`\^=\active\def^{\ifmmode\sp\else\^{}\fi}\catcode`\%=\active\def
\end{pgfscope}%
\begin{pgfscope}%
\pgfpathrectangle{\pgfqpoint{0.504568in}{0.240679in}}{\pgfqpoint{1.937500in}{1.886500in}}%
\pgfusepath{clip}%
\pgfsetbuttcap%
\pgfsetmiterjoin%
\definecolor{currentfill}{rgb}{1.000000,0.690196,0.000000}%
\pgfsetfillcolor{currentfill}%
\pgfsetlinewidth{1.003750pt}%
\definecolor{currentstroke}{rgb}{0.000000,0.000000,0.000000}%
\pgfsetstrokecolor{currentstroke}%
\pgfsetdash{}{0pt}%
\pgfpathmoveto{\pgfqpoint{0.730610in}{1.884712in}}%
\pgfpathlineto{\pgfqpoint{0.924360in}{1.884712in}}%
\pgfpathlineto{\pgfqpoint{0.924360in}{1.967979in}}%
\pgfpathlineto{\pgfqpoint{0.730610in}{1.967979in}}%
\pgfpathlineto{\pgfqpoint{0.730610in}{1.884712in}}%
\pgfpathlineto{\pgfqpoint{0.730610in}{1.884712in}}%
\pgfpathclose%
\pgfusepath{stroke,fill}%
\end{pgfscope}%
\begin{pgfscope}%
\pgfpathrectangle{\pgfqpoint{0.504568in}{0.240679in}}{\pgfqpoint{1.937500in}{1.886500in}}%
\pgfusepath{clip}%
\pgfsetrectcap%
\pgfsetroundjoin%
\pgfsetlinewidth{1.003750pt}%
\definecolor{currentstroke}{rgb}{0.000000,0.000000,0.000000}%
\pgfsetstrokecolor{currentstroke}%
\pgfsetdash{}{0pt}%
\pgfpathmoveto{\pgfqpoint{0.827485in}{1.884712in}}%
\pgfpathlineto{\pgfqpoint{0.827485in}{1.819982in}}%
\pgfusepath{stroke}%
\end{pgfscope}%
\begin{pgfscope}%
\pgfpathrectangle{\pgfqpoint{0.504568in}{0.240679in}}{\pgfqpoint{1.937500in}{1.886500in}}%
\pgfusepath{clip}%
\pgfsetrectcap%
\pgfsetroundjoin%
\pgfsetlinewidth{1.003750pt}%
\definecolor{currentstroke}{rgb}{0.000000,0.000000,0.000000}%
\pgfsetstrokecolor{currentstroke}%
\pgfsetdash{}{0pt}%
\pgfpathmoveto{\pgfqpoint{0.827485in}{1.967979in}}%
\pgfpathlineto{\pgfqpoint{0.827485in}{2.045381in}}%
\pgfusepath{stroke}%
\end{pgfscope}%
\begin{pgfscope}%
\pgfpathrectangle{\pgfqpoint{0.504568in}{0.240679in}}{\pgfqpoint{1.937500in}{1.886500in}}%
\pgfusepath{clip}%
\pgfsetrectcap%
\pgfsetroundjoin%
\pgfsetlinewidth{1.003750pt}%
\definecolor{currentstroke}{rgb}{0.000000,0.000000,0.000000}%
\pgfsetstrokecolor{currentstroke}%
\pgfsetdash{}{0pt}%
\pgfpathmoveto{\pgfqpoint{0.779048in}{1.819982in}}%
\pgfpathlineto{\pgfqpoint{0.875923in}{1.819982in}}%
\pgfusepath{stroke}%
\end{pgfscope}%
\begin{pgfscope}%
\pgfpathrectangle{\pgfqpoint{0.504568in}{0.240679in}}{\pgfqpoint{1.937500in}{1.886500in}}%
\pgfusepath{clip}%
\pgfsetrectcap%
\pgfsetroundjoin%
\pgfsetlinewidth{1.003750pt}%
\definecolor{currentstroke}{rgb}{0.000000,0.000000,0.000000}%
\pgfsetstrokecolor{currentstroke}%
\pgfsetdash{}{0pt}%
\pgfpathmoveto{\pgfqpoint{0.779048in}{2.045381in}}%
\pgfpathlineto{\pgfqpoint{0.875923in}{2.045381in}}%
\pgfusepath{stroke}%
\end{pgfscope}%
\begin{pgfscope}%
\pgfpathrectangle{\pgfqpoint{0.504568in}{0.240679in}}{\pgfqpoint{1.937500in}{1.886500in}}%
\pgfusepath{clip}%
\pgfsetbuttcap%
\pgfsetroundjoin%
\definecolor{currentfill}{rgb}{0.000000,0.000000,0.000000}%
\pgfsetfillcolor{currentfill}%
\pgfsetfillopacity{0.000000}%
\pgfsetlinewidth{1.003750pt}%
\definecolor{currentstroke}{rgb}{0.000000,0.000000,0.000000}%
\pgfsetstrokecolor{currentstroke}%
\pgfsetdash{}{0pt}%
\pgfsys@defobject{currentmarker}{\pgfqpoint{-0.041667in}{-0.041667in}}{\pgfqpoint{0.041667in}{0.041667in}}{%
\pgfpathmoveto{\pgfqpoint{0.000000in}{-0.041667in}}%
\pgfpathcurveto{\pgfqpoint{0.011050in}{-0.041667in}}{\pgfqpoint{0.021649in}{-0.037276in}}{\pgfqpoint{0.029463in}{-0.029463in}}%
\pgfpathcurveto{\pgfqpoint{0.037276in}{-0.021649in}}{\pgfqpoint{0.041667in}{-0.011050in}}{\pgfqpoint{0.041667in}{0.000000in}}%
\pgfpathcurveto{\pgfqpoint{0.041667in}{0.011050in}}{\pgfqpoint{0.037276in}{0.021649in}}{\pgfqpoint{0.029463in}{0.029463in}}%
\pgfpathcurveto{\pgfqpoint{0.021649in}{0.037276in}}{\pgfqpoint{0.011050in}{0.041667in}}{\pgfqpoint{0.000000in}{0.041667in}}%
\pgfpathcurveto{\pgfqpoint{-0.011050in}{0.041667in}}{\pgfqpoint{-0.021649in}{0.037276in}}{\pgfqpoint{-0.029463in}{0.029463in}}%
\pgfpathcurveto{\pgfqpoint{-0.037276in}{0.021649in}}{\pgfqpoint{-0.041667in}{0.011050in}}{\pgfqpoint{-0.041667in}{0.000000in}}%
\pgfpathcurveto{\pgfqpoint{-0.041667in}{-0.011050in}}{\pgfqpoint{-0.037276in}{-0.021649in}}{\pgfqpoint{-0.029463in}{-0.029463in}}%
\pgfpathcurveto{\pgfqpoint{-0.021649in}{-0.037276in}}{\pgfqpoint{-0.011050in}{-0.041667in}}{\pgfqpoint{0.000000in}{-0.041667in}}%
\pgfpathlineto{\pgfqpoint{0.000000in}{-0.041667in}}%
\pgfpathclose%
\pgfusepath{stroke,fill}%
}%
\begin{pgfscope}%
\pgfsys@transformshift{0.827485in}{0.464460in}%
\pgfsys@useobject{currentmarker}{}%
\end{pgfscope}%
\end{pgfscope}%
\begin{pgfscope}%
\pgfpathrectangle{\pgfqpoint{0.504568in}{0.240679in}}{\pgfqpoint{1.937500in}{1.886500in}}%
\pgfusepath{clip}%
\pgfsetbuttcap%
\pgfsetmiterjoin%
\definecolor{currentfill}{rgb}{0.862745,0.149020,0.498039}%
\pgfsetfillcolor{currentfill}%
\pgfsetlinewidth{1.003750pt}%
\definecolor{currentstroke}{rgb}{0.000000,0.000000,0.000000}%
\pgfsetstrokecolor{currentstroke}%
\pgfsetdash{}{0pt}%
\pgfpathmoveto{\pgfqpoint{1.376443in}{1.836409in}}%
\pgfpathlineto{\pgfqpoint{1.570193in}{1.836409in}}%
\pgfpathlineto{\pgfqpoint{1.570193in}{1.933370in}}%
\pgfpathlineto{\pgfqpoint{1.376443in}{1.933370in}}%
\pgfpathlineto{\pgfqpoint{1.376443in}{1.836409in}}%
\pgfpathlineto{\pgfqpoint{1.376443in}{1.836409in}}%
\pgfpathclose%
\pgfusepath{stroke,fill}%
\end{pgfscope}%
\begin{pgfscope}%
\pgfpathrectangle{\pgfqpoint{0.504568in}{0.240679in}}{\pgfqpoint{1.937500in}{1.886500in}}%
\pgfusepath{clip}%
\pgfsetrectcap%
\pgfsetroundjoin%
\pgfsetlinewidth{1.003750pt}%
\definecolor{currentstroke}{rgb}{0.000000,0.000000,0.000000}%
\pgfsetstrokecolor{currentstroke}%
\pgfsetdash{}{0pt}%
\pgfpathmoveto{\pgfqpoint{1.473318in}{1.836409in}}%
\pgfpathlineto{\pgfqpoint{1.473318in}{1.761424in}}%
\pgfusepath{stroke}%
\end{pgfscope}%
\begin{pgfscope}%
\pgfpathrectangle{\pgfqpoint{0.504568in}{0.240679in}}{\pgfqpoint{1.937500in}{1.886500in}}%
\pgfusepath{clip}%
\pgfsetrectcap%
\pgfsetroundjoin%
\pgfsetlinewidth{1.003750pt}%
\definecolor{currentstroke}{rgb}{0.000000,0.000000,0.000000}%
\pgfsetstrokecolor{currentstroke}%
\pgfsetdash{}{0pt}%
\pgfpathmoveto{\pgfqpoint{1.473318in}{1.933370in}}%
\pgfpathlineto{\pgfqpoint{1.473318in}{2.044260in}}%
\pgfusepath{stroke}%
\end{pgfscope}%
\begin{pgfscope}%
\pgfpathrectangle{\pgfqpoint{0.504568in}{0.240679in}}{\pgfqpoint{1.937500in}{1.886500in}}%
\pgfusepath{clip}%
\pgfsetrectcap%
\pgfsetroundjoin%
\pgfsetlinewidth{1.003750pt}%
\definecolor{currentstroke}{rgb}{0.000000,0.000000,0.000000}%
\pgfsetstrokecolor{currentstroke}%
\pgfsetdash{}{0pt}%
\pgfpathmoveto{\pgfqpoint{1.424881in}{1.761424in}}%
\pgfpathlineto{\pgfqpoint{1.521756in}{1.761424in}}%
\pgfusepath{stroke}%
\end{pgfscope}%
\begin{pgfscope}%
\pgfpathrectangle{\pgfqpoint{0.504568in}{0.240679in}}{\pgfqpoint{1.937500in}{1.886500in}}%
\pgfusepath{clip}%
\pgfsetrectcap%
\pgfsetroundjoin%
\pgfsetlinewidth{1.003750pt}%
\definecolor{currentstroke}{rgb}{0.000000,0.000000,0.000000}%
\pgfsetstrokecolor{currentstroke}%
\pgfsetdash{}{0pt}%
\pgfpathmoveto{\pgfqpoint{1.424881in}{2.044260in}}%
\pgfpathlineto{\pgfqpoint{1.521756in}{2.044260in}}%
\pgfusepath{stroke}%
\end{pgfscope}%
\begin{pgfscope}%
\pgfpathrectangle{\pgfqpoint{0.504568in}{0.240679in}}{\pgfqpoint{1.937500in}{1.886500in}}%
\pgfusepath{clip}%
\pgfsetbuttcap%
\pgfsetroundjoin%
\definecolor{currentfill}{rgb}{0.000000,0.000000,0.000000}%
\pgfsetfillcolor{currentfill}%
\pgfsetfillopacity{0.000000}%
\pgfsetlinewidth{1.003750pt}%
\definecolor{currentstroke}{rgb}{0.000000,0.000000,0.000000}%
\pgfsetstrokecolor{currentstroke}%
\pgfsetdash{}{0pt}%
\pgfsys@defobject{currentmarker}{\pgfqpoint{-0.041667in}{-0.041667in}}{\pgfqpoint{0.041667in}{0.041667in}}{%
\pgfpathmoveto{\pgfqpoint{0.000000in}{-0.041667in}}%
\pgfpathcurveto{\pgfqpoint{0.011050in}{-0.041667in}}{\pgfqpoint{0.021649in}{-0.037276in}}{\pgfqpoint{0.029463in}{-0.029463in}}%
\pgfpathcurveto{\pgfqpoint{0.037276in}{-0.021649in}}{\pgfqpoint{0.041667in}{-0.011050in}}{\pgfqpoint{0.041667in}{0.000000in}}%
\pgfpathcurveto{\pgfqpoint{0.041667in}{0.011050in}}{\pgfqpoint{0.037276in}{0.021649in}}{\pgfqpoint{0.029463in}{0.029463in}}%
\pgfpathcurveto{\pgfqpoint{0.021649in}{0.037276in}}{\pgfqpoint{0.011050in}{0.041667in}}{\pgfqpoint{0.000000in}{0.041667in}}%
\pgfpathcurveto{\pgfqpoint{-0.011050in}{0.041667in}}{\pgfqpoint{-0.021649in}{0.037276in}}{\pgfqpoint{-0.029463in}{0.029463in}}%
\pgfpathcurveto{\pgfqpoint{-0.037276in}{0.021649in}}{\pgfqpoint{-0.041667in}{0.011050in}}{\pgfqpoint{-0.041667in}{0.000000in}}%
\pgfpathcurveto{\pgfqpoint{-0.041667in}{-0.011050in}}{\pgfqpoint{-0.037276in}{-0.021649in}}{\pgfqpoint{-0.029463in}{-0.029463in}}%
\pgfpathcurveto{\pgfqpoint{-0.021649in}{-0.037276in}}{\pgfqpoint{-0.011050in}{-0.041667in}}{\pgfqpoint{0.000000in}{-0.041667in}}%
\pgfpathlineto{\pgfqpoint{0.000000in}{-0.041667in}}%
\pgfpathclose%
\pgfusepath{stroke,fill}%
}%
\begin{pgfscope}%
\pgfsys@transformshift{1.473318in}{1.629012in}%
\pgfsys@useobject{currentmarker}{}%
\end{pgfscope}%
\begin{pgfscope}%
\pgfsys@transformshift{1.473318in}{1.046129in}%
\pgfsys@useobject{currentmarker}{}%
\end{pgfscope}%
\begin{pgfscope}%
\pgfsys@transformshift{1.473318in}{0.462490in}%
\pgfsys@useobject{currentmarker}{}%
\end{pgfscope}%
\end{pgfscope}%
\begin{pgfscope}%
\pgfpathrectangle{\pgfqpoint{0.504568in}{0.240679in}}{\pgfqpoint{1.937500in}{1.886500in}}%
\pgfusepath{clip}%
\pgfsetbuttcap%
\pgfsetmiterjoin%
\definecolor{currentfill}{rgb}{0.392157,0.560784,1.000000}%
\pgfsetfillcolor{currentfill}%
\pgfsetlinewidth{1.003750pt}%
\definecolor{currentstroke}{rgb}{0.000000,0.000000,0.000000}%
\pgfsetstrokecolor{currentstroke}%
\pgfsetdash{}{0pt}%
\pgfpathmoveto{\pgfqpoint{2.022277in}{1.885749in}}%
\pgfpathlineto{\pgfqpoint{2.216027in}{1.885749in}}%
\pgfpathlineto{\pgfqpoint{2.216027in}{1.967823in}}%
\pgfpathlineto{\pgfqpoint{2.022277in}{1.967823in}}%
\pgfpathlineto{\pgfqpoint{2.022277in}{1.885749in}}%
\pgfpathlineto{\pgfqpoint{2.022277in}{1.885749in}}%
\pgfpathclose%
\pgfusepath{stroke,fill}%
\end{pgfscope}%
\begin{pgfscope}%
\pgfpathrectangle{\pgfqpoint{0.504568in}{0.240679in}}{\pgfqpoint{1.937500in}{1.886500in}}%
\pgfusepath{clip}%
\pgfsetrectcap%
\pgfsetroundjoin%
\pgfsetlinewidth{1.003750pt}%
\definecolor{currentstroke}{rgb}{0.000000,0.000000,0.000000}%
\pgfsetstrokecolor{currentstroke}%
\pgfsetdash{}{0pt}%
\pgfpathmoveto{\pgfqpoint{2.119152in}{1.885749in}}%
\pgfpathlineto{\pgfqpoint{2.119152in}{1.788737in}}%
\pgfusepath{stroke}%
\end{pgfscope}%
\begin{pgfscope}%
\pgfpathrectangle{\pgfqpoint{0.504568in}{0.240679in}}{\pgfqpoint{1.937500in}{1.886500in}}%
\pgfusepath{clip}%
\pgfsetrectcap%
\pgfsetroundjoin%
\pgfsetlinewidth{1.003750pt}%
\definecolor{currentstroke}{rgb}{0.000000,0.000000,0.000000}%
\pgfsetstrokecolor{currentstroke}%
\pgfsetdash{}{0pt}%
\pgfpathmoveto{\pgfqpoint{2.119152in}{1.967823in}}%
\pgfpathlineto{\pgfqpoint{2.119152in}{2.017382in}}%
\pgfusepath{stroke}%
\end{pgfscope}%
\begin{pgfscope}%
\pgfpathrectangle{\pgfqpoint{0.504568in}{0.240679in}}{\pgfqpoint{1.937500in}{1.886500in}}%
\pgfusepath{clip}%
\pgfsetrectcap%
\pgfsetroundjoin%
\pgfsetlinewidth{1.003750pt}%
\definecolor{currentstroke}{rgb}{0.000000,0.000000,0.000000}%
\pgfsetstrokecolor{currentstroke}%
\pgfsetdash{}{0pt}%
\pgfpathmoveto{\pgfqpoint{2.070714in}{1.788737in}}%
\pgfpathlineto{\pgfqpoint{2.167589in}{1.788737in}}%
\pgfusepath{stroke}%
\end{pgfscope}%
\begin{pgfscope}%
\pgfpathrectangle{\pgfqpoint{0.504568in}{0.240679in}}{\pgfqpoint{1.937500in}{1.886500in}}%
\pgfusepath{clip}%
\pgfsetrectcap%
\pgfsetroundjoin%
\pgfsetlinewidth{1.003750pt}%
\definecolor{currentstroke}{rgb}{0.000000,0.000000,0.000000}%
\pgfsetstrokecolor{currentstroke}%
\pgfsetdash{}{0pt}%
\pgfpathmoveto{\pgfqpoint{2.070714in}{2.017382in}}%
\pgfpathlineto{\pgfqpoint{2.167589in}{2.017382in}}%
\pgfusepath{stroke}%
\end{pgfscope}%
\begin{pgfscope}%
\pgfpathrectangle{\pgfqpoint{0.504568in}{0.240679in}}{\pgfqpoint{1.937500in}{1.886500in}}%
\pgfusepath{clip}%
\pgfsetbuttcap%
\pgfsetroundjoin%
\definecolor{currentfill}{rgb}{0.000000,0.000000,0.000000}%
\pgfsetfillcolor{currentfill}%
\pgfsetfillopacity{0.000000}%
\pgfsetlinewidth{1.003750pt}%
\definecolor{currentstroke}{rgb}{0.000000,0.000000,0.000000}%
\pgfsetstrokecolor{currentstroke}%
\pgfsetdash{}{0pt}%
\pgfsys@defobject{currentmarker}{\pgfqpoint{-0.041667in}{-0.041667in}}{\pgfqpoint{0.041667in}{0.041667in}}{%
\pgfpathmoveto{\pgfqpoint{0.000000in}{-0.041667in}}%
\pgfpathcurveto{\pgfqpoint{0.011050in}{-0.041667in}}{\pgfqpoint{0.021649in}{-0.037276in}}{\pgfqpoint{0.029463in}{-0.029463in}}%
\pgfpathcurveto{\pgfqpoint{0.037276in}{-0.021649in}}{\pgfqpoint{0.041667in}{-0.011050in}}{\pgfqpoint{0.041667in}{0.000000in}}%
\pgfpathcurveto{\pgfqpoint{0.041667in}{0.011050in}}{\pgfqpoint{0.037276in}{0.021649in}}{\pgfqpoint{0.029463in}{0.029463in}}%
\pgfpathcurveto{\pgfqpoint{0.021649in}{0.037276in}}{\pgfqpoint{0.011050in}{0.041667in}}{\pgfqpoint{0.000000in}{0.041667in}}%
\pgfpathcurveto{\pgfqpoint{-0.011050in}{0.041667in}}{\pgfqpoint{-0.021649in}{0.037276in}}{\pgfqpoint{-0.029463in}{0.029463in}}%
\pgfpathcurveto{\pgfqpoint{-0.037276in}{0.021649in}}{\pgfqpoint{-0.041667in}{0.011050in}}{\pgfqpoint{-0.041667in}{0.000000in}}%
\pgfpathcurveto{\pgfqpoint{-0.041667in}{-0.011050in}}{\pgfqpoint{-0.037276in}{-0.021649in}}{\pgfqpoint{-0.029463in}{-0.029463in}}%
\pgfpathcurveto{\pgfqpoint{-0.021649in}{-0.037276in}}{\pgfqpoint{-0.011050in}{-0.041667in}}{\pgfqpoint{0.000000in}{-0.041667in}}%
\pgfpathlineto{\pgfqpoint{0.000000in}{-0.041667in}}%
\pgfpathclose%
\pgfusepath{stroke,fill}%
}%
\begin{pgfscope}%
\pgfsys@transformshift{2.119152in}{0.463729in}%
\pgfsys@useobject{currentmarker}{}%
\end{pgfscope}%
\end{pgfscope}%
\begin{pgfscope}%
\pgfpathrectangle{\pgfqpoint{0.504568in}{0.240679in}}{\pgfqpoint{1.937500in}{1.886500in}}%
\pgfusepath{clip}%
\pgfsetbuttcap%
\pgfsetroundjoin%
\pgfsetlinewidth{1.003750pt}%
\definecolor{currentstroke}{rgb}{0.000000,0.000000,0.000000}%
\pgfsetstrokecolor{currentstroke}%
\pgfsetdash{}{0pt}%
\pgfpathmoveto{\pgfqpoint{0.730610in}{1.922602in}}%
\pgfpathlineto{\pgfqpoint{0.924360in}{1.922602in}}%
\pgfusepath{stroke}%
\end{pgfscope}%
\begin{pgfscope}%
\pgfpathrectangle{\pgfqpoint{0.504568in}{0.240679in}}{\pgfqpoint{1.937500in}{1.886500in}}%
\pgfusepath{clip}%
\pgfsetbuttcap%
\pgfsetroundjoin%
\pgfsetlinewidth{1.003750pt}%
\definecolor{currentstroke}{rgb}{0.000000,0.000000,0.000000}%
\pgfsetstrokecolor{currentstroke}%
\pgfsetdash{}{0pt}%
\pgfpathmoveto{\pgfqpoint{1.376443in}{1.889997in}}%
\pgfpathlineto{\pgfqpoint{1.570193in}{1.889997in}}%
\pgfusepath{stroke}%
\end{pgfscope}%
\begin{pgfscope}%
\pgfpathrectangle{\pgfqpoint{0.504568in}{0.240679in}}{\pgfqpoint{1.937500in}{1.886500in}}%
\pgfusepath{clip}%
\pgfsetbuttcap%
\pgfsetroundjoin%
\pgfsetlinewidth{1.003750pt}%
\definecolor{currentstroke}{rgb}{0.000000,0.000000,0.000000}%
\pgfsetstrokecolor{currentstroke}%
\pgfsetdash{}{0pt}%
\pgfpathmoveto{\pgfqpoint{2.022277in}{1.921754in}}%
\pgfpathlineto{\pgfqpoint{2.216027in}{1.921754in}}%
\pgfusepath{stroke}%
\end{pgfscope}%
\begin{pgfscope}%
\pgfsetrectcap%
\pgfsetmiterjoin%
\pgfsetlinewidth{0.803000pt}%
\definecolor{currentstroke}{rgb}{0.000000,0.000000,0.000000}%
\pgfsetstrokecolor{currentstroke}%
\pgfsetdash{}{0pt}%
\pgfpathmoveto{\pgfqpoint{0.504568in}{0.240679in}}%
\pgfpathlineto{\pgfqpoint{0.504568in}{2.127179in}}%
\pgfusepath{stroke}%
\end{pgfscope}%
\begin{pgfscope}%
\pgfsetrectcap%
\pgfsetmiterjoin%
\pgfsetlinewidth{0.803000pt}%
\definecolor{currentstroke}{rgb}{0.000000,0.000000,0.000000}%
\pgfsetstrokecolor{currentstroke}%
\pgfsetdash{}{0pt}%
\pgfpathmoveto{\pgfqpoint{2.442068in}{0.240679in}}%
\pgfpathlineto{\pgfqpoint{2.442068in}{2.127179in}}%
\pgfusepath{stroke}%
\end{pgfscope}%
\begin{pgfscope}%
\pgfsetrectcap%
\pgfsetmiterjoin%
\pgfsetlinewidth{0.803000pt}%
\definecolor{currentstroke}{rgb}{0.000000,0.000000,0.000000}%
\pgfsetstrokecolor{currentstroke}%
\pgfsetdash{}{0pt}%
\pgfpathmoveto{\pgfqpoint{0.504568in}{0.240679in}}%
\pgfpathlineto{\pgfqpoint{2.442068in}{0.240679in}}%
\pgfusepath{stroke}%
\end{pgfscope}%
\begin{pgfscope}%
\pgfsetrectcap%
\pgfsetmiterjoin%
\pgfsetlinewidth{0.803000pt}%
\definecolor{currentstroke}{rgb}{0.000000,0.000000,0.000000}%
\pgfsetstrokecolor{currentstroke}%
\pgfsetdash{}{0pt}%
\pgfpathmoveto{\pgfqpoint{0.504568in}{2.127179in}}%
\pgfpathlineto{\pgfqpoint{2.442068in}{2.127179in}}%
\pgfusepath{stroke}%
\end{pgfscope}%
\end{pgfpicture}%
\makeatother%
\endgroup%

%% file: plots/comparison_large_times.pgf
\begingroup%
\makeatletter%
\begin{pgfpicture}%
\pgfpathrectangle{\pgfpointorigin}{\pgfqpoint{2.477501in}{2.170111in}}%
\pgfusepath{use as bounding box, clip}%
\begin{pgfscope}%
\pgfsetbuttcap%
\pgfsetmiterjoin%
\definecolor{currentfill}{rgb}{1.000000,1.000000,1.000000}%
\pgfsetfillcolor{currentfill}%
\pgfsetlinewidth{0.000000pt}%
\definecolor{currentstroke}{rgb}{1.000000,1.000000,1.000000}%
\pgfsetstrokecolor{currentstroke}%
\pgfsetdash{}{0pt}%
\pgfpathmoveto{\pgfqpoint{0.000000in}{0.000000in}}%
\pgfpathlineto{\pgfqpoint{2.477501in}{0.000000in}}%
\pgfpathlineto{\pgfqpoint{2.477501in}{2.170111in}}%
\pgfpathlineto{\pgfqpoint{0.000000in}{2.170111in}}%
\pgfpathlineto{\pgfqpoint{0.000000in}{0.000000in}}%
\pgfpathclose%
\pgfusepath{fill}%
\end{pgfscope}%
\begin{pgfscope}%
\pgfsetbuttcap%
\pgfsetmiterjoin%
\definecolor{currentfill}{rgb}{1.000000,1.000000,1.000000}%
\pgfsetfillcolor{currentfill}%
\pgfsetlinewidth{0.000000pt}%
\definecolor{currentstroke}{rgb}{0.000000,0.000000,0.000000}%
\pgfsetstrokecolor{currentstroke}%
\pgfsetstrokeopacity{0.000000}%
\pgfsetdash{}{0pt}%
\pgfpathmoveto{\pgfqpoint{0.520001in}{0.240679in}}%
\pgfpathlineto{\pgfqpoint{2.457501in}{0.240679in}}%
\pgfpathlineto{\pgfqpoint{2.457501in}{2.127179in}}%
\pgfpathlineto{\pgfqpoint{0.520001in}{2.127179in}}%
\pgfpathlineto{\pgfqpoint{0.520001in}{0.240679in}}%
\pgfpathclose%
\pgfusepath{fill}%
\end{pgfscope}%
\begin{pgfscope}%
\pgfpathrectangle{\pgfqpoint{0.520001in}{0.240679in}}{\pgfqpoint{1.937500in}{1.886500in}}%
\pgfusepath{clip}%
\pgfsetrectcap%
\pgfsetroundjoin%
\pgfsetlinewidth{0.250937pt}%
\definecolor{currentstroke}{rgb}{0.000000,0.000000,0.000000}%
\pgfsetstrokecolor{currentstroke}%
\pgfsetstrokeopacity{0.200000}%
\pgfsetdash{}{0pt}%
\pgfpathmoveto{\pgfqpoint{0.842917in}{0.240679in}}%
\pgfpathlineto{\pgfqpoint{0.842917in}{2.127179in}}%
\pgfusepath{stroke}%
\end{pgfscope}%
\begin{pgfscope}%
\pgfsetbuttcap%
\pgfsetroundjoin%
\definecolor{currentfill}{rgb}{0.000000,0.000000,0.000000}%
\pgfsetfillcolor{currentfill}%
\pgfsetlinewidth{0.803000pt}%
\definecolor{currentstroke}{rgb}{0.000000,0.000000,0.000000}%
\pgfsetstrokecolor{currentstroke}%
\pgfsetdash{}{0pt}%
\pgfsys@defobject{currentmarker}{\pgfqpoint{0.000000in}{-0.048611in}}{\pgfqpoint{0.000000in}{0.000000in}}{%
\pgfpathmoveto{\pgfqpoint{0.000000in}{0.000000in}}%
\pgfpathlineto{\pgfqpoint{0.000000in}{-0.048611in}}%
\pgfusepath{stroke,fill}%
}%
\begin{pgfscope}%
\pgfsys@transformshift{0.842917in}{0.240679in}%
\pgfsys@useobject{currentmarker}{}%
\end{pgfscope}%
\end{pgfscope}%
\begin{pgfscope}%
\definecolor{textcolor}{rgb}{0.000000,0.000000,0.000000}%
\pgfsetstrokecolor{textcolor}%
\pgfsetfillcolor{textcolor}%
\pgftext[x=0.842917in,y=0.143457in,,top]{\color{textcolor}{\rmfamily\fontsize{10.000000}{12.000000}\selectfont\catcode`\^=\active\def^{\ifmmode\sp\else\^{}\fi}\catcode`\%=\active\def
\end{pgfscope}%
\begin{pgfscope}%
\pgfpathrectangle{\pgfqpoint{0.520001in}{0.240679in}}{\pgfqpoint{1.937500in}{1.886500in}}%
\pgfusepath{clip}%
\pgfsetrectcap%
\pgfsetroundjoin%
\pgfsetlinewidth{0.250937pt}%
\definecolor{currentstroke}{rgb}{0.000000,0.000000,0.000000}%
\pgfsetstrokecolor{currentstroke}%
\pgfsetstrokeopacity{0.200000}%
\pgfsetdash{}{0pt}%
\pgfpathmoveto{\pgfqpoint{1.488751in}{0.240679in}}%
\pgfpathlineto{\pgfqpoint{1.488751in}{2.127179in}}%
\pgfusepath{stroke}%
\end{pgfscope}%
\begin{pgfscope}%
\pgfsetbuttcap%
\pgfsetroundjoin%
\definecolor{currentfill}{rgb}{0.000000,0.000000,0.000000}%
\pgfsetfillcolor{currentfill}%
\pgfsetlinewidth{0.803000pt}%
\definecolor{currentstroke}{rgb}{0.000000,0.000000,0.000000}%
\pgfsetstrokecolor{currentstroke}%
\pgfsetdash{}{0pt}%
\pgfsys@defobject{currentmarker}{\pgfqpoint{0.000000in}{-0.048611in}}{\pgfqpoint{0.000000in}{0.000000in}}{%
\pgfpathmoveto{\pgfqpoint{0.000000in}{0.000000in}}%
\pgfpathlineto{\pgfqpoint{0.000000in}{-0.048611in}}%
\pgfusepath{stroke,fill}%
}%
\begin{pgfscope}%
\pgfsys@transformshift{1.488751in}{0.240679in}%
\pgfsys@useobject{currentmarker}{}%
\end{pgfscope}%
\end{pgfscope}%
\begin{pgfscope}%
\definecolor{textcolor}{rgb}{0.000000,0.000000,0.000000}%
\pgfsetstrokecolor{textcolor}%
\pgfsetfillcolor{textcolor}%
\pgftext[x=1.488751in,y=0.143457in,,top]{\color{textcolor}{\rmfamily\fontsize{10.000000}{12.000000}\selectfont\catcode`\^=\active\def^{\ifmmode\sp\else\^{}\fi}\catcode`\%=\active\def
\end{pgfscope}%
\begin{pgfscope}%
\pgfpathrectangle{\pgfqpoint{0.520001in}{0.240679in}}{\pgfqpoint{1.937500in}{1.886500in}}%
\pgfusepath{clip}%
\pgfsetrectcap%
\pgfsetroundjoin%
\pgfsetlinewidth{0.250937pt}%
\definecolor{currentstroke}{rgb}{0.000000,0.000000,0.000000}%
\pgfsetstrokecolor{currentstroke}%
\pgfsetstrokeopacity{0.200000}%
\pgfsetdash{}{0pt}%
\pgfpathmoveto{\pgfqpoint{2.134584in}{0.240679in}}%
\pgfpathlineto{\pgfqpoint{2.134584in}{2.127179in}}%
\pgfusepath{stroke}%
\end{pgfscope}%
\begin{pgfscope}%
\pgfsetbuttcap%
\pgfsetroundjoin%
\definecolor{currentfill}{rgb}{0.000000,0.000000,0.000000}%
\pgfsetfillcolor{currentfill}%
\pgfsetlinewidth{0.803000pt}%
\definecolor{currentstroke}{rgb}{0.000000,0.000000,0.000000}%
\pgfsetstrokecolor{currentstroke}%
\pgfsetdash{}{0pt}%
\pgfsys@defobject{currentmarker}{\pgfqpoint{0.000000in}{-0.048611in}}{\pgfqpoint{0.000000in}{0.000000in}}{%
\pgfpathmoveto{\pgfqpoint{0.000000in}{0.000000in}}%
\pgfpathlineto{\pgfqpoint{0.000000in}{-0.048611in}}%
\pgfusepath{stroke,fill}%
}%
\begin{pgfscope}%
\pgfsys@transformshift{2.134584in}{0.240679in}%
\pgfsys@useobject{currentmarker}{}%
\end{pgfscope}%
\end{pgfscope}%
\begin{pgfscope}%
\definecolor{textcolor}{rgb}{0.000000,0.000000,0.000000}%
\pgfsetstrokecolor{textcolor}%
\pgfsetfillcolor{textcolor}%
\pgftext[x=2.134584in,y=0.143457in,,top]{\color{textcolor}{\rmfamily\fontsize{10.000000}{12.000000}\selectfont\catcode`\^=\active\def^{\ifmmode\sp\else\^{}\fi}\catcode`\%=\active\def
\end{pgfscope}%
\begin{pgfscope}%
\pgfpathrectangle{\pgfqpoint{0.520001in}{0.240679in}}{\pgfqpoint{1.937500in}{1.886500in}}%
\pgfusepath{clip}%
\pgfsetrectcap%
\pgfsetroundjoin%
\pgfsetlinewidth{0.250937pt}%
\definecolor{currentstroke}{rgb}{0.000000,0.000000,0.000000}%
\pgfsetstrokecolor{currentstroke}%
\pgfsetstrokeopacity{0.200000}%
\pgfsetdash{}{0pt}%
\pgfpathmoveto{\pgfqpoint{0.520001in}{0.605991in}}%
\pgfpathlineto{\pgfqpoint{2.457501in}{0.605991in}}%
\pgfusepath{stroke}%
\end{pgfscope}%
\begin{pgfscope}%
\pgfsetbuttcap%
\pgfsetroundjoin%
\definecolor{currentfill}{rgb}{0.000000,0.000000,0.000000}%
\pgfsetfillcolor{currentfill}%
\pgfsetlinewidth{0.803000pt}%
\definecolor{currentstroke}{rgb}{0.000000,0.000000,0.000000}%
\pgfsetstrokecolor{currentstroke}%
\pgfsetdash{}{0pt}%
\pgfsys@defobject{currentmarker}{\pgfqpoint{-0.048611in}{0.000000in}}{\pgfqpoint{-0.000000in}{0.000000in}}{%
\pgfpathmoveto{\pgfqpoint{-0.000000in}{0.000000in}}%
\pgfpathlineto{\pgfqpoint{-0.048611in}{0.000000in}}%
\pgfusepath{stroke,fill}%
}%
\begin{pgfscope}%
\pgfsys@transformshift{0.520001in}{0.605991in}%
\pgfsys@useobject{currentmarker}{}%
\end{pgfscope}%
\end{pgfscope}%
\begin{pgfscope}%
\definecolor{textcolor}{rgb}{0.000000,0.000000,0.000000}%
\pgfsetstrokecolor{textcolor}%
\pgfsetfillcolor{textcolor}%
\pgftext[x=0.283889in, y=0.557766in, left, base]{\color{textcolor}{\rmfamily\fontsize{10.000000}{12.000000}\selectfont\catcode`\^=\active\def^{\ifmmode\sp\else\^{}\fi}\catcode`\%=\active\def
\end{pgfscope}%
\begin{pgfscope}%
\pgfpathrectangle{\pgfqpoint{0.520001in}{0.240679in}}{\pgfqpoint{1.937500in}{1.886500in}}%
\pgfusepath{clip}%
\pgfsetrectcap%
\pgfsetroundjoin%
\pgfsetlinewidth{0.250937pt}%
\definecolor{currentstroke}{rgb}{0.000000,0.000000,0.000000}%
\pgfsetstrokecolor{currentstroke}%
\pgfsetstrokeopacity{0.200000}%
\pgfsetdash{}{0pt}%
\pgfpathmoveto{\pgfqpoint{0.520001in}{0.979965in}}%
\pgfpathlineto{\pgfqpoint{2.457501in}{0.979965in}}%
\pgfusepath{stroke}%
\end{pgfscope}%
\begin{pgfscope}%
\pgfsetbuttcap%
\pgfsetroundjoin%
\definecolor{currentfill}{rgb}{0.000000,0.000000,0.000000}%
\pgfsetfillcolor{currentfill}%
\pgfsetlinewidth{0.803000pt}%
\definecolor{currentstroke}{rgb}{0.000000,0.000000,0.000000}%
\pgfsetstrokecolor{currentstroke}%
\pgfsetdash{}{0pt}%
\pgfsys@defobject{currentmarker}{\pgfqpoint{-0.048611in}{0.000000in}}{\pgfqpoint{-0.000000in}{0.000000in}}{%
\pgfpathmoveto{\pgfqpoint{-0.000000in}{0.000000in}}%
\pgfpathlineto{\pgfqpoint{-0.048611in}{0.000000in}}%
\pgfusepath{stroke,fill}%
}%
\begin{pgfscope}%
\pgfsys@transformshift{0.520001in}{0.979965in}%
\pgfsys@useobject{currentmarker}{}%
\end{pgfscope}%
\end{pgfscope}%
\begin{pgfscope}%
\definecolor{textcolor}{rgb}{0.000000,0.000000,0.000000}%
\pgfsetstrokecolor{textcolor}%
\pgfsetfillcolor{textcolor}%
\pgftext[x=0.283889in, y=0.931740in, left, base]{\color{textcolor}{\rmfamily\fontsize{10.000000}{12.000000}\selectfont\catcode`\^=\active\def^{\ifmmode\sp\else\^{}\fi}\catcode`\%=\active\def
\end{pgfscope}%
\begin{pgfscope}%
\pgfpathrectangle{\pgfqpoint{0.520001in}{0.240679in}}{\pgfqpoint{1.937500in}{1.886500in}}%
\pgfusepath{clip}%
\pgfsetrectcap%
\pgfsetroundjoin%
\pgfsetlinewidth{0.250937pt}%
\definecolor{currentstroke}{rgb}{0.000000,0.000000,0.000000}%
\pgfsetstrokecolor{currentstroke}%
\pgfsetstrokeopacity{0.200000}%
\pgfsetdash{}{0pt}%
\pgfpathmoveto{\pgfqpoint{0.520001in}{1.353939in}}%
\pgfpathlineto{\pgfqpoint{2.457501in}{1.353939in}}%
\pgfusepath{stroke}%
\end{pgfscope}%
\begin{pgfscope}%
\pgfsetbuttcap%
\pgfsetroundjoin%
\definecolor{currentfill}{rgb}{0.000000,0.000000,0.000000}%
\pgfsetfillcolor{currentfill}%
\pgfsetlinewidth{0.803000pt}%
\definecolor{currentstroke}{rgb}{0.000000,0.000000,0.000000}%
\pgfsetstrokecolor{currentstroke}%
\pgfsetdash{}{0pt}%
\pgfsys@defobject{currentmarker}{\pgfqpoint{-0.048611in}{0.000000in}}{\pgfqpoint{-0.000000in}{0.000000in}}{%
\pgfpathmoveto{\pgfqpoint{-0.000000in}{0.000000in}}%
\pgfpathlineto{\pgfqpoint{-0.048611in}{0.000000in}}%
\pgfusepath{stroke,fill}%
}%
\begin{pgfscope}%
\pgfsys@transformshift{0.520001in}{1.353939in}%
\pgfsys@useobject{currentmarker}{}%
\end{pgfscope}%
\end{pgfscope}%
\begin{pgfscope}%
\definecolor{textcolor}{rgb}{0.000000,0.000000,0.000000}%
\pgfsetstrokecolor{textcolor}%
\pgfsetfillcolor{textcolor}%
\pgftext[x=0.283889in, y=1.305713in, left, base]{\color{textcolor}{\rmfamily\fontsize{10.000000}{12.000000}\selectfont\catcode`\^=\active\def^{\ifmmode\sp\else\^{}\fi}\catcode`\%=\active\def
\end{pgfscope}%
\begin{pgfscope}%
\pgfpathrectangle{\pgfqpoint{0.520001in}{0.240679in}}{\pgfqpoint{1.937500in}{1.886500in}}%
\pgfusepath{clip}%
\pgfsetrectcap%
\pgfsetroundjoin%
\pgfsetlinewidth{0.250937pt}%
\definecolor{currentstroke}{rgb}{0.000000,0.000000,0.000000}%
\pgfsetstrokecolor{currentstroke}%
\pgfsetstrokeopacity{0.200000}%
\pgfsetdash{}{0pt}%
\pgfpathmoveto{\pgfqpoint{0.520001in}{1.727912in}}%
\pgfpathlineto{\pgfqpoint{2.457501in}{1.727912in}}%
\pgfusepath{stroke}%
\end{pgfscope}%
\begin{pgfscope}%
\pgfsetbuttcap%
\pgfsetroundjoin%
\definecolor{currentfill}{rgb}{0.000000,0.000000,0.000000}%
\pgfsetfillcolor{currentfill}%
\pgfsetlinewidth{0.803000pt}%
\definecolor{currentstroke}{rgb}{0.000000,0.000000,0.000000}%
\pgfsetstrokecolor{currentstroke}%
\pgfsetdash{}{0pt}%
\pgfsys@defobject{currentmarker}{\pgfqpoint{-0.048611in}{0.000000in}}{\pgfqpoint{-0.000000in}{0.000000in}}{%
\pgfpathmoveto{\pgfqpoint{-0.000000in}{0.000000in}}%
\pgfpathlineto{\pgfqpoint{-0.048611in}{0.000000in}}%
\pgfusepath{stroke,fill}%
}%
\begin{pgfscope}%
\pgfsys@transformshift{0.520001in}{1.727912in}%
\pgfsys@useobject{currentmarker}{}%
\end{pgfscope}%
\end{pgfscope}%
\begin{pgfscope}%
\definecolor{textcolor}{rgb}{0.000000,0.000000,0.000000}%
\pgfsetstrokecolor{textcolor}%
\pgfsetfillcolor{textcolor}%
\pgftext[x=0.214444in, y=1.679687in, left, base]{\color{textcolor}{\rmfamily\fontsize{10.000000}{12.000000}\selectfont\catcode`\^=\active\def^{\ifmmode\sp\else\^{}\fi}\catcode`\%=\active\def
\end{pgfscope}%
\begin{pgfscope}%
\pgfpathrectangle{\pgfqpoint{0.520001in}{0.240679in}}{\pgfqpoint{1.937500in}{1.886500in}}%
\pgfusepath{clip}%
\pgfsetrectcap%
\pgfsetroundjoin%
\pgfsetlinewidth{0.250937pt}%
\definecolor{currentstroke}{rgb}{0.000000,0.000000,0.000000}%
\pgfsetstrokecolor{currentstroke}%
\pgfsetstrokeopacity{0.200000}%
\pgfsetdash{}{0pt}%
\pgfpathmoveto{\pgfqpoint{0.520001in}{2.101886in}}%
\pgfpathlineto{\pgfqpoint{2.457501in}{2.101886in}}%
\pgfusepath{stroke}%
\end{pgfscope}%
\begin{pgfscope}%
\pgfsetbuttcap%
\pgfsetroundjoin%
\definecolor{currentfill}{rgb}{0.000000,0.000000,0.000000}%
\pgfsetfillcolor{currentfill}%
\pgfsetlinewidth{0.803000pt}%
\definecolor{currentstroke}{rgb}{0.000000,0.000000,0.000000}%
\pgfsetstrokecolor{currentstroke}%
\pgfsetdash{}{0pt}%
\pgfsys@defobject{currentmarker}{\pgfqpoint{-0.048611in}{0.000000in}}{\pgfqpoint{-0.000000in}{0.000000in}}{%
\pgfpathmoveto{\pgfqpoint{-0.000000in}{0.000000in}}%
\pgfpathlineto{\pgfqpoint{-0.048611in}{0.000000in}}%
\pgfusepath{stroke,fill}%
}%
\begin{pgfscope}%
\pgfsys@transformshift{0.520001in}{2.101886in}%
\pgfsys@useobject{currentmarker}{}%
\end{pgfscope}%
\end{pgfscope}%
\begin{pgfscope}%
\definecolor{textcolor}{rgb}{0.000000,0.000000,0.000000}%
\pgfsetstrokecolor{textcolor}%
\pgfsetfillcolor{textcolor}%
\pgftext[x=0.214444in, y=2.053661in, left, base]{\color{textcolor}{\rmfamily\fontsize{10.000000}{12.000000}\selectfont\catcode`\^=\active\def^{\ifmmode\sp\else\^{}\fi}\catcode`\%=\active\def
\end{pgfscope}%
\begin{pgfscope}%
\definecolor{textcolor}{rgb}{0.000000,0.000000,0.000000}%
\pgfsetstrokecolor{textcolor}%
\pgfsetfillcolor{textcolor}%
\pgftext[x=0.158889in,y=1.183929in,,bottom,rotate=90.000000]{\color{textcolor}{\rmfamily\fontsize{10.000000}{12.000000}\selectfont\catcode`\^=\active\def^{\ifmmode\sp\else\^{}\fi}\catcode`\%=\active\def
\end{pgfscope}%
\begin{pgfscope}%
\pgfpathrectangle{\pgfqpoint{0.520001in}{0.240679in}}{\pgfqpoint{1.937500in}{1.886500in}}%
\pgfusepath{clip}%
\pgfsetbuttcap%
\pgfsetmiterjoin%
\definecolor{currentfill}{rgb}{1.000000,0.690196,0.000000}%
\pgfsetfillcolor{currentfill}%
\pgfsetlinewidth{1.003750pt}%
\definecolor{currentstroke}{rgb}{0.000000,0.000000,0.000000}%
\pgfsetstrokecolor{currentstroke}%
\pgfsetdash{}{0pt}%
\pgfpathmoveto{\pgfqpoint{0.746042in}{1.539346in}}%
\pgfpathlineto{\pgfqpoint{0.939792in}{1.539346in}}%
\pgfpathlineto{\pgfqpoint{0.939792in}{1.571821in}}%
\pgfpathlineto{\pgfqpoint{0.746042in}{1.571821in}}%
\pgfpathlineto{\pgfqpoint{0.746042in}{1.539346in}}%
\pgfpathlineto{\pgfqpoint{0.746042in}{1.539346in}}%
\pgfpathclose%
\pgfusepath{stroke,fill}%
\end{pgfscope}%
\begin{pgfscope}%
\pgfpathrectangle{\pgfqpoint{0.520001in}{0.240679in}}{\pgfqpoint{1.937500in}{1.886500in}}%
\pgfusepath{clip}%
\pgfsetrectcap%
\pgfsetroundjoin%
\pgfsetlinewidth{1.003750pt}%
\definecolor{currentstroke}{rgb}{0.000000,0.000000,0.000000}%
\pgfsetstrokecolor{currentstroke}%
\pgfsetdash{}{0pt}%
\pgfpathmoveto{\pgfqpoint{0.842917in}{1.539346in}}%
\pgfpathlineto{\pgfqpoint{0.842917in}{1.492051in}}%
\pgfusepath{stroke}%
\end{pgfscope}%
\begin{pgfscope}%
\pgfpathrectangle{\pgfqpoint{0.520001in}{0.240679in}}{\pgfqpoint{1.937500in}{1.886500in}}%
\pgfusepath{clip}%
\pgfsetrectcap%
\pgfsetroundjoin%
\pgfsetlinewidth{1.003750pt}%
\definecolor{currentstroke}{rgb}{0.000000,0.000000,0.000000}%
\pgfsetstrokecolor{currentstroke}%
\pgfsetdash{}{0pt}%
\pgfpathmoveto{\pgfqpoint{0.842917in}{1.571821in}}%
\pgfpathlineto{\pgfqpoint{0.842917in}{1.612817in}}%
\pgfusepath{stroke}%
\end{pgfscope}%
\begin{pgfscope}%
\pgfpathrectangle{\pgfqpoint{0.520001in}{0.240679in}}{\pgfqpoint{1.937500in}{1.886500in}}%
\pgfusepath{clip}%
\pgfsetrectcap%
\pgfsetroundjoin%
\pgfsetlinewidth{1.003750pt}%
\definecolor{currentstroke}{rgb}{0.000000,0.000000,0.000000}%
\pgfsetstrokecolor{currentstroke}%
\pgfsetdash{}{0pt}%
\pgfpathmoveto{\pgfqpoint{0.794480in}{1.492051in}}%
\pgfpathlineto{\pgfqpoint{0.891355in}{1.492051in}}%
\pgfusepath{stroke}%
\end{pgfscope}%
\begin{pgfscope}%
\pgfpathrectangle{\pgfqpoint{0.520001in}{0.240679in}}{\pgfqpoint{1.937500in}{1.886500in}}%
\pgfusepath{clip}%
\pgfsetrectcap%
\pgfsetroundjoin%
\pgfsetlinewidth{1.003750pt}%
\definecolor{currentstroke}{rgb}{0.000000,0.000000,0.000000}%
\pgfsetstrokecolor{currentstroke}%
\pgfsetdash{}{0pt}%
\pgfpathmoveto{\pgfqpoint{0.794480in}{1.612817in}}%
\pgfpathlineto{\pgfqpoint{0.891355in}{1.612817in}}%
\pgfusepath{stroke}%
\end{pgfscope}%
\begin{pgfscope}%
\pgfpathrectangle{\pgfqpoint{0.520001in}{0.240679in}}{\pgfqpoint{1.937500in}{1.886500in}}%
\pgfusepath{clip}%
\pgfsetbuttcap%
\pgfsetroundjoin%
\definecolor{currentfill}{rgb}{0.000000,0.000000,0.000000}%
\pgfsetfillcolor{currentfill}%
\pgfsetfillopacity{0.000000}%
\pgfsetlinewidth{1.003750pt}%
\definecolor{currentstroke}{rgb}{0.000000,0.000000,0.000000}%
\pgfsetstrokecolor{currentstroke}%
\pgfsetdash{}{0pt}%
\pgfsys@defobject{currentmarker}{\pgfqpoint{-0.041667in}{-0.041667in}}{\pgfqpoint{0.041667in}{0.041667in}}{%
\pgfpathmoveto{\pgfqpoint{0.000000in}{-0.041667in}}%
\pgfpathcurveto{\pgfqpoint{0.011050in}{-0.041667in}}{\pgfqpoint{0.021649in}{-0.037276in}}{\pgfqpoint{0.029463in}{-0.029463in}}%
\pgfpathcurveto{\pgfqpoint{0.037276in}{-0.021649in}}{\pgfqpoint{0.041667in}{-0.011050in}}{\pgfqpoint{0.041667in}{0.000000in}}%
\pgfpathcurveto{\pgfqpoint{0.041667in}{0.011050in}}{\pgfqpoint{0.037276in}{0.021649in}}{\pgfqpoint{0.029463in}{0.029463in}}%
\pgfpathcurveto{\pgfqpoint{0.021649in}{0.037276in}}{\pgfqpoint{0.011050in}{0.041667in}}{\pgfqpoint{0.000000in}{0.041667in}}%
\pgfpathcurveto{\pgfqpoint{-0.011050in}{0.041667in}}{\pgfqpoint{-0.021649in}{0.037276in}}{\pgfqpoint{-0.029463in}{0.029463in}}%
\pgfpathcurveto{\pgfqpoint{-0.037276in}{0.021649in}}{\pgfqpoint{-0.041667in}{0.011050in}}{\pgfqpoint{-0.041667in}{0.000000in}}%
\pgfpathcurveto{\pgfqpoint{-0.041667in}{-0.011050in}}{\pgfqpoint{-0.037276in}{-0.021649in}}{\pgfqpoint{-0.029463in}{-0.029463in}}%
\pgfpathcurveto{\pgfqpoint{-0.021649in}{-0.037276in}}{\pgfqpoint{-0.011050in}{-0.041667in}}{\pgfqpoint{0.000000in}{-0.041667in}}%
\pgfpathlineto{\pgfqpoint{0.000000in}{-0.041667in}}%
\pgfpathclose%
\pgfusepath{stroke,fill}%
}%
\begin{pgfscope}%
\pgfsys@transformshift{0.842917in}{1.367187in}%
\pgfsys@useobject{currentmarker}{}%
\end{pgfscope}%
\begin{pgfscope}%
\pgfsys@transformshift{0.842917in}{1.969971in}%
\pgfsys@useobject{currentmarker}{}%
\end{pgfscope}%
\end{pgfscope}%
\begin{pgfscope}%
\pgfpathrectangle{\pgfqpoint{0.520001in}{0.240679in}}{\pgfqpoint{1.937500in}{1.886500in}}%
\pgfusepath{clip}%
\pgfsetbuttcap%
\pgfsetmiterjoin%
\definecolor{currentfill}{rgb}{0.862745,0.149020,0.498039}%
\pgfsetfillcolor{currentfill}%
\pgfsetlinewidth{1.003750pt}%
\definecolor{currentstroke}{rgb}{0.000000,0.000000,0.000000}%
\pgfsetstrokecolor{currentstroke}%
\pgfsetdash{}{0pt}%
\pgfpathmoveto{\pgfqpoint{1.391876in}{0.476872in}}%
\pgfpathlineto{\pgfqpoint{1.585626in}{0.476872in}}%
\pgfpathlineto{\pgfqpoint{1.585626in}{0.571703in}}%
\pgfpathlineto{\pgfqpoint{1.391876in}{0.571703in}}%
\pgfpathlineto{\pgfqpoint{1.391876in}{0.476872in}}%
\pgfpathlineto{\pgfqpoint{1.391876in}{0.476872in}}%
\pgfpathclose%
\pgfusepath{stroke,fill}%
\end{pgfscope}%
\begin{pgfscope}%
\pgfpathrectangle{\pgfqpoint{0.520001in}{0.240679in}}{\pgfqpoint{1.937500in}{1.886500in}}%
\pgfusepath{clip}%
\pgfsetrectcap%
\pgfsetroundjoin%
\pgfsetlinewidth{1.003750pt}%
\definecolor{currentstroke}{rgb}{0.000000,0.000000,0.000000}%
\pgfsetstrokecolor{currentstroke}%
\pgfsetdash{}{0pt}%
\pgfpathmoveto{\pgfqpoint{1.488751in}{0.476872in}}%
\pgfpathlineto{\pgfqpoint{1.488751in}{0.404903in}}%
\pgfusepath{stroke}%
\end{pgfscope}%
\begin{pgfscope}%
\pgfpathrectangle{\pgfqpoint{0.520001in}{0.240679in}}{\pgfqpoint{1.937500in}{1.886500in}}%
\pgfusepath{clip}%
\pgfsetrectcap%
\pgfsetroundjoin%
\pgfsetlinewidth{1.003750pt}%
\definecolor{currentstroke}{rgb}{0.000000,0.000000,0.000000}%
\pgfsetstrokecolor{currentstroke}%
\pgfsetdash{}{0pt}%
\pgfpathmoveto{\pgfqpoint{1.488751in}{0.571703in}}%
\pgfpathlineto{\pgfqpoint{1.488751in}{0.630522in}}%
\pgfusepath{stroke}%
\end{pgfscope}%
\begin{pgfscope}%
\pgfpathrectangle{\pgfqpoint{0.520001in}{0.240679in}}{\pgfqpoint{1.937500in}{1.886500in}}%
\pgfusepath{clip}%
\pgfsetrectcap%
\pgfsetroundjoin%
\pgfsetlinewidth{1.003750pt}%
\definecolor{currentstroke}{rgb}{0.000000,0.000000,0.000000}%
\pgfsetstrokecolor{currentstroke}%
\pgfsetdash{}{0pt}%
\pgfpathmoveto{\pgfqpoint{1.440313in}{0.404903in}}%
\pgfpathlineto{\pgfqpoint{1.537188in}{0.404903in}}%
\pgfusepath{stroke}%
\end{pgfscope}%
\begin{pgfscope}%
\pgfpathrectangle{\pgfqpoint{0.520001in}{0.240679in}}{\pgfqpoint{1.937500in}{1.886500in}}%
\pgfusepath{clip}%
\pgfsetrectcap%
\pgfsetroundjoin%
\pgfsetlinewidth{1.003750pt}%
\definecolor{currentstroke}{rgb}{0.000000,0.000000,0.000000}%
\pgfsetstrokecolor{currentstroke}%
\pgfsetdash{}{0pt}%
\pgfpathmoveto{\pgfqpoint{1.440313in}{0.630522in}}%
\pgfpathlineto{\pgfqpoint{1.537188in}{0.630522in}}%
\pgfusepath{stroke}%
\end{pgfscope}%
\begin{pgfscope}%
\pgfpathrectangle{\pgfqpoint{0.520001in}{0.240679in}}{\pgfqpoint{1.937500in}{1.886500in}}%
\pgfusepath{clip}%
\pgfsetbuttcap%
\pgfsetmiterjoin%
\definecolor{currentfill}{rgb}{0.392157,0.560784,1.000000}%
\pgfsetfillcolor{currentfill}%
\pgfsetlinewidth{1.003750pt}%
\definecolor{currentstroke}{rgb}{0.000000,0.000000,0.000000}%
\pgfsetstrokecolor{currentstroke}%
\pgfsetdash{}{0pt}%
\pgfpathmoveto{\pgfqpoint{2.037709in}{0.400268in}}%
\pgfpathlineto{\pgfqpoint{2.231459in}{0.400268in}}%
\pgfpathlineto{\pgfqpoint{2.231459in}{0.409010in}}%
\pgfpathlineto{\pgfqpoint{2.037709in}{0.409010in}}%
\pgfpathlineto{\pgfqpoint{2.037709in}{0.400268in}}%
\pgfpathlineto{\pgfqpoint{2.037709in}{0.400268in}}%
\pgfpathclose%
\pgfusepath{stroke,fill}%
\end{pgfscope}%
\begin{pgfscope}%
\pgfpathrectangle{\pgfqpoint{0.520001in}{0.240679in}}{\pgfqpoint{1.937500in}{1.886500in}}%
\pgfusepath{clip}%
\pgfsetrectcap%
\pgfsetroundjoin%
\pgfsetlinewidth{1.003750pt}%
\definecolor{currentstroke}{rgb}{0.000000,0.000000,0.000000}%
\pgfsetstrokecolor{currentstroke}%
\pgfsetdash{}{0pt}%
\pgfpathmoveto{\pgfqpoint{2.134584in}{0.400268in}}%
\pgfpathlineto{\pgfqpoint{2.134584in}{0.397887in}}%
\pgfusepath{stroke}%
\end{pgfscope}%
\begin{pgfscope}%
\pgfpathrectangle{\pgfqpoint{0.520001in}{0.240679in}}{\pgfqpoint{1.937500in}{1.886500in}}%
\pgfusepath{clip}%
\pgfsetrectcap%
\pgfsetroundjoin%
\pgfsetlinewidth{1.003750pt}%
\definecolor{currentstroke}{rgb}{0.000000,0.000000,0.000000}%
\pgfsetstrokecolor{currentstroke}%
\pgfsetdash{}{0pt}%
\pgfpathmoveto{\pgfqpoint{2.134584in}{0.409010in}}%
\pgfpathlineto{\pgfqpoint{2.134584in}{0.417830in}}%
\pgfusepath{stroke}%
\end{pgfscope}%
\begin{pgfscope}%
\pgfpathrectangle{\pgfqpoint{0.520001in}{0.240679in}}{\pgfqpoint{1.937500in}{1.886500in}}%
\pgfusepath{clip}%
\pgfsetrectcap%
\pgfsetroundjoin%
\pgfsetlinewidth{1.003750pt}%
\definecolor{currentstroke}{rgb}{0.000000,0.000000,0.000000}%
\pgfsetstrokecolor{currentstroke}%
\pgfsetdash{}{0pt}%
\pgfpathmoveto{\pgfqpoint{2.086147in}{0.397887in}}%
\pgfpathlineto{\pgfqpoint{2.183022in}{0.397887in}}%
\pgfusepath{stroke}%
\end{pgfscope}%
\begin{pgfscope}%
\pgfpathrectangle{\pgfqpoint{0.520001in}{0.240679in}}{\pgfqpoint{1.937500in}{1.886500in}}%
\pgfusepath{clip}%
\pgfsetrectcap%
\pgfsetroundjoin%
\pgfsetlinewidth{1.003750pt}%
\definecolor{currentstroke}{rgb}{0.000000,0.000000,0.000000}%
\pgfsetstrokecolor{currentstroke}%
\pgfsetdash{}{0pt}%
\pgfpathmoveto{\pgfqpoint{2.086147in}{0.417830in}}%
\pgfpathlineto{\pgfqpoint{2.183022in}{0.417830in}}%
\pgfusepath{stroke}%
\end{pgfscope}%
\begin{pgfscope}%
\pgfpathrectangle{\pgfqpoint{0.520001in}{0.240679in}}{\pgfqpoint{1.937500in}{1.886500in}}%
\pgfusepath{clip}%
\pgfsetbuttcap%
\pgfsetroundjoin%
\definecolor{currentfill}{rgb}{0.000000,0.000000,0.000000}%
\pgfsetfillcolor{currentfill}%
\pgfsetfillopacity{0.000000}%
\pgfsetlinewidth{1.003750pt}%
\definecolor{currentstroke}{rgb}{0.000000,0.000000,0.000000}%
\pgfsetstrokecolor{currentstroke}%
\pgfsetdash{}{0pt}%
\pgfsys@defobject{currentmarker}{\pgfqpoint{-0.041667in}{-0.041667in}}{\pgfqpoint{0.041667in}{0.041667in}}{%
\pgfpathmoveto{\pgfqpoint{0.000000in}{-0.041667in}}%
\pgfpathcurveto{\pgfqpoint{0.011050in}{-0.041667in}}{\pgfqpoint{0.021649in}{-0.037276in}}{\pgfqpoint{0.029463in}{-0.029463in}}%
\pgfpathcurveto{\pgfqpoint{0.037276in}{-0.021649in}}{\pgfqpoint{0.041667in}{-0.011050in}}{\pgfqpoint{0.041667in}{0.000000in}}%
\pgfpathcurveto{\pgfqpoint{0.041667in}{0.011050in}}{\pgfqpoint{0.037276in}{0.021649in}}{\pgfqpoint{0.029463in}{0.029463in}}%
\pgfpathcurveto{\pgfqpoint{0.021649in}{0.037276in}}{\pgfqpoint{0.011050in}{0.041667in}}{\pgfqpoint{0.000000in}{0.041667in}}%
\pgfpathcurveto{\pgfqpoint{-0.011050in}{0.041667in}}{\pgfqpoint{-0.021649in}{0.037276in}}{\pgfqpoint{-0.029463in}{0.029463in}}%
\pgfpathcurveto{\pgfqpoint{-0.037276in}{0.021649in}}{\pgfqpoint{-0.041667in}{0.011050in}}{\pgfqpoint{-0.041667in}{0.000000in}}%
\pgfpathcurveto{\pgfqpoint{-0.041667in}{-0.011050in}}{\pgfqpoint{-0.037276in}{-0.021649in}}{\pgfqpoint{-0.029463in}{-0.029463in}}%
\pgfpathcurveto{\pgfqpoint{-0.021649in}{-0.037276in}}{\pgfqpoint{-0.011050in}{-0.041667in}}{\pgfqpoint{0.000000in}{-0.041667in}}%
\pgfpathlineto{\pgfqpoint{0.000000in}{-0.041667in}}%
\pgfpathclose%
\pgfusepath{stroke,fill}%
}%
\begin{pgfscope}%
\pgfsys@transformshift{2.134584in}{0.447565in}%
\pgfsys@useobject{currentmarker}{}%
\end{pgfscope}%
\end{pgfscope}%
\begin{pgfscope}%
\pgfpathrectangle{\pgfqpoint{0.520001in}{0.240679in}}{\pgfqpoint{1.937500in}{1.886500in}}%
\pgfusepath{clip}%
\pgfsetbuttcap%
\pgfsetroundjoin%
\pgfsetlinewidth{1.003750pt}%
\definecolor{currentstroke}{rgb}{0.000000,0.000000,0.000000}%
\pgfsetstrokecolor{currentstroke}%
\pgfsetdash{}{0pt}%
\pgfpathmoveto{\pgfqpoint{0.746042in}{1.554317in}}%
\pgfpathlineto{\pgfqpoint{0.939792in}{1.554317in}}%
\pgfusepath{stroke}%
\end{pgfscope}%
\begin{pgfscope}%
\pgfpathrectangle{\pgfqpoint{0.520001in}{0.240679in}}{\pgfqpoint{1.937500in}{1.886500in}}%
\pgfusepath{clip}%
\pgfsetbuttcap%
\pgfsetroundjoin%
\pgfsetlinewidth{1.003750pt}%
\definecolor{currentstroke}{rgb}{0.000000,0.000000,0.000000}%
\pgfsetstrokecolor{currentstroke}%
\pgfsetdash{}{0pt}%
\pgfpathmoveto{\pgfqpoint{1.391876in}{0.515834in}}%
\pgfpathlineto{\pgfqpoint{1.585626in}{0.515834in}}%
\pgfusepath{stroke}%
\end{pgfscope}%
\begin{pgfscope}%
\pgfpathrectangle{\pgfqpoint{0.520001in}{0.240679in}}{\pgfqpoint{1.937500in}{1.886500in}}%
\pgfusepath{clip}%
\pgfsetbuttcap%
\pgfsetroundjoin%
\pgfsetlinewidth{1.003750pt}%
\definecolor{currentstroke}{rgb}{0.000000,0.000000,0.000000}%
\pgfsetstrokecolor{currentstroke}%
\pgfsetdash{}{0pt}%
\pgfpathmoveto{\pgfqpoint{2.037709in}{0.402802in}}%
\pgfpathlineto{\pgfqpoint{2.231459in}{0.402802in}}%
\pgfusepath{stroke}%
\end{pgfscope}%
\begin{pgfscope}%
\pgfsetrectcap%
\pgfsetmiterjoin%
\pgfsetlinewidth{0.803000pt}%
\definecolor{currentstroke}{rgb}{0.000000,0.000000,0.000000}%
\pgfsetstrokecolor{currentstroke}%
\pgfsetdash{}{0pt}%
\pgfpathmoveto{\pgfqpoint{0.520001in}{0.240679in}}%
\pgfpathlineto{\pgfqpoint{0.520001in}{2.127179in}}%
\pgfusepath{stroke}%
\end{pgfscope}%
\begin{pgfscope}%
\pgfsetrectcap%
\pgfsetmiterjoin%
\pgfsetlinewidth{0.803000pt}%
\definecolor{currentstroke}{rgb}{0.000000,0.000000,0.000000}%
\pgfsetstrokecolor{currentstroke}%
\pgfsetdash{}{0pt}%
\pgfpathmoveto{\pgfqpoint{2.457501in}{0.240679in}}%
\pgfpathlineto{\pgfqpoint{2.457501in}{2.127179in}}%
\pgfusepath{stroke}%
\end{pgfscope}%
\begin{pgfscope}%
\pgfsetrectcap%
\pgfsetmiterjoin%
\pgfsetlinewidth{0.803000pt}%
\definecolor{currentstroke}{rgb}{0.000000,0.000000,0.000000}%
\pgfsetstrokecolor{currentstroke}%
\pgfsetdash{}{0pt}%
\pgfpathmoveto{\pgfqpoint{0.520001in}{0.240679in}}%
\pgfpathlineto{\pgfqpoint{2.457501in}{0.240679in}}%
\pgfusepath{stroke}%
\end{pgfscope}%
\begin{pgfscope}%
\pgfsetrectcap%
\pgfsetmiterjoin%
\pgfsetlinewidth{0.803000pt}%
\definecolor{currentstroke}{rgb}{0.000000,0.000000,0.000000}%
\pgfsetstrokecolor{currentstroke}%
\pgfsetdash{}{0pt}%
\pgfpathmoveto{\pgfqpoint{0.520001in}{2.127179in}}%
\pgfpathlineto{\pgfqpoint{2.457501in}{2.127179in}}%
\pgfusepath{stroke}%
\end{pgfscope}%
\end{pgfpicture}%
\makeatother%
\endgroup%

%% file: plots/comparison_large_snr_fits.pgf
\begingroup%
\makeatletter%
\begin{pgfpicture}%
\pgfpathrectangle{\pgfpointorigin}{\pgfqpoint{2.462068in}{2.195404in}}%
\pgfusepath{use as bounding box, clip}%
\begin{pgfscope}%
\pgfsetbuttcap%
\pgfsetmiterjoin%
\definecolor{currentfill}{rgb}{1.000000,1.000000,1.000000}%
\pgfsetfillcolor{currentfill}%
\pgfsetlinewidth{0.000000pt}%
\definecolor{currentstroke}{rgb}{1.000000,1.000000,1.000000}%
\pgfsetstrokecolor{currentstroke}%
\pgfsetdash{}{0pt}%
\pgfpathmoveto{\pgfqpoint{0.000000in}{0.000000in}}%
\pgfpathlineto{\pgfqpoint{2.462068in}{0.000000in}}%
\pgfpathlineto{\pgfqpoint{2.462068in}{2.195404in}}%
\pgfpathlineto{\pgfqpoint{0.000000in}{2.195404in}}%
\pgfpathlineto{\pgfqpoint{0.000000in}{0.000000in}}%
\pgfpathclose%
\pgfusepath{fill}%
\end{pgfscope}%
\begin{pgfscope}%
\pgfsetbuttcap%
\pgfsetmiterjoin%
\definecolor{currentfill}{rgb}{1.000000,1.000000,1.000000}%
\pgfsetfillcolor{currentfill}%
\pgfsetlinewidth{0.000000pt}%
\definecolor{currentstroke}{rgb}{0.000000,0.000000,0.000000}%
\pgfsetstrokecolor{currentstroke}%
\pgfsetstrokeopacity{0.000000}%
\pgfsetdash{}{0pt}%
\pgfpathmoveto{\pgfqpoint{0.504568in}{0.240679in}}%
\pgfpathlineto{\pgfqpoint{2.442068in}{0.240679in}}%
\pgfpathlineto{\pgfqpoint{2.442068in}{2.127179in}}%
\pgfpathlineto{\pgfqpoint{0.504568in}{2.127179in}}%
\pgfpathlineto{\pgfqpoint{0.504568in}{0.240679in}}%
\pgfpathclose%
\pgfusepath{fill}%
\end{pgfscope}%
\begin{pgfscope}%
\pgfpathrectangle{\pgfqpoint{0.504568in}{0.240679in}}{\pgfqpoint{1.937500in}{1.886500in}}%
\pgfusepath{clip}%
\pgfsetrectcap%
\pgfsetroundjoin%
\pgfsetlinewidth{0.250937pt}%
\definecolor{currentstroke}{rgb}{0.000000,0.000000,0.000000}%
\pgfsetstrokecolor{currentstroke}%
\pgfsetstrokeopacity{0.200000}%
\pgfsetdash{}{0pt}%
\pgfpathmoveto{\pgfqpoint{0.827485in}{0.240679in}}%
\pgfpathlineto{\pgfqpoint{0.827485in}{2.127179in}}%
\pgfusepath{stroke}%
\end{pgfscope}%
\begin{pgfscope}%
\pgfsetbuttcap%
\pgfsetroundjoin%
\definecolor{currentfill}{rgb}{0.000000,0.000000,0.000000}%
\pgfsetfillcolor{currentfill}%
\pgfsetlinewidth{0.803000pt}%
\definecolor{currentstroke}{rgb}{0.000000,0.000000,0.000000}%
\pgfsetstrokecolor{currentstroke}%
\pgfsetdash{}{0pt}%
\pgfsys@defobject{currentmarker}{\pgfqpoint{0.000000in}{-0.048611in}}{\pgfqpoint{0.000000in}{0.000000in}}{%
\pgfpathmoveto{\pgfqpoint{0.000000in}{0.000000in}}%
\pgfpathlineto{\pgfqpoint{0.000000in}{-0.048611in}}%
\pgfusepath{stroke,fill}%
}%
\begin{pgfscope}%
\pgfsys@transformshift{0.827485in}{0.240679in}%
\pgfsys@useobject{currentmarker}{}%
\end{pgfscope}%
\end{pgfscope}%
\begin{pgfscope}%
\definecolor{textcolor}{rgb}{0.000000,0.000000,0.000000}%
\pgfsetstrokecolor{textcolor}%
\pgfsetfillcolor{textcolor}%
\pgftext[x=0.827485in,y=0.143457in,,top]{\color{textcolor}{\rmfamily\fontsize{10.000000}{12.000000}\selectfont\catcode`\^=\active\def^{\ifmmode\sp\else\^{}\fi}\catcode`\%=\active\def
\end{pgfscope}%
\begin{pgfscope}%
\pgfpathrectangle{\pgfqpoint{0.504568in}{0.240679in}}{\pgfqpoint{1.937500in}{1.886500in}}%
\pgfusepath{clip}%
\pgfsetrectcap%
\pgfsetroundjoin%
\pgfsetlinewidth{0.250937pt}%
\definecolor{currentstroke}{rgb}{0.000000,0.000000,0.000000}%
\pgfsetstrokecolor{currentstroke}%
\pgfsetstrokeopacity{0.200000}%
\pgfsetdash{}{0pt}%
\pgfpathmoveto{\pgfqpoint{1.473318in}{0.240679in}}%
\pgfpathlineto{\pgfqpoint{1.473318in}{2.127179in}}%
\pgfusepath{stroke}%
\end{pgfscope}%
\begin{pgfscope}%
\pgfsetbuttcap%
\pgfsetroundjoin%
\definecolor{currentfill}{rgb}{0.000000,0.000000,0.000000}%
\pgfsetfillcolor{currentfill}%
\pgfsetlinewidth{0.803000pt}%
\definecolor{currentstroke}{rgb}{0.000000,0.000000,0.000000}%
\pgfsetstrokecolor{currentstroke}%
\pgfsetdash{}{0pt}%
\pgfsys@defobject{currentmarker}{\pgfqpoint{0.000000in}{-0.048611in}}{\pgfqpoint{0.000000in}{0.000000in}}{%
\pgfpathmoveto{\pgfqpoint{0.000000in}{0.000000in}}%
\pgfpathlineto{\pgfqpoint{0.000000in}{-0.048611in}}%
\pgfusepath{stroke,fill}%
}%
\begin{pgfscope}%
\pgfsys@transformshift{1.473318in}{0.240679in}%
\pgfsys@useobject{currentmarker}{}%
\end{pgfscope}%
\end{pgfscope}%
\begin{pgfscope}%
\definecolor{textcolor}{rgb}{0.000000,0.000000,0.000000}%
\pgfsetstrokecolor{textcolor}%
\pgfsetfillcolor{textcolor}%
\pgftext[x=1.473318in,y=0.143457in,,top]{\color{textcolor}{\rmfamily\fontsize{10.000000}{12.000000}\selectfont\catcode`\^=\active\def^{\ifmmode\sp\else\^{}\fi}\catcode`\%=\active\def
\end{pgfscope}%
\begin{pgfscope}%
\pgfpathrectangle{\pgfqpoint{0.504568in}{0.240679in}}{\pgfqpoint{1.937500in}{1.886500in}}%
\pgfusepath{clip}%
\pgfsetrectcap%
\pgfsetroundjoin%
\pgfsetlinewidth{0.250937pt}%
\definecolor{currentstroke}{rgb}{0.000000,0.000000,0.000000}%
\pgfsetstrokecolor{currentstroke}%
\pgfsetstrokeopacity{0.200000}%
\pgfsetdash{}{0pt}%
\pgfpathmoveto{\pgfqpoint{2.119152in}{0.240679in}}%
\pgfpathlineto{\pgfqpoint{2.119152in}{2.127179in}}%
\pgfusepath{stroke}%
\end{pgfscope}%
\begin{pgfscope}%
\pgfsetbuttcap%
\pgfsetroundjoin%
\definecolor{currentfill}{rgb}{0.000000,0.000000,0.000000}%
\pgfsetfillcolor{currentfill}%
\pgfsetlinewidth{0.803000pt}%
\definecolor{currentstroke}{rgb}{0.000000,0.000000,0.000000}%
\pgfsetstrokecolor{currentstroke}%
\pgfsetdash{}{0pt}%
\pgfsys@defobject{currentmarker}{\pgfqpoint{0.000000in}{-0.048611in}}{\pgfqpoint{0.000000in}{0.000000in}}{%
\pgfpathmoveto{\pgfqpoint{0.000000in}{0.000000in}}%
\pgfpathlineto{\pgfqpoint{0.000000in}{-0.048611in}}%
\pgfusepath{stroke,fill}%
}%
\begin{pgfscope}%
\pgfsys@transformshift{2.119152in}{0.240679in}%
\pgfsys@useobject{currentmarker}{}%
\end{pgfscope}%
\end{pgfscope}%
\begin{pgfscope}%
\definecolor{textcolor}{rgb}{0.000000,0.000000,0.000000}%
\pgfsetstrokecolor{textcolor}%
\pgfsetfillcolor{textcolor}%
\pgftext[x=2.119152in,y=0.143457in,,top]{\color{textcolor}{\rmfamily\fontsize{10.000000}{12.000000}\selectfont\catcode`\^=\active\def^{\ifmmode\sp\else\^{}\fi}\catcode`\%=\active\def
\end{pgfscope}%
\begin{pgfscope}%
\pgfpathrectangle{\pgfqpoint{0.504568in}{0.240679in}}{\pgfqpoint{1.937500in}{1.886500in}}%
\pgfusepath{clip}%
\pgfsetrectcap%
\pgfsetroundjoin%
\pgfsetlinewidth{0.250937pt}%
\definecolor{currentstroke}{rgb}{0.000000,0.000000,0.000000}%
\pgfsetstrokecolor{currentstroke}%
\pgfsetstrokeopacity{0.200000}%
\pgfsetdash{}{0pt}%
\pgfpathmoveto{\pgfqpoint{0.504568in}{0.240679in}}%
\pgfpathlineto{\pgfqpoint{2.442068in}{0.240679in}}%
\pgfusepath{stroke}%
\end{pgfscope}%
\begin{pgfscope}%
\pgfsetbuttcap%
\pgfsetroundjoin%
\definecolor{currentfill}{rgb}{0.000000,0.000000,0.000000}%
\pgfsetfillcolor{currentfill}%
\pgfsetlinewidth{0.803000pt}%
\definecolor{currentstroke}{rgb}{0.000000,0.000000,0.000000}%
\pgfsetstrokecolor{currentstroke}%
\pgfsetdash{}{0pt}%
\pgfsys@defobject{currentmarker}{\pgfqpoint{-0.048611in}{0.000000in}}{\pgfqpoint{-0.000000in}{0.000000in}}{%
\pgfpathmoveto{\pgfqpoint{-0.000000in}{0.000000in}}%
\pgfpathlineto{\pgfqpoint{-0.048611in}{0.000000in}}%
\pgfusepath{stroke,fill}%
}%
\begin{pgfscope}%
\pgfsys@transformshift{0.504568in}{0.240679in}%
\pgfsys@useobject{currentmarker}{}%
\end{pgfscope}%
\end{pgfscope}%
\begin{pgfscope}%
\definecolor{textcolor}{rgb}{0.000000,0.000000,0.000000}%
\pgfsetstrokecolor{textcolor}%
\pgfsetfillcolor{textcolor}%
\pgftext[x=0.268457in, y=0.192454in, left, base]{\color{textcolor}{\rmfamily\fontsize{10.000000}{12.000000}\selectfont\catcode`\^=\active\def^{\ifmmode\sp\else\^{}\fi}\catcode`\%=\active\def
\end{pgfscope}%
\begin{pgfscope}%
\pgfpathrectangle{\pgfqpoint{0.504568in}{0.240679in}}{\pgfqpoint{1.937500in}{1.886500in}}%
\pgfusepath{clip}%
\pgfsetrectcap%
\pgfsetroundjoin%
\pgfsetlinewidth{0.250937pt}%
\definecolor{currentstroke}{rgb}{0.000000,0.000000,0.000000}%
\pgfsetstrokecolor{currentstroke}%
\pgfsetstrokeopacity{0.200000}%
\pgfsetdash{}{0pt}%
\pgfpathmoveto{\pgfqpoint{0.504568in}{0.555096in}}%
\pgfpathlineto{\pgfqpoint{2.442068in}{0.555096in}}%
\pgfusepath{stroke}%
\end{pgfscope}%
\begin{pgfscope}%
\pgfsetbuttcap%
\pgfsetroundjoin%
\definecolor{currentfill}{rgb}{0.000000,0.000000,0.000000}%
\pgfsetfillcolor{currentfill}%
\pgfsetlinewidth{0.803000pt}%
\definecolor{currentstroke}{rgb}{0.000000,0.000000,0.000000}%
\pgfsetstrokecolor{currentstroke}%
\pgfsetdash{}{0pt}%
\pgfsys@defobject{currentmarker}{\pgfqpoint{-0.048611in}{0.000000in}}{\pgfqpoint{-0.000000in}{0.000000in}}{%
\pgfpathmoveto{\pgfqpoint{-0.000000in}{0.000000in}}%
\pgfpathlineto{\pgfqpoint{-0.048611in}{0.000000in}}%
\pgfusepath{stroke,fill}%
}%
\begin{pgfscope}%
\pgfsys@transformshift{0.504568in}{0.555096in}%
\pgfsys@useobject{currentmarker}{}%
\end{pgfscope}%
\end{pgfscope}%
\begin{pgfscope}%
\definecolor{textcolor}{rgb}{0.000000,0.000000,0.000000}%
\pgfsetstrokecolor{textcolor}%
\pgfsetfillcolor{textcolor}%
\pgftext[x=0.268457in, y=0.506870in, left, base]{\color{textcolor}{\rmfamily\fontsize{10.000000}{12.000000}\selectfont\catcode`\^=\active\def^{\ifmmode\sp\else\^{}\fi}\catcode`\%=\active\def
\end{pgfscope}%
\begin{pgfscope}%
\pgfpathrectangle{\pgfqpoint{0.504568in}{0.240679in}}{\pgfqpoint{1.937500in}{1.886500in}}%
\pgfusepath{clip}%
\pgfsetrectcap%
\pgfsetroundjoin%
\pgfsetlinewidth{0.250937pt}%
\definecolor{currentstroke}{rgb}{0.000000,0.000000,0.000000}%
\pgfsetstrokecolor{currentstroke}%
\pgfsetstrokeopacity{0.200000}%
\pgfsetdash{}{0pt}%
\pgfpathmoveto{\pgfqpoint{0.504568in}{0.869512in}}%
\pgfpathlineto{\pgfqpoint{2.442068in}{0.869512in}}%
\pgfusepath{stroke}%
\end{pgfscope}%
\begin{pgfscope}%
\pgfsetbuttcap%
\pgfsetroundjoin%
\definecolor{currentfill}{rgb}{0.000000,0.000000,0.000000}%
\pgfsetfillcolor{currentfill}%
\pgfsetlinewidth{0.803000pt}%
\definecolor{currentstroke}{rgb}{0.000000,0.000000,0.000000}%
\pgfsetstrokecolor{currentstroke}%
\pgfsetdash{}{0pt}%
\pgfsys@defobject{currentmarker}{\pgfqpoint{-0.048611in}{0.000000in}}{\pgfqpoint{-0.000000in}{0.000000in}}{%
\pgfpathmoveto{\pgfqpoint{-0.000000in}{0.000000in}}%
\pgfpathlineto{\pgfqpoint{-0.048611in}{0.000000in}}%
\pgfusepath{stroke,fill}%
}%
\begin{pgfscope}%
\pgfsys@transformshift{0.504568in}{0.869512in}%
\pgfsys@useobject{currentmarker}{}%
\end{pgfscope}%
\end{pgfscope}%
\begin{pgfscope}%
\definecolor{textcolor}{rgb}{0.000000,0.000000,0.000000}%
\pgfsetstrokecolor{textcolor}%
\pgfsetfillcolor{textcolor}%
\pgftext[x=0.268457in, y=0.821287in, left, base]{\color{textcolor}{\rmfamily\fontsize{10.000000}{12.000000}\selectfont\catcode`\^=\active\def^{\ifmmode\sp\else\^{}\fi}\catcode`\%=\active\def
\end{pgfscope}%
\begin{pgfscope}%
\pgfpathrectangle{\pgfqpoint{0.504568in}{0.240679in}}{\pgfqpoint{1.937500in}{1.886500in}}%
\pgfusepath{clip}%
\pgfsetrectcap%
\pgfsetroundjoin%
\pgfsetlinewidth{0.250937pt}%
\definecolor{currentstroke}{rgb}{0.000000,0.000000,0.000000}%
\pgfsetstrokecolor{currentstroke}%
\pgfsetstrokeopacity{0.200000}%
\pgfsetdash{}{0pt}%
\pgfpathmoveto{\pgfqpoint{0.504568in}{1.183929in}}%
\pgfpathlineto{\pgfqpoint{2.442068in}{1.183929in}}%
\pgfusepath{stroke}%
\end{pgfscope}%
\begin{pgfscope}%
\pgfsetbuttcap%
\pgfsetroundjoin%
\definecolor{currentfill}{rgb}{0.000000,0.000000,0.000000}%
\pgfsetfillcolor{currentfill}%
\pgfsetlinewidth{0.803000pt}%
\definecolor{currentstroke}{rgb}{0.000000,0.000000,0.000000}%
\pgfsetstrokecolor{currentstroke}%
\pgfsetdash{}{0pt}%
\pgfsys@defobject{currentmarker}{\pgfqpoint{-0.048611in}{0.000000in}}{\pgfqpoint{-0.000000in}{0.000000in}}{%
\pgfpathmoveto{\pgfqpoint{-0.000000in}{0.000000in}}%
\pgfpathlineto{\pgfqpoint{-0.048611in}{0.000000in}}%
\pgfusepath{stroke,fill}%
}%
\begin{pgfscope}%
\pgfsys@transformshift{0.504568in}{1.183929in}%
\pgfsys@useobject{currentmarker}{}%
\end{pgfscope}%
\end{pgfscope}%
\begin{pgfscope}%
\definecolor{textcolor}{rgb}{0.000000,0.000000,0.000000}%
\pgfsetstrokecolor{textcolor}%
\pgfsetfillcolor{textcolor}%
\pgftext[x=0.268457in, y=1.135704in, left, base]{\color{textcolor}{\rmfamily\fontsize{10.000000}{12.000000}\selectfont\catcode`\^=\active\def^{\ifmmode\sp\else\^{}\fi}\catcode`\%=\active\def
\end{pgfscope}%
\begin{pgfscope}%
\pgfpathrectangle{\pgfqpoint{0.504568in}{0.240679in}}{\pgfqpoint{1.937500in}{1.886500in}}%
\pgfusepath{clip}%
\pgfsetrectcap%
\pgfsetroundjoin%
\pgfsetlinewidth{0.250937pt}%
\definecolor{currentstroke}{rgb}{0.000000,0.000000,0.000000}%
\pgfsetstrokecolor{currentstroke}%
\pgfsetstrokeopacity{0.200000}%
\pgfsetdash{}{0pt}%
\pgfpathmoveto{\pgfqpoint{0.504568in}{1.498346in}}%
\pgfpathlineto{\pgfqpoint{2.442068in}{1.498346in}}%
\pgfusepath{stroke}%
\end{pgfscope}%
\begin{pgfscope}%
\pgfsetbuttcap%
\pgfsetroundjoin%
\definecolor{currentfill}{rgb}{0.000000,0.000000,0.000000}%
\pgfsetfillcolor{currentfill}%
\pgfsetlinewidth{0.803000pt}%
\definecolor{currentstroke}{rgb}{0.000000,0.000000,0.000000}%
\pgfsetstrokecolor{currentstroke}%
\pgfsetdash{}{0pt}%
\pgfsys@defobject{currentmarker}{\pgfqpoint{-0.048611in}{0.000000in}}{\pgfqpoint{-0.000000in}{0.000000in}}{%
\pgfpathmoveto{\pgfqpoint{-0.000000in}{0.000000in}}%
\pgfpathlineto{\pgfqpoint{-0.048611in}{0.000000in}}%
\pgfusepath{stroke,fill}%
}%
\begin{pgfscope}%
\pgfsys@transformshift{0.504568in}{1.498346in}%
\pgfsys@useobject{currentmarker}{}%
\end{pgfscope}%
\end{pgfscope}%
\begin{pgfscope}%
\definecolor{textcolor}{rgb}{0.000000,0.000000,0.000000}%
\pgfsetstrokecolor{textcolor}%
\pgfsetfillcolor{textcolor}%
\pgftext[x=0.268457in, y=1.450120in, left, base]{\color{textcolor}{\rmfamily\fontsize{10.000000}{12.000000}\selectfont\catcode`\^=\active\def^{\ifmmode\sp\else\^{}\fi}\catcode`\%=\active\def
\end{pgfscope}%
\begin{pgfscope}%
\pgfpathrectangle{\pgfqpoint{0.504568in}{0.240679in}}{\pgfqpoint{1.937500in}{1.886500in}}%
\pgfusepath{clip}%
\pgfsetrectcap%
\pgfsetroundjoin%
\pgfsetlinewidth{0.250937pt}%
\definecolor{currentstroke}{rgb}{0.000000,0.000000,0.000000}%
\pgfsetstrokecolor{currentstroke}%
\pgfsetstrokeopacity{0.200000}%
\pgfsetdash{}{0pt}%
\pgfpathmoveto{\pgfqpoint{0.504568in}{1.812762in}}%
\pgfpathlineto{\pgfqpoint{2.442068in}{1.812762in}}%
\pgfusepath{stroke}%
\end{pgfscope}%
\begin{pgfscope}%
\pgfsetbuttcap%
\pgfsetroundjoin%
\definecolor{currentfill}{rgb}{0.000000,0.000000,0.000000}%
\pgfsetfillcolor{currentfill}%
\pgfsetlinewidth{0.803000pt}%
\definecolor{currentstroke}{rgb}{0.000000,0.000000,0.000000}%
\pgfsetstrokecolor{currentstroke}%
\pgfsetdash{}{0pt}%
\pgfsys@defobject{currentmarker}{\pgfqpoint{-0.048611in}{0.000000in}}{\pgfqpoint{-0.000000in}{0.000000in}}{%
\pgfpathmoveto{\pgfqpoint{-0.000000in}{0.000000in}}%
\pgfpathlineto{\pgfqpoint{-0.048611in}{0.000000in}}%
\pgfusepath{stroke,fill}%
}%
\begin{pgfscope}%
\pgfsys@transformshift{0.504568in}{1.812762in}%
\pgfsys@useobject{currentmarker}{}%
\end{pgfscope}%
\end{pgfscope}%
\begin{pgfscope}%
\definecolor{textcolor}{rgb}{0.000000,0.000000,0.000000}%
\pgfsetstrokecolor{textcolor}%
\pgfsetfillcolor{textcolor}%
\pgftext[x=0.268457in, y=1.764537in, left, base]{\color{textcolor}{\rmfamily\fontsize{10.000000}{12.000000}\selectfont\catcode`\^=\active\def^{\ifmmode\sp\else\^{}\fi}\catcode`\%=\active\def
\end{pgfscope}%
\begin{pgfscope}%
\pgfpathrectangle{\pgfqpoint{0.504568in}{0.240679in}}{\pgfqpoint{1.937500in}{1.886500in}}%
\pgfusepath{clip}%
\pgfsetrectcap%
\pgfsetroundjoin%
\pgfsetlinewidth{0.250937pt}%
\definecolor{currentstroke}{rgb}{0.000000,0.000000,0.000000}%
\pgfsetstrokecolor{currentstroke}%
\pgfsetstrokeopacity{0.200000}%
\pgfsetdash{}{0pt}%
\pgfpathmoveto{\pgfqpoint{0.504568in}{2.127179in}}%
\pgfpathlineto{\pgfqpoint{2.442068in}{2.127179in}}%
\pgfusepath{stroke}%
\end{pgfscope}%
\begin{pgfscope}%
\pgfsetbuttcap%
\pgfsetroundjoin%
\definecolor{currentfill}{rgb}{0.000000,0.000000,0.000000}%
\pgfsetfillcolor{currentfill}%
\pgfsetlinewidth{0.803000pt}%
\definecolor{currentstroke}{rgb}{0.000000,0.000000,0.000000}%
\pgfsetstrokecolor{currentstroke}%
\pgfsetdash{}{0pt}%
\pgfsys@defobject{currentmarker}{\pgfqpoint{-0.048611in}{0.000000in}}{\pgfqpoint{-0.000000in}{0.000000in}}{%
\pgfpathmoveto{\pgfqpoint{-0.000000in}{0.000000in}}%
\pgfpathlineto{\pgfqpoint{-0.048611in}{0.000000in}}%
\pgfusepath{stroke,fill}%
}%
\begin{pgfscope}%
\pgfsys@transformshift{0.504568in}{2.127179in}%
\pgfsys@useobject{currentmarker}{}%
\end{pgfscope}%
\end{pgfscope}%
\begin{pgfscope}%
\definecolor{textcolor}{rgb}{0.000000,0.000000,0.000000}%
\pgfsetstrokecolor{textcolor}%
\pgfsetfillcolor{textcolor}%
\pgftext[x=0.199012in, y=2.078954in, left, base]{\color{textcolor}{\rmfamily\fontsize{10.000000}{12.000000}\selectfont\catcode`\^=\active\def^{\ifmmode\sp\else\^{}\fi}\catcode`\%=\active\def
\end{pgfscope}%
\begin{pgfscope}%
\definecolor{textcolor}{rgb}{0.000000,0.000000,0.000000}%
\pgfsetstrokecolor{textcolor}%
\pgfsetfillcolor{textcolor}%
\pgftext[x=0.143457in,y=1.183929in,,bottom,rotate=90.000000]{\color{textcolor}{\rmfamily\fontsize{10.000000}{12.000000}\selectfont\catcode`\^=\active\def^{\ifmmode\sp\else\^{}\fi}\catcode`\%=\active\def
\end{pgfscope}%
\begin{pgfscope}%
\pgfpathrectangle{\pgfqpoint{0.504568in}{0.240679in}}{\pgfqpoint{1.937500in}{1.886500in}}%
\pgfusepath{clip}%
\pgfsetbuttcap%
\pgfsetmiterjoin%
\definecolor{currentfill}{rgb}{1.000000,0.690196,0.000000}%
\pgfsetfillcolor{currentfill}%
\pgfsetlinewidth{1.003750pt}%
\definecolor{currentstroke}{rgb}{0.000000,0.000000,0.000000}%
\pgfsetstrokecolor{currentstroke}%
\pgfsetdash{}{0pt}%
\pgfpathmoveto{\pgfqpoint{0.730610in}{1.715589in}}%
\pgfpathlineto{\pgfqpoint{0.924360in}{1.715589in}}%
\pgfpathlineto{\pgfqpoint{0.924360in}{1.865358in}}%
\pgfpathlineto{\pgfqpoint{0.730610in}{1.865358in}}%
\pgfpathlineto{\pgfqpoint{0.730610in}{1.715589in}}%
\pgfpathlineto{\pgfqpoint{0.730610in}{1.715589in}}%
\pgfpathclose%
\pgfusepath{stroke,fill}%
\end{pgfscope}%
\begin{pgfscope}%
\pgfpathrectangle{\pgfqpoint{0.504568in}{0.240679in}}{\pgfqpoint{1.937500in}{1.886500in}}%
\pgfusepath{clip}%
\pgfsetrectcap%
\pgfsetroundjoin%
\pgfsetlinewidth{1.003750pt}%
\definecolor{currentstroke}{rgb}{0.000000,0.000000,0.000000}%
\pgfsetstrokecolor{currentstroke}%
\pgfsetdash{}{0pt}%
\pgfpathmoveto{\pgfqpoint{0.827485in}{1.715589in}}%
\pgfpathlineto{\pgfqpoint{0.827485in}{1.685458in}}%
\pgfusepath{stroke}%
\end{pgfscope}%
\begin{pgfscope}%
\pgfpathrectangle{\pgfqpoint{0.504568in}{0.240679in}}{\pgfqpoint{1.937500in}{1.886500in}}%
\pgfusepath{clip}%
\pgfsetrectcap%
\pgfsetroundjoin%
\pgfsetlinewidth{1.003750pt}%
\definecolor{currentstroke}{rgb}{0.000000,0.000000,0.000000}%
\pgfsetstrokecolor{currentstroke}%
\pgfsetdash{}{0pt}%
\pgfpathmoveto{\pgfqpoint{0.827485in}{1.865358in}}%
\pgfpathlineto{\pgfqpoint{0.827485in}{2.017065in}}%
\pgfusepath{stroke}%
\end{pgfscope}%
\begin{pgfscope}%
\pgfpathrectangle{\pgfqpoint{0.504568in}{0.240679in}}{\pgfqpoint{1.937500in}{1.886500in}}%
\pgfusepath{clip}%
\pgfsetrectcap%
\pgfsetroundjoin%
\pgfsetlinewidth{1.003750pt}%
\definecolor{currentstroke}{rgb}{0.000000,0.000000,0.000000}%
\pgfsetstrokecolor{currentstroke}%
\pgfsetdash{}{0pt}%
\pgfpathmoveto{\pgfqpoint{0.779048in}{1.685458in}}%
\pgfpathlineto{\pgfqpoint{0.875923in}{1.685458in}}%
\pgfusepath{stroke}%
\end{pgfscope}%
\begin{pgfscope}%
\pgfpathrectangle{\pgfqpoint{0.504568in}{0.240679in}}{\pgfqpoint{1.937500in}{1.886500in}}%
\pgfusepath{clip}%
\pgfsetrectcap%
\pgfsetroundjoin%
\pgfsetlinewidth{1.003750pt}%
\definecolor{currentstroke}{rgb}{0.000000,0.000000,0.000000}%
\pgfsetstrokecolor{currentstroke}%
\pgfsetdash{}{0pt}%
\pgfpathmoveto{\pgfqpoint{0.779048in}{2.017065in}}%
\pgfpathlineto{\pgfqpoint{0.875923in}{2.017065in}}%
\pgfusepath{stroke}%
\end{pgfscope}%
\begin{pgfscope}%
\pgfpathrectangle{\pgfqpoint{0.504568in}{0.240679in}}{\pgfqpoint{1.937500in}{1.886500in}}%
\pgfusepath{clip}%
\pgfsetbuttcap%
\pgfsetroundjoin%
\definecolor{currentfill}{rgb}{0.000000,0.000000,0.000000}%
\pgfsetfillcolor{currentfill}%
\pgfsetfillopacity{0.000000}%
\pgfsetlinewidth{1.003750pt}%
\definecolor{currentstroke}{rgb}{0.000000,0.000000,0.000000}%
\pgfsetstrokecolor{currentstroke}%
\pgfsetdash{}{0pt}%
\pgfsys@defobject{currentmarker}{\pgfqpoint{-0.041667in}{-0.041667in}}{\pgfqpoint{0.041667in}{0.041667in}}{%
\pgfpathmoveto{\pgfqpoint{0.000000in}{-0.041667in}}%
\pgfpathcurveto{\pgfqpoint{0.011050in}{-0.041667in}}{\pgfqpoint{0.021649in}{-0.037276in}}{\pgfqpoint{0.029463in}{-0.029463in}}%
\pgfpathcurveto{\pgfqpoint{0.037276in}{-0.021649in}}{\pgfqpoint{0.041667in}{-0.011050in}}{\pgfqpoint{0.041667in}{0.000000in}}%
\pgfpathcurveto{\pgfqpoint{0.041667in}{0.011050in}}{\pgfqpoint{0.037276in}{0.021649in}}{\pgfqpoint{0.029463in}{0.029463in}}%
\pgfpathcurveto{\pgfqpoint{0.021649in}{0.037276in}}{\pgfqpoint{0.011050in}{0.041667in}}{\pgfqpoint{0.000000in}{0.041667in}}%
\pgfpathcurveto{\pgfqpoint{-0.011050in}{0.041667in}}{\pgfqpoint{-0.021649in}{0.037276in}}{\pgfqpoint{-0.029463in}{0.029463in}}%
\pgfpathcurveto{\pgfqpoint{-0.037276in}{0.021649in}}{\pgfqpoint{-0.041667in}{0.011050in}}{\pgfqpoint{-0.041667in}{0.000000in}}%
\pgfpathcurveto{\pgfqpoint{-0.041667in}{-0.011050in}}{\pgfqpoint{-0.037276in}{-0.021649in}}{\pgfqpoint{-0.029463in}{-0.029463in}}%
\pgfpathcurveto{\pgfqpoint{-0.021649in}{-0.037276in}}{\pgfqpoint{-0.011050in}{-0.041667in}}{\pgfqpoint{0.000000in}{-0.041667in}}%
\pgfpathlineto{\pgfqpoint{0.000000in}{-0.041667in}}%
\pgfpathclose%
\pgfusepath{stroke,fill}%
}%
\begin{pgfscope}%
\pgfsys@transformshift{0.827485in}{1.256052in}%
\pgfsys@useobject{currentmarker}{}%
\end{pgfscope}%
\begin{pgfscope}%
\pgfsys@transformshift{0.827485in}{0.458124in}%
\pgfsys@useobject{currentmarker}{}%
\end{pgfscope}%
\begin{pgfscope}%
\pgfsys@transformshift{0.827485in}{-0.586365in}%
\pgfsys@useobject{currentmarker}{}%
\end{pgfscope}%
\end{pgfscope}%
\begin{pgfscope}%
\pgfpathrectangle{\pgfqpoint{0.504568in}{0.240679in}}{\pgfqpoint{1.937500in}{1.886500in}}%
\pgfusepath{clip}%
\pgfsetbuttcap%
\pgfsetmiterjoin%
\definecolor{currentfill}{rgb}{0.862745,0.149020,0.498039}%
\pgfsetfillcolor{currentfill}%
\pgfsetlinewidth{1.003750pt}%
\definecolor{currentstroke}{rgb}{0.000000,0.000000,0.000000}%
\pgfsetstrokecolor{currentstroke}%
\pgfsetdash{}{0pt}%
\pgfpathmoveto{\pgfqpoint{1.376443in}{1.447186in}}%
\pgfpathlineto{\pgfqpoint{1.570193in}{1.447186in}}%
\pgfpathlineto{\pgfqpoint{1.570193in}{1.776845in}}%
\pgfpathlineto{\pgfqpoint{1.376443in}{1.776845in}}%
\pgfpathlineto{\pgfqpoint{1.376443in}{1.447186in}}%
\pgfpathlineto{\pgfqpoint{1.376443in}{1.447186in}}%
\pgfpathclose%
\pgfusepath{stroke,fill}%
\end{pgfscope}%
\begin{pgfscope}%
\pgfpathrectangle{\pgfqpoint{0.504568in}{0.240679in}}{\pgfqpoint{1.937500in}{1.886500in}}%
\pgfusepath{clip}%
\pgfsetrectcap%
\pgfsetroundjoin%
\pgfsetlinewidth{1.003750pt}%
\definecolor{currentstroke}{rgb}{0.000000,0.000000,0.000000}%
\pgfsetstrokecolor{currentstroke}%
\pgfsetdash{}{0pt}%
\pgfpathmoveto{\pgfqpoint{1.473318in}{1.447186in}}%
\pgfpathlineto{\pgfqpoint{1.473318in}{1.126118in}}%
\pgfusepath{stroke}%
\end{pgfscope}%
\begin{pgfscope}%
\pgfpathrectangle{\pgfqpoint{0.504568in}{0.240679in}}{\pgfqpoint{1.937500in}{1.886500in}}%
\pgfusepath{clip}%
\pgfsetrectcap%
\pgfsetroundjoin%
\pgfsetlinewidth{1.003750pt}%
\definecolor{currentstroke}{rgb}{0.000000,0.000000,0.000000}%
\pgfsetstrokecolor{currentstroke}%
\pgfsetdash{}{0pt}%
\pgfpathmoveto{\pgfqpoint{1.473318in}{1.776845in}}%
\pgfpathlineto{\pgfqpoint{1.473318in}{1.882757in}}%
\pgfusepath{stroke}%
\end{pgfscope}%
\begin{pgfscope}%
\pgfpathrectangle{\pgfqpoint{0.504568in}{0.240679in}}{\pgfqpoint{1.937500in}{1.886500in}}%
\pgfusepath{clip}%
\pgfsetrectcap%
\pgfsetroundjoin%
\pgfsetlinewidth{1.003750pt}%
\definecolor{currentstroke}{rgb}{0.000000,0.000000,0.000000}%
\pgfsetstrokecolor{currentstroke}%
\pgfsetdash{}{0pt}%
\pgfpathmoveto{\pgfqpoint{1.424881in}{1.126118in}}%
\pgfpathlineto{\pgfqpoint{1.521756in}{1.126118in}}%
\pgfusepath{stroke}%
\end{pgfscope}%
\begin{pgfscope}%
\pgfpathrectangle{\pgfqpoint{0.504568in}{0.240679in}}{\pgfqpoint{1.937500in}{1.886500in}}%
\pgfusepath{clip}%
\pgfsetrectcap%
\pgfsetroundjoin%
\pgfsetlinewidth{1.003750pt}%
\definecolor{currentstroke}{rgb}{0.000000,0.000000,0.000000}%
\pgfsetstrokecolor{currentstroke}%
\pgfsetdash{}{0pt}%
\pgfpathmoveto{\pgfqpoint{1.424881in}{1.882757in}}%
\pgfpathlineto{\pgfqpoint{1.521756in}{1.882757in}}%
\pgfusepath{stroke}%
\end{pgfscope}%
\begin{pgfscope}%
\pgfpathrectangle{\pgfqpoint{0.504568in}{0.240679in}}{\pgfqpoint{1.937500in}{1.886500in}}%
\pgfusepath{clip}%
\pgfsetbuttcap%
\pgfsetroundjoin%
\definecolor{currentfill}{rgb}{0.000000,0.000000,0.000000}%
\pgfsetfillcolor{currentfill}%
\pgfsetfillopacity{0.000000}%
\pgfsetlinewidth{1.003750pt}%
\definecolor{currentstroke}{rgb}{0.000000,0.000000,0.000000}%
\pgfsetstrokecolor{currentstroke}%
\pgfsetdash{}{0pt}%
\pgfsys@defobject{currentmarker}{\pgfqpoint{-0.041667in}{-0.041667in}}{\pgfqpoint{0.041667in}{0.041667in}}{%
\pgfpathmoveto{\pgfqpoint{0.000000in}{-0.041667in}}%
\pgfpathcurveto{\pgfqpoint{0.011050in}{-0.041667in}}{\pgfqpoint{0.021649in}{-0.037276in}}{\pgfqpoint{0.029463in}{-0.029463in}}%
\pgfpathcurveto{\pgfqpoint{0.037276in}{-0.021649in}}{\pgfqpoint{0.041667in}{-0.011050in}}{\pgfqpoint{0.041667in}{0.000000in}}%
\pgfpathcurveto{\pgfqpoint{0.041667in}{0.011050in}}{\pgfqpoint{0.037276in}{0.021649in}}{\pgfqpoint{0.029463in}{0.029463in}}%
\pgfpathcurveto{\pgfqpoint{0.021649in}{0.037276in}}{\pgfqpoint{0.011050in}{0.041667in}}{\pgfqpoint{0.000000in}{0.041667in}}%
\pgfpathcurveto{\pgfqpoint{-0.011050in}{0.041667in}}{\pgfqpoint{-0.021649in}{0.037276in}}{\pgfqpoint{-0.029463in}{0.029463in}}%
\pgfpathcurveto{\pgfqpoint{-0.037276in}{0.021649in}}{\pgfqpoint{-0.041667in}{0.011050in}}{\pgfqpoint{-0.041667in}{0.000000in}}%
\pgfpathcurveto{\pgfqpoint{-0.041667in}{-0.011050in}}{\pgfqpoint{-0.037276in}{-0.021649in}}{\pgfqpoint{-0.029463in}{-0.029463in}}%
\pgfpathcurveto{\pgfqpoint{-0.021649in}{-0.037276in}}{\pgfqpoint{-0.011050in}{-0.041667in}}{\pgfqpoint{0.000000in}{-0.041667in}}%
\pgfpathlineto{\pgfqpoint{0.000000in}{-0.041667in}}%
\pgfpathclose%
\pgfusepath{stroke,fill}%
}%
\begin{pgfscope}%
\pgfsys@transformshift{1.473318in}{-2.679366in}%
\pgfsys@useobject{currentmarker}{}%
\end{pgfscope}%
\begin{pgfscope}%
\pgfsys@transformshift{1.473318in}{0.444891in}%
\pgfsys@useobject{currentmarker}{}%
\end{pgfscope}%
\begin{pgfscope}%
\pgfsys@transformshift{1.473318in}{0.547169in}%
\pgfsys@useobject{currentmarker}{}%
\end{pgfscope}%
\end{pgfscope}%
\begin{pgfscope}%
\pgfpathrectangle{\pgfqpoint{0.504568in}{0.240679in}}{\pgfqpoint{1.937500in}{1.886500in}}%
\pgfusepath{clip}%
\pgfsetbuttcap%
\pgfsetmiterjoin%
\definecolor{currentfill}{rgb}{0.392157,0.560784,1.000000}%
\pgfsetfillcolor{currentfill}%
\pgfsetlinewidth{1.003750pt}%
\definecolor{currentstroke}{rgb}{0.000000,0.000000,0.000000}%
\pgfsetstrokecolor{currentstroke}%
\pgfsetdash{}{0pt}%
\pgfpathmoveto{\pgfqpoint{2.022277in}{1.246373in}}%
\pgfpathlineto{\pgfqpoint{2.216027in}{1.246373in}}%
\pgfpathlineto{\pgfqpoint{2.216027in}{1.691205in}}%
\pgfpathlineto{\pgfqpoint{2.022277in}{1.691205in}}%
\pgfpathlineto{\pgfqpoint{2.022277in}{1.246373in}}%
\pgfpathlineto{\pgfqpoint{2.022277in}{1.246373in}}%
\pgfpathclose%
\pgfusepath{stroke,fill}%
\end{pgfscope}%
\begin{pgfscope}%
\pgfpathrectangle{\pgfqpoint{0.504568in}{0.240679in}}{\pgfqpoint{1.937500in}{1.886500in}}%
\pgfusepath{clip}%
\pgfsetrectcap%
\pgfsetroundjoin%
\pgfsetlinewidth{1.003750pt}%
\definecolor{currentstroke}{rgb}{0.000000,0.000000,0.000000}%
\pgfsetstrokecolor{currentstroke}%
\pgfsetdash{}{0pt}%
\pgfpathmoveto{\pgfqpoint{2.119152in}{1.246373in}}%
\pgfpathlineto{\pgfqpoint{2.119152in}{0.668217in}}%
\pgfusepath{stroke}%
\end{pgfscope}%
\begin{pgfscope}%
\pgfpathrectangle{\pgfqpoint{0.504568in}{0.240679in}}{\pgfqpoint{1.937500in}{1.886500in}}%
\pgfusepath{clip}%
\pgfsetrectcap%
\pgfsetroundjoin%
\pgfsetlinewidth{1.003750pt}%
\definecolor{currentstroke}{rgb}{0.000000,0.000000,0.000000}%
\pgfsetstrokecolor{currentstroke}%
\pgfsetdash{}{0pt}%
\pgfpathmoveto{\pgfqpoint{2.119152in}{1.691205in}}%
\pgfpathlineto{\pgfqpoint{2.119152in}{1.908692in}}%
\pgfusepath{stroke}%
\end{pgfscope}%
\begin{pgfscope}%
\pgfpathrectangle{\pgfqpoint{0.504568in}{0.240679in}}{\pgfqpoint{1.937500in}{1.886500in}}%
\pgfusepath{clip}%
\pgfsetrectcap%
\pgfsetroundjoin%
\pgfsetlinewidth{1.003750pt}%
\definecolor{currentstroke}{rgb}{0.000000,0.000000,0.000000}%
\pgfsetstrokecolor{currentstroke}%
\pgfsetdash{}{0pt}%
\pgfpathmoveto{\pgfqpoint{2.070714in}{0.668217in}}%
\pgfpathlineto{\pgfqpoint{2.167589in}{0.668217in}}%
\pgfusepath{stroke}%
\end{pgfscope}%
\begin{pgfscope}%
\pgfpathrectangle{\pgfqpoint{0.504568in}{0.240679in}}{\pgfqpoint{1.937500in}{1.886500in}}%
\pgfusepath{clip}%
\pgfsetrectcap%
\pgfsetroundjoin%
\pgfsetlinewidth{1.003750pt}%
\definecolor{currentstroke}{rgb}{0.000000,0.000000,0.000000}%
\pgfsetstrokecolor{currentstroke}%
\pgfsetdash{}{0pt}%
\pgfpathmoveto{\pgfqpoint{2.070714in}{1.908692in}}%
\pgfpathlineto{\pgfqpoint{2.167589in}{1.908692in}}%
\pgfusepath{stroke}%
\end{pgfscope}%
\begin{pgfscope}%
\pgfpathrectangle{\pgfqpoint{0.504568in}{0.240679in}}{\pgfqpoint{1.937500in}{1.886500in}}%
\pgfusepath{clip}%
\pgfsetbuttcap%
\pgfsetroundjoin%
\definecolor{currentfill}{rgb}{0.000000,0.000000,0.000000}%
\pgfsetfillcolor{currentfill}%
\pgfsetfillopacity{0.000000}%
\pgfsetlinewidth{1.003750pt}%
\definecolor{currentstroke}{rgb}{0.000000,0.000000,0.000000}%
\pgfsetstrokecolor{currentstroke}%
\pgfsetdash{}{0pt}%
\pgfsys@defobject{currentmarker}{\pgfqpoint{-0.041667in}{-0.041667in}}{\pgfqpoint{0.041667in}{0.041667in}}{%
\pgfpathmoveto{\pgfqpoint{0.000000in}{-0.041667in}}%
\pgfpathcurveto{\pgfqpoint{0.011050in}{-0.041667in}}{\pgfqpoint{0.021649in}{-0.037276in}}{\pgfqpoint{0.029463in}{-0.029463in}}%
\pgfpathcurveto{\pgfqpoint{0.037276in}{-0.021649in}}{\pgfqpoint{0.041667in}{-0.011050in}}{\pgfqpoint{0.041667in}{0.000000in}}%
\pgfpathcurveto{\pgfqpoint{0.041667in}{0.011050in}}{\pgfqpoint{0.037276in}{0.021649in}}{\pgfqpoint{0.029463in}{0.029463in}}%
\pgfpathcurveto{\pgfqpoint{0.021649in}{0.037276in}}{\pgfqpoint{0.011050in}{0.041667in}}{\pgfqpoint{0.000000in}{0.041667in}}%
\pgfpathcurveto{\pgfqpoint{-0.011050in}{0.041667in}}{\pgfqpoint{-0.021649in}{0.037276in}}{\pgfqpoint{-0.029463in}{0.029463in}}%
\pgfpathcurveto{\pgfqpoint{-0.037276in}{0.021649in}}{\pgfqpoint{-0.041667in}{0.011050in}}{\pgfqpoint{-0.041667in}{0.000000in}}%
\pgfpathcurveto{\pgfqpoint{-0.041667in}{-0.011050in}}{\pgfqpoint{-0.037276in}{-0.021649in}}{\pgfqpoint{-0.029463in}{-0.029463in}}%
\pgfpathcurveto{\pgfqpoint{-0.021649in}{-0.037276in}}{\pgfqpoint{-0.011050in}{-0.041667in}}{\pgfqpoint{0.000000in}{-0.041667in}}%
\pgfpathlineto{\pgfqpoint{0.000000in}{-0.041667in}}%
\pgfpathclose%
\pgfusepath{stroke,fill}%
}%
\begin{pgfscope}%
\pgfsys@transformshift{2.119152in}{0.424592in}%
\pgfsys@useobject{currentmarker}{}%
\end{pgfscope}%
\begin{pgfscope}%
\pgfsys@transformshift{2.119152in}{0.325165in}%
\pgfsys@useobject{currentmarker}{}%
\end{pgfscope}%
\begin{pgfscope}%
\pgfsys@transformshift{2.119152in}{0.216986in}%
\pgfsys@useobject{currentmarker}{}%
\end{pgfscope}%
\end{pgfscope}%
\begin{pgfscope}%
\pgfpathrectangle{\pgfqpoint{0.504568in}{0.240679in}}{\pgfqpoint{1.937500in}{1.886500in}}%
\pgfusepath{clip}%
\pgfsetbuttcap%
\pgfsetroundjoin%
\pgfsetlinewidth{1.003750pt}%
\definecolor{currentstroke}{rgb}{0.000000,0.000000,0.000000}%
\pgfsetstrokecolor{currentstroke}%
\pgfsetdash{}{0pt}%
\pgfpathmoveto{\pgfqpoint{0.730610in}{1.836148in}}%
\pgfpathlineto{\pgfqpoint{0.924360in}{1.836148in}}%
\pgfusepath{stroke}%
\end{pgfscope}%
\begin{pgfscope}%
\pgfpathrectangle{\pgfqpoint{0.504568in}{0.240679in}}{\pgfqpoint{1.937500in}{1.886500in}}%
\pgfusepath{clip}%
\pgfsetbuttcap%
\pgfsetroundjoin%
\pgfsetlinewidth{1.003750pt}%
\definecolor{currentstroke}{rgb}{0.000000,0.000000,0.000000}%
\pgfsetstrokecolor{currentstroke}%
\pgfsetdash{}{0pt}%
\pgfpathmoveto{\pgfqpoint{1.376443in}{1.707054in}}%
\pgfpathlineto{\pgfqpoint{1.570193in}{1.707054in}}%
\pgfusepath{stroke}%
\end{pgfscope}%
\begin{pgfscope}%
\pgfpathrectangle{\pgfqpoint{0.504568in}{0.240679in}}{\pgfqpoint{1.937500in}{1.886500in}}%
\pgfusepath{clip}%
\pgfsetbuttcap%
\pgfsetroundjoin%
\pgfsetlinewidth{1.003750pt}%
\definecolor{currentstroke}{rgb}{0.000000,0.000000,0.000000}%
\pgfsetstrokecolor{currentstroke}%
\pgfsetdash{}{0pt}%
\pgfpathmoveto{\pgfqpoint{2.022277in}{1.572269in}}%
\pgfpathlineto{\pgfqpoint{2.216027in}{1.572269in}}%
\pgfusepath{stroke}%
\end{pgfscope}%
\begin{pgfscope}%
\pgfsetrectcap%
\pgfsetmiterjoin%
\pgfsetlinewidth{0.803000pt}%
\definecolor{currentstroke}{rgb}{0.000000,0.000000,0.000000}%
\pgfsetstrokecolor{currentstroke}%
\pgfsetdash{}{0pt}%
\pgfpathmoveto{\pgfqpoint{0.504568in}{0.240679in}}%
\pgfpathlineto{\pgfqpoint{0.504568in}{2.127179in}}%
\pgfusepath{stroke}%
\end{pgfscope}%
\begin{pgfscope}%
\pgfsetrectcap%
\pgfsetmiterjoin%
\pgfsetlinewidth{0.803000pt}%
\definecolor{currentstroke}{rgb}{0.000000,0.000000,0.000000}%
\pgfsetstrokecolor{currentstroke}%
\pgfsetdash{}{0pt}%
\pgfpathmoveto{\pgfqpoint{2.442068in}{0.240679in}}%
\pgfpathlineto{\pgfqpoint{2.442068in}{2.127179in}}%
\pgfusepath{stroke}%
\end{pgfscope}%
\begin{pgfscope}%
\pgfsetrectcap%
\pgfsetmiterjoin%
\pgfsetlinewidth{0.803000pt}%
\definecolor{currentstroke}{rgb}{0.000000,0.000000,0.000000}%
\pgfsetstrokecolor{currentstroke}%
\pgfsetdash{}{0pt}%
\pgfpathmoveto{\pgfqpoint{0.504568in}{0.240679in}}%
\pgfpathlineto{\pgfqpoint{2.442068in}{0.240679in}}%
\pgfusepath{stroke}%
\end{pgfscope}%
\begin{pgfscope}%
\pgfsetrectcap%
\pgfsetmiterjoin%
\pgfsetlinewidth{0.803000pt}%
\definecolor{currentstroke}{rgb}{0.000000,0.000000,0.000000}%
\pgfsetstrokecolor{currentstroke}%
\pgfsetdash{}{0pt}%
\pgfpathmoveto{\pgfqpoint{0.504568in}{2.127179in}}%
\pgfpathlineto{\pgfqpoint{2.442068in}{2.127179in}}%
\pgfusepath{stroke}%
\end{pgfscope}%
\end{pgfpicture}%
\makeatother%
\endgroup%

%% file: plots/comparison_large_kernels_fits.pgf
\begingroup%
\makeatletter%
\begin{pgfpicture}%
\pgfpathrectangle{\pgfpointorigin}{\pgfqpoint{2.462068in}{2.195404in}}%
\pgfusepath{use as bounding box, clip}%
\begin{pgfscope}%
\pgfsetbuttcap%
\pgfsetmiterjoin%
\definecolor{currentfill}{rgb}{1.000000,1.000000,1.000000}%
\pgfsetfillcolor{currentfill}%
\pgfsetlinewidth{0.000000pt}%
\definecolor{currentstroke}{rgb}{1.000000,1.000000,1.000000}%
\pgfsetstrokecolor{currentstroke}%
\pgfsetdash{}{0pt}%
\pgfpathmoveto{\pgfqpoint{0.000000in}{0.000000in}}%
\pgfpathlineto{\pgfqpoint{2.462068in}{0.000000in}}%
\pgfpathlineto{\pgfqpoint{2.462068in}{2.195404in}}%
\pgfpathlineto{\pgfqpoint{0.000000in}{2.195404in}}%
\pgfpathlineto{\pgfqpoint{0.000000in}{0.000000in}}%
\pgfpathclose%
\pgfusepath{fill}%
\end{pgfscope}%
\begin{pgfscope}%
\pgfsetbuttcap%
\pgfsetmiterjoin%
\definecolor{currentfill}{rgb}{1.000000,1.000000,1.000000}%
\pgfsetfillcolor{currentfill}%
\pgfsetlinewidth{0.000000pt}%
\definecolor{currentstroke}{rgb}{0.000000,0.000000,0.000000}%
\pgfsetstrokecolor{currentstroke}%
\pgfsetstrokeopacity{0.000000}%
\pgfsetdash{}{0pt}%
\pgfpathmoveto{\pgfqpoint{0.504568in}{0.240679in}}%
\pgfpathlineto{\pgfqpoint{2.442068in}{0.240679in}}%
\pgfpathlineto{\pgfqpoint{2.442068in}{2.127179in}}%
\pgfpathlineto{\pgfqpoint{0.504568in}{2.127179in}}%
\pgfpathlineto{\pgfqpoint{0.504568in}{0.240679in}}%
\pgfpathclose%
\pgfusepath{fill}%
\end{pgfscope}%
\begin{pgfscope}%
\pgfpathrectangle{\pgfqpoint{0.504568in}{0.240679in}}{\pgfqpoint{1.937500in}{1.886500in}}%
\pgfusepath{clip}%
\pgfsetrectcap%
\pgfsetroundjoin%
\pgfsetlinewidth{0.250937pt}%
\definecolor{currentstroke}{rgb}{0.000000,0.000000,0.000000}%
\pgfsetstrokecolor{currentstroke}%
\pgfsetstrokeopacity{0.200000}%
\pgfsetdash{}{0pt}%
\pgfpathmoveto{\pgfqpoint{0.827485in}{0.240679in}}%
\pgfpathlineto{\pgfqpoint{0.827485in}{2.127179in}}%
\pgfusepath{stroke}%
\end{pgfscope}%
\begin{pgfscope}%
\pgfsetbuttcap%
\pgfsetroundjoin%
\definecolor{currentfill}{rgb}{0.000000,0.000000,0.000000}%
\pgfsetfillcolor{currentfill}%
\pgfsetlinewidth{0.803000pt}%
\definecolor{currentstroke}{rgb}{0.000000,0.000000,0.000000}%
\pgfsetstrokecolor{currentstroke}%
\pgfsetdash{}{0pt}%
\pgfsys@defobject{currentmarker}{\pgfqpoint{0.000000in}{-0.048611in}}{\pgfqpoint{0.000000in}{0.000000in}}{%
\pgfpathmoveto{\pgfqpoint{0.000000in}{0.000000in}}%
\pgfpathlineto{\pgfqpoint{0.000000in}{-0.048611in}}%
\pgfusepath{stroke,fill}%
}%
\begin{pgfscope}%
\pgfsys@transformshift{0.827485in}{0.240679in}%
\pgfsys@useobject{currentmarker}{}%
\end{pgfscope}%
\end{pgfscope}%
\begin{pgfscope}%
\definecolor{textcolor}{rgb}{0.000000,0.000000,0.000000}%
\pgfsetstrokecolor{textcolor}%
\pgfsetfillcolor{textcolor}%
\pgftext[x=0.827485in,y=0.143457in,,top]{\color{textcolor}{\rmfamily\fontsize{10.000000}{12.000000}\selectfont\catcode`\^=\active\def^{\ifmmode\sp\else\^{}\fi}\catcode`\%=\active\def
\end{pgfscope}%
\begin{pgfscope}%
\pgfpathrectangle{\pgfqpoint{0.504568in}{0.240679in}}{\pgfqpoint{1.937500in}{1.886500in}}%
\pgfusepath{clip}%
\pgfsetrectcap%
\pgfsetroundjoin%
\pgfsetlinewidth{0.250937pt}%
\definecolor{currentstroke}{rgb}{0.000000,0.000000,0.000000}%
\pgfsetstrokecolor{currentstroke}%
\pgfsetstrokeopacity{0.200000}%
\pgfsetdash{}{0pt}%
\pgfpathmoveto{\pgfqpoint{1.473318in}{0.240679in}}%
\pgfpathlineto{\pgfqpoint{1.473318in}{2.127179in}}%
\pgfusepath{stroke}%
\end{pgfscope}%
\begin{pgfscope}%
\pgfsetbuttcap%
\pgfsetroundjoin%
\definecolor{currentfill}{rgb}{0.000000,0.000000,0.000000}%
\pgfsetfillcolor{currentfill}%
\pgfsetlinewidth{0.803000pt}%
\definecolor{currentstroke}{rgb}{0.000000,0.000000,0.000000}%
\pgfsetstrokecolor{currentstroke}%
\pgfsetdash{}{0pt}%
\pgfsys@defobject{currentmarker}{\pgfqpoint{0.000000in}{-0.048611in}}{\pgfqpoint{0.000000in}{0.000000in}}{%
\pgfpathmoveto{\pgfqpoint{0.000000in}{0.000000in}}%
\pgfpathlineto{\pgfqpoint{0.000000in}{-0.048611in}}%
\pgfusepath{stroke,fill}%
}%
\begin{pgfscope}%
\pgfsys@transformshift{1.473318in}{0.240679in}%
\pgfsys@useobject{currentmarker}{}%
\end{pgfscope}%
\end{pgfscope}%
\begin{pgfscope}%
\definecolor{textcolor}{rgb}{0.000000,0.000000,0.000000}%
\pgfsetstrokecolor{textcolor}%
\pgfsetfillcolor{textcolor}%
\pgftext[x=1.473318in,y=0.143457in,,top]{\color{textcolor}{\rmfamily\fontsize{10.000000}{12.000000}\selectfont\catcode`\^=\active\def^{\ifmmode\sp\else\^{}\fi}\catcode`\%=\active\def
\end{pgfscope}%
\begin{pgfscope}%
\pgfpathrectangle{\pgfqpoint{0.504568in}{0.240679in}}{\pgfqpoint{1.937500in}{1.886500in}}%
\pgfusepath{clip}%
\pgfsetrectcap%
\pgfsetroundjoin%
\pgfsetlinewidth{0.250937pt}%
\definecolor{currentstroke}{rgb}{0.000000,0.000000,0.000000}%
\pgfsetstrokecolor{currentstroke}%
\pgfsetstrokeopacity{0.200000}%
\pgfsetdash{}{0pt}%
\pgfpathmoveto{\pgfqpoint{2.119152in}{0.240679in}}%
\pgfpathlineto{\pgfqpoint{2.119152in}{2.127179in}}%
\pgfusepath{stroke}%
\end{pgfscope}%
\begin{pgfscope}%
\pgfsetbuttcap%
\pgfsetroundjoin%
\definecolor{currentfill}{rgb}{0.000000,0.000000,0.000000}%
\pgfsetfillcolor{currentfill}%
\pgfsetlinewidth{0.803000pt}%
\definecolor{currentstroke}{rgb}{0.000000,0.000000,0.000000}%
\pgfsetstrokecolor{currentstroke}%
\pgfsetdash{}{0pt}%
\pgfsys@defobject{currentmarker}{\pgfqpoint{0.000000in}{-0.048611in}}{\pgfqpoint{0.000000in}{0.000000in}}{%
\pgfpathmoveto{\pgfqpoint{0.000000in}{0.000000in}}%
\pgfpathlineto{\pgfqpoint{0.000000in}{-0.048611in}}%
\pgfusepath{stroke,fill}%
}%
\begin{pgfscope}%
\pgfsys@transformshift{2.119152in}{0.240679in}%
\pgfsys@useobject{currentmarker}{}%
\end{pgfscope}%
\end{pgfscope}%
\begin{pgfscope}%
\definecolor{textcolor}{rgb}{0.000000,0.000000,0.000000}%
\pgfsetstrokecolor{textcolor}%
\pgfsetfillcolor{textcolor}%
\pgftext[x=2.119152in,y=0.143457in,,top]{\color{textcolor}{\rmfamily\fontsize{10.000000}{12.000000}\selectfont\catcode`\^=\active\def^{\ifmmode\sp\else\^{}\fi}\catcode`\%=\active\def
\end{pgfscope}%
\begin{pgfscope}%
\pgfpathrectangle{\pgfqpoint{0.504568in}{0.240679in}}{\pgfqpoint{1.937500in}{1.886500in}}%
\pgfusepath{clip}%
\pgfsetrectcap%
\pgfsetroundjoin%
\pgfsetlinewidth{0.250937pt}%
\definecolor{currentstroke}{rgb}{0.000000,0.000000,0.000000}%
\pgfsetstrokecolor{currentstroke}%
\pgfsetstrokeopacity{0.200000}%
\pgfsetdash{}{0pt}%
\pgfpathmoveto{\pgfqpoint{0.504568in}{0.240679in}}%
\pgfpathlineto{\pgfqpoint{2.442068in}{0.240679in}}%
\pgfusepath{stroke}%
\end{pgfscope}%
\begin{pgfscope}%
\pgfsetbuttcap%
\pgfsetroundjoin%
\definecolor{currentfill}{rgb}{0.000000,0.000000,0.000000}%
\pgfsetfillcolor{currentfill}%
\pgfsetlinewidth{0.803000pt}%
\definecolor{currentstroke}{rgb}{0.000000,0.000000,0.000000}%
\pgfsetstrokecolor{currentstroke}%
\pgfsetdash{}{0pt}%
\pgfsys@defobject{currentmarker}{\pgfqpoint{-0.048611in}{0.000000in}}{\pgfqpoint{-0.000000in}{0.000000in}}{%
\pgfpathmoveto{\pgfqpoint{-0.000000in}{0.000000in}}%
\pgfpathlineto{\pgfqpoint{-0.048611in}{0.000000in}}%
\pgfusepath{stroke,fill}%
}%
\begin{pgfscope}%
\pgfsys@transformshift{0.504568in}{0.240679in}%
\pgfsys@useobject{currentmarker}{}%
\end{pgfscope}%
\end{pgfscope}%
\begin{pgfscope}%
\definecolor{textcolor}{rgb}{0.000000,0.000000,0.000000}%
\pgfsetstrokecolor{textcolor}%
\pgfsetfillcolor{textcolor}%
\pgftext[x=0.268457in, y=0.192454in, left, base]{\color{textcolor}{\rmfamily\fontsize{10.000000}{12.000000}\selectfont\catcode`\^=\active\def^{\ifmmode\sp\else\^{}\fi}\catcode`\%=\active\def
\end{pgfscope}%
\begin{pgfscope}%
\pgfpathrectangle{\pgfqpoint{0.504568in}{0.240679in}}{\pgfqpoint{1.937500in}{1.886500in}}%
\pgfusepath{clip}%
\pgfsetrectcap%
\pgfsetroundjoin%
\pgfsetlinewidth{0.250937pt}%
\definecolor{currentstroke}{rgb}{0.000000,0.000000,0.000000}%
\pgfsetstrokecolor{currentstroke}%
\pgfsetstrokeopacity{0.200000}%
\pgfsetdash{}{0pt}%
\pgfpathmoveto{\pgfqpoint{0.504568in}{0.555096in}}%
\pgfpathlineto{\pgfqpoint{2.442068in}{0.555096in}}%
\pgfusepath{stroke}%
\end{pgfscope}%
\begin{pgfscope}%
\pgfsetbuttcap%
\pgfsetroundjoin%
\definecolor{currentfill}{rgb}{0.000000,0.000000,0.000000}%
\pgfsetfillcolor{currentfill}%
\pgfsetlinewidth{0.803000pt}%
\definecolor{currentstroke}{rgb}{0.000000,0.000000,0.000000}%
\pgfsetstrokecolor{currentstroke}%
\pgfsetdash{}{0pt}%
\pgfsys@defobject{currentmarker}{\pgfqpoint{-0.048611in}{0.000000in}}{\pgfqpoint{-0.000000in}{0.000000in}}{%
\pgfpathmoveto{\pgfqpoint{-0.000000in}{0.000000in}}%
\pgfpathlineto{\pgfqpoint{-0.048611in}{0.000000in}}%
\pgfusepath{stroke,fill}%
}%
\begin{pgfscope}%
\pgfsys@transformshift{0.504568in}{0.555096in}%
\pgfsys@useobject{currentmarker}{}%
\end{pgfscope}%
\end{pgfscope}%
\begin{pgfscope}%
\definecolor{textcolor}{rgb}{0.000000,0.000000,0.000000}%
\pgfsetstrokecolor{textcolor}%
\pgfsetfillcolor{textcolor}%
\pgftext[x=0.268457in, y=0.506870in, left, base]{\color{textcolor}{\rmfamily\fontsize{10.000000}{12.000000}\selectfont\catcode`\^=\active\def^{\ifmmode\sp\else\^{}\fi}\catcode`\%=\active\def
\end{pgfscope}%
\begin{pgfscope}%
\pgfpathrectangle{\pgfqpoint{0.504568in}{0.240679in}}{\pgfqpoint{1.937500in}{1.886500in}}%
\pgfusepath{clip}%
\pgfsetrectcap%
\pgfsetroundjoin%
\pgfsetlinewidth{0.250937pt}%
\definecolor{currentstroke}{rgb}{0.000000,0.000000,0.000000}%
\pgfsetstrokecolor{currentstroke}%
\pgfsetstrokeopacity{0.200000}%
\pgfsetdash{}{0pt}%
\pgfpathmoveto{\pgfqpoint{0.504568in}{0.869512in}}%
\pgfpathlineto{\pgfqpoint{2.442068in}{0.869512in}}%
\pgfusepath{stroke}%
\end{pgfscope}%
\begin{pgfscope}%
\pgfsetbuttcap%
\pgfsetroundjoin%
\definecolor{currentfill}{rgb}{0.000000,0.000000,0.000000}%
\pgfsetfillcolor{currentfill}%
\pgfsetlinewidth{0.803000pt}%
\definecolor{currentstroke}{rgb}{0.000000,0.000000,0.000000}%
\pgfsetstrokecolor{currentstroke}%
\pgfsetdash{}{0pt}%
\pgfsys@defobject{currentmarker}{\pgfqpoint{-0.048611in}{0.000000in}}{\pgfqpoint{-0.000000in}{0.000000in}}{%
\pgfpathmoveto{\pgfqpoint{-0.000000in}{0.000000in}}%
\pgfpathlineto{\pgfqpoint{-0.048611in}{0.000000in}}%
\pgfusepath{stroke,fill}%
}%
\begin{pgfscope}%
\pgfsys@transformshift{0.504568in}{0.869512in}%
\pgfsys@useobject{currentmarker}{}%
\end{pgfscope}%
\end{pgfscope}%
\begin{pgfscope}%
\definecolor{textcolor}{rgb}{0.000000,0.000000,0.000000}%
\pgfsetstrokecolor{textcolor}%
\pgfsetfillcolor{textcolor}%
\pgftext[x=0.268457in, y=0.821287in, left, base]{\color{textcolor}{\rmfamily\fontsize{10.000000}{12.000000}\selectfont\catcode`\^=\active\def^{\ifmmode\sp\else\^{}\fi}\catcode`\%=\active\def
\end{pgfscope}%
\begin{pgfscope}%
\pgfpathrectangle{\pgfqpoint{0.504568in}{0.240679in}}{\pgfqpoint{1.937500in}{1.886500in}}%
\pgfusepath{clip}%
\pgfsetrectcap%
\pgfsetroundjoin%
\pgfsetlinewidth{0.250937pt}%
\definecolor{currentstroke}{rgb}{0.000000,0.000000,0.000000}%
\pgfsetstrokecolor{currentstroke}%
\pgfsetstrokeopacity{0.200000}%
\pgfsetdash{}{0pt}%
\pgfpathmoveto{\pgfqpoint{0.504568in}{1.183929in}}%
\pgfpathlineto{\pgfqpoint{2.442068in}{1.183929in}}%
\pgfusepath{stroke}%
\end{pgfscope}%
\begin{pgfscope}%
\pgfsetbuttcap%
\pgfsetroundjoin%
\definecolor{currentfill}{rgb}{0.000000,0.000000,0.000000}%
\pgfsetfillcolor{currentfill}%
\pgfsetlinewidth{0.803000pt}%
\definecolor{currentstroke}{rgb}{0.000000,0.000000,0.000000}%
\pgfsetstrokecolor{currentstroke}%
\pgfsetdash{}{0pt}%
\pgfsys@defobject{currentmarker}{\pgfqpoint{-0.048611in}{0.000000in}}{\pgfqpoint{-0.000000in}{0.000000in}}{%
\pgfpathmoveto{\pgfqpoint{-0.000000in}{0.000000in}}%
\pgfpathlineto{\pgfqpoint{-0.048611in}{0.000000in}}%
\pgfusepath{stroke,fill}%
}%
\begin{pgfscope}%
\pgfsys@transformshift{0.504568in}{1.183929in}%
\pgfsys@useobject{currentmarker}{}%
\end{pgfscope}%
\end{pgfscope}%
\begin{pgfscope}%
\definecolor{textcolor}{rgb}{0.000000,0.000000,0.000000}%
\pgfsetstrokecolor{textcolor}%
\pgfsetfillcolor{textcolor}%
\pgftext[x=0.268457in, y=1.135704in, left, base]{\color{textcolor}{\rmfamily\fontsize{10.000000}{12.000000}\selectfont\catcode`\^=\active\def^{\ifmmode\sp\else\^{}\fi}\catcode`\%=\active\def
\end{pgfscope}%
\begin{pgfscope}%
\pgfpathrectangle{\pgfqpoint{0.504568in}{0.240679in}}{\pgfqpoint{1.937500in}{1.886500in}}%
\pgfusepath{clip}%
\pgfsetrectcap%
\pgfsetroundjoin%
\pgfsetlinewidth{0.250937pt}%
\definecolor{currentstroke}{rgb}{0.000000,0.000000,0.000000}%
\pgfsetstrokecolor{currentstroke}%
\pgfsetstrokeopacity{0.200000}%
\pgfsetdash{}{0pt}%
\pgfpathmoveto{\pgfqpoint{0.504568in}{1.498346in}}%
\pgfpathlineto{\pgfqpoint{2.442068in}{1.498346in}}%
\pgfusepath{stroke}%
\end{pgfscope}%
\begin{pgfscope}%
\pgfsetbuttcap%
\pgfsetroundjoin%
\definecolor{currentfill}{rgb}{0.000000,0.000000,0.000000}%
\pgfsetfillcolor{currentfill}%
\pgfsetlinewidth{0.803000pt}%
\definecolor{currentstroke}{rgb}{0.000000,0.000000,0.000000}%
\pgfsetstrokecolor{currentstroke}%
\pgfsetdash{}{0pt}%
\pgfsys@defobject{currentmarker}{\pgfqpoint{-0.048611in}{0.000000in}}{\pgfqpoint{-0.000000in}{0.000000in}}{%
\pgfpathmoveto{\pgfqpoint{-0.000000in}{0.000000in}}%
\pgfpathlineto{\pgfqpoint{-0.048611in}{0.000000in}}%
\pgfusepath{stroke,fill}%
}%
\begin{pgfscope}%
\pgfsys@transformshift{0.504568in}{1.498346in}%
\pgfsys@useobject{currentmarker}{}%
\end{pgfscope}%
\end{pgfscope}%
\begin{pgfscope}%
\definecolor{textcolor}{rgb}{0.000000,0.000000,0.000000}%
\pgfsetstrokecolor{textcolor}%
\pgfsetfillcolor{textcolor}%
\pgftext[x=0.268457in, y=1.450120in, left, base]{\color{textcolor}{\rmfamily\fontsize{10.000000}{12.000000}\selectfont\catcode`\^=\active\def^{\ifmmode\sp\else\^{}\fi}\catcode`\%=\active\def
\end{pgfscope}%
\begin{pgfscope}%
\pgfpathrectangle{\pgfqpoint{0.504568in}{0.240679in}}{\pgfqpoint{1.937500in}{1.886500in}}%
\pgfusepath{clip}%
\pgfsetrectcap%
\pgfsetroundjoin%
\pgfsetlinewidth{0.250937pt}%
\definecolor{currentstroke}{rgb}{0.000000,0.000000,0.000000}%
\pgfsetstrokecolor{currentstroke}%
\pgfsetstrokeopacity{0.200000}%
\pgfsetdash{}{0pt}%
\pgfpathmoveto{\pgfqpoint{0.504568in}{1.812762in}}%
\pgfpathlineto{\pgfqpoint{2.442068in}{1.812762in}}%
\pgfusepath{stroke}%
\end{pgfscope}%
\begin{pgfscope}%
\pgfsetbuttcap%
\pgfsetroundjoin%
\definecolor{currentfill}{rgb}{0.000000,0.000000,0.000000}%
\pgfsetfillcolor{currentfill}%
\pgfsetlinewidth{0.803000pt}%
\definecolor{currentstroke}{rgb}{0.000000,0.000000,0.000000}%
\pgfsetstrokecolor{currentstroke}%
\pgfsetdash{}{0pt}%
\pgfsys@defobject{currentmarker}{\pgfqpoint{-0.048611in}{0.000000in}}{\pgfqpoint{-0.000000in}{0.000000in}}{%
\pgfpathmoveto{\pgfqpoint{-0.000000in}{0.000000in}}%
\pgfpathlineto{\pgfqpoint{-0.048611in}{0.000000in}}%
\pgfusepath{stroke,fill}%
}%
\begin{pgfscope}%
\pgfsys@transformshift{0.504568in}{1.812762in}%
\pgfsys@useobject{currentmarker}{}%
\end{pgfscope}%
\end{pgfscope}%
\begin{pgfscope}%
\definecolor{textcolor}{rgb}{0.000000,0.000000,0.000000}%
\pgfsetstrokecolor{textcolor}%
\pgfsetfillcolor{textcolor}%
\pgftext[x=0.268457in, y=1.764537in, left, base]{\color{textcolor}{\rmfamily\fontsize{10.000000}{12.000000}\selectfont\catcode`\^=\active\def^{\ifmmode\sp\else\^{}\fi}\catcode`\%=\active\def
\end{pgfscope}%
\begin{pgfscope}%
\pgfpathrectangle{\pgfqpoint{0.504568in}{0.240679in}}{\pgfqpoint{1.937500in}{1.886500in}}%
\pgfusepath{clip}%
\pgfsetrectcap%
\pgfsetroundjoin%
\pgfsetlinewidth{0.250937pt}%
\definecolor{currentstroke}{rgb}{0.000000,0.000000,0.000000}%
\pgfsetstrokecolor{currentstroke}%
\pgfsetstrokeopacity{0.200000}%
\pgfsetdash{}{0pt}%
\pgfpathmoveto{\pgfqpoint{0.504568in}{2.127179in}}%
\pgfpathlineto{\pgfqpoint{2.442068in}{2.127179in}}%
\pgfusepath{stroke}%
\end{pgfscope}%
\begin{pgfscope}%
\pgfsetbuttcap%
\pgfsetroundjoin%
\definecolor{currentfill}{rgb}{0.000000,0.000000,0.000000}%
\pgfsetfillcolor{currentfill}%
\pgfsetlinewidth{0.803000pt}%
\definecolor{currentstroke}{rgb}{0.000000,0.000000,0.000000}%
\pgfsetstrokecolor{currentstroke}%
\pgfsetdash{}{0pt}%
\pgfsys@defobject{currentmarker}{\pgfqpoint{-0.048611in}{0.000000in}}{\pgfqpoint{-0.000000in}{0.000000in}}{%
\pgfpathmoveto{\pgfqpoint{-0.000000in}{0.000000in}}%
\pgfpathlineto{\pgfqpoint{-0.048611in}{0.000000in}}%
\pgfusepath{stroke,fill}%
}%
\begin{pgfscope}%
\pgfsys@transformshift{0.504568in}{2.127179in}%
\pgfsys@useobject{currentmarker}{}%
\end{pgfscope}%
\end{pgfscope}%
\begin{pgfscope}%
\definecolor{textcolor}{rgb}{0.000000,0.000000,0.000000}%
\pgfsetstrokecolor{textcolor}%
\pgfsetfillcolor{textcolor}%
\pgftext[x=0.199012in, y=2.078954in, left, base]{\color{textcolor}{\rmfamily\fontsize{10.000000}{12.000000}\selectfont\catcode`\^=\active\def^{\ifmmode\sp\else\^{}\fi}\catcode`\%=\active\def
\end{pgfscope}%
\begin{pgfscope}%
\definecolor{textcolor}{rgb}{0.000000,0.000000,0.000000}%
\pgfsetstrokecolor{textcolor}%
\pgfsetfillcolor{textcolor}%
\pgftext[x=0.143457in,y=1.183929in,,bottom,rotate=90.000000]{\color{textcolor}{\rmfamily\fontsize{10.000000}{12.000000}\selectfont\catcode`\^=\active\def^{\ifmmode\sp\else\^{}\fi}\catcode`\%=\active\def
\end{pgfscope}%
\begin{pgfscope}%
\pgfpathrectangle{\pgfqpoint{0.504568in}{0.240679in}}{\pgfqpoint{1.937500in}{1.886500in}}%
\pgfusepath{clip}%
\pgfsetbuttcap%
\pgfsetmiterjoin%
\definecolor{currentfill}{rgb}{1.000000,0.690196,0.000000}%
\pgfsetfillcolor{currentfill}%
\pgfsetlinewidth{1.003750pt}%
\definecolor{currentstroke}{rgb}{0.000000,0.000000,0.000000}%
\pgfsetstrokecolor{currentstroke}%
\pgfsetdash{}{0pt}%
\pgfpathmoveto{\pgfqpoint{0.730610in}{1.860554in}}%
\pgfpathlineto{\pgfqpoint{0.924360in}{1.860554in}}%
\pgfpathlineto{\pgfqpoint{0.924360in}{1.953992in}}%
\pgfpathlineto{\pgfqpoint{0.730610in}{1.953992in}}%
\pgfpathlineto{\pgfqpoint{0.730610in}{1.860554in}}%
\pgfpathlineto{\pgfqpoint{0.730610in}{1.860554in}}%
\pgfpathclose%
\pgfusepath{stroke,fill}%
\end{pgfscope}%
\begin{pgfscope}%
\pgfpathrectangle{\pgfqpoint{0.504568in}{0.240679in}}{\pgfqpoint{1.937500in}{1.886500in}}%
\pgfusepath{clip}%
\pgfsetrectcap%
\pgfsetroundjoin%
\pgfsetlinewidth{1.003750pt}%
\definecolor{currentstroke}{rgb}{0.000000,0.000000,0.000000}%
\pgfsetstrokecolor{currentstroke}%
\pgfsetdash{}{0pt}%
\pgfpathmoveto{\pgfqpoint{0.827485in}{1.860554in}}%
\pgfpathlineto{\pgfqpoint{0.827485in}{1.815411in}}%
\pgfusepath{stroke}%
\end{pgfscope}%
\begin{pgfscope}%
\pgfpathrectangle{\pgfqpoint{0.504568in}{0.240679in}}{\pgfqpoint{1.937500in}{1.886500in}}%
\pgfusepath{clip}%
\pgfsetrectcap%
\pgfsetroundjoin%
\pgfsetlinewidth{1.003750pt}%
\definecolor{currentstroke}{rgb}{0.000000,0.000000,0.000000}%
\pgfsetstrokecolor{currentstroke}%
\pgfsetdash{}{0pt}%
\pgfpathmoveto{\pgfqpoint{0.827485in}{1.953992in}}%
\pgfpathlineto{\pgfqpoint{0.827485in}{2.044948in}}%
\pgfusepath{stroke}%
\end{pgfscope}%
\begin{pgfscope}%
\pgfpathrectangle{\pgfqpoint{0.504568in}{0.240679in}}{\pgfqpoint{1.937500in}{1.886500in}}%
\pgfusepath{clip}%
\pgfsetrectcap%
\pgfsetroundjoin%
\pgfsetlinewidth{1.003750pt}%
\definecolor{currentstroke}{rgb}{0.000000,0.000000,0.000000}%
\pgfsetstrokecolor{currentstroke}%
\pgfsetdash{}{0pt}%
\pgfpathmoveto{\pgfqpoint{0.779048in}{1.815411in}}%
\pgfpathlineto{\pgfqpoint{0.875923in}{1.815411in}}%
\pgfusepath{stroke}%
\end{pgfscope}%
\begin{pgfscope}%
\pgfpathrectangle{\pgfqpoint{0.504568in}{0.240679in}}{\pgfqpoint{1.937500in}{1.886500in}}%
\pgfusepath{clip}%
\pgfsetrectcap%
\pgfsetroundjoin%
\pgfsetlinewidth{1.003750pt}%
\definecolor{currentstroke}{rgb}{0.000000,0.000000,0.000000}%
\pgfsetstrokecolor{currentstroke}%
\pgfsetdash{}{0pt}%
\pgfpathmoveto{\pgfqpoint{0.779048in}{2.044948in}}%
\pgfpathlineto{\pgfqpoint{0.875923in}{2.044948in}}%
\pgfusepath{stroke}%
\end{pgfscope}%
\begin{pgfscope}%
\pgfpathrectangle{\pgfqpoint{0.504568in}{0.240679in}}{\pgfqpoint{1.937500in}{1.886500in}}%
\pgfusepath{clip}%
\pgfsetbuttcap%
\pgfsetroundjoin%
\definecolor{currentfill}{rgb}{0.000000,0.000000,0.000000}%
\pgfsetfillcolor{currentfill}%
\pgfsetfillopacity{0.000000}%
\pgfsetlinewidth{1.003750pt}%
\definecolor{currentstroke}{rgb}{0.000000,0.000000,0.000000}%
\pgfsetstrokecolor{currentstroke}%
\pgfsetdash{}{0pt}%
\pgfsys@defobject{currentmarker}{\pgfqpoint{-0.041667in}{-0.041667in}}{\pgfqpoint{0.041667in}{0.041667in}}{%
\pgfpathmoveto{\pgfqpoint{0.000000in}{-0.041667in}}%
\pgfpathcurveto{\pgfqpoint{0.011050in}{-0.041667in}}{\pgfqpoint{0.021649in}{-0.037276in}}{\pgfqpoint{0.029463in}{-0.029463in}}%
\pgfpathcurveto{\pgfqpoint{0.037276in}{-0.021649in}}{\pgfqpoint{0.041667in}{-0.011050in}}{\pgfqpoint{0.041667in}{0.000000in}}%
\pgfpathcurveto{\pgfqpoint{0.041667in}{0.011050in}}{\pgfqpoint{0.037276in}{0.021649in}}{\pgfqpoint{0.029463in}{0.029463in}}%
\pgfpathcurveto{\pgfqpoint{0.021649in}{0.037276in}}{\pgfqpoint{0.011050in}{0.041667in}}{\pgfqpoint{0.000000in}{0.041667in}}%
\pgfpathcurveto{\pgfqpoint{-0.011050in}{0.041667in}}{\pgfqpoint{-0.021649in}{0.037276in}}{\pgfqpoint{-0.029463in}{0.029463in}}%
\pgfpathcurveto{\pgfqpoint{-0.037276in}{0.021649in}}{\pgfqpoint{-0.041667in}{0.011050in}}{\pgfqpoint{-0.041667in}{0.000000in}}%
\pgfpathcurveto{\pgfqpoint{-0.041667in}{-0.011050in}}{\pgfqpoint{-0.037276in}{-0.021649in}}{\pgfqpoint{-0.029463in}{-0.029463in}}%
\pgfpathcurveto{\pgfqpoint{-0.021649in}{-0.037276in}}{\pgfqpoint{-0.011050in}{-0.041667in}}{\pgfqpoint{0.000000in}{-0.041667in}}%
\pgfpathlineto{\pgfqpoint{0.000000in}{-0.041667in}}%
\pgfpathclose%
\pgfusepath{stroke,fill}%
}%
\begin{pgfscope}%
\pgfsys@transformshift{0.827485in}{0.464759in}%
\pgfsys@useobject{currentmarker}{}%
\end{pgfscope}%
\begin{pgfscope}%
\pgfsys@transformshift{0.827485in}{1.552642in}%
\pgfsys@useobject{currentmarker}{}%
\end{pgfscope}%
\end{pgfscope}%
\begin{pgfscope}%
\pgfpathrectangle{\pgfqpoint{0.504568in}{0.240679in}}{\pgfqpoint{1.937500in}{1.886500in}}%
\pgfusepath{clip}%
\pgfsetbuttcap%
\pgfsetmiterjoin%
\definecolor{currentfill}{rgb}{0.862745,0.149020,0.498039}%
\pgfsetfillcolor{currentfill}%
\pgfsetlinewidth{1.003750pt}%
\definecolor{currentstroke}{rgb}{0.000000,0.000000,0.000000}%
\pgfsetstrokecolor{currentstroke}%
\pgfsetdash{}{0pt}%
\pgfpathmoveto{\pgfqpoint{1.376443in}{1.615603in}}%
\pgfpathlineto{\pgfqpoint{1.570193in}{1.615603in}}%
\pgfpathlineto{\pgfqpoint{1.570193in}{1.850016in}}%
\pgfpathlineto{\pgfqpoint{1.376443in}{1.850016in}}%
\pgfpathlineto{\pgfqpoint{1.376443in}{1.615603in}}%
\pgfpathlineto{\pgfqpoint{1.376443in}{1.615603in}}%
\pgfpathclose%
\pgfusepath{stroke,fill}%
\end{pgfscope}%
\begin{pgfscope}%
\pgfpathrectangle{\pgfqpoint{0.504568in}{0.240679in}}{\pgfqpoint{1.937500in}{1.886500in}}%
\pgfusepath{clip}%
\pgfsetrectcap%
\pgfsetroundjoin%
\pgfsetlinewidth{1.003750pt}%
\definecolor{currentstroke}{rgb}{0.000000,0.000000,0.000000}%
\pgfsetstrokecolor{currentstroke}%
\pgfsetdash{}{0pt}%
\pgfpathmoveto{\pgfqpoint{1.473318in}{1.615603in}}%
\pgfpathlineto{\pgfqpoint{1.473318in}{1.523869in}}%
\pgfusepath{stroke}%
\end{pgfscope}%
\begin{pgfscope}%
\pgfpathrectangle{\pgfqpoint{0.504568in}{0.240679in}}{\pgfqpoint{1.937500in}{1.886500in}}%
\pgfusepath{clip}%
\pgfsetrectcap%
\pgfsetroundjoin%
\pgfsetlinewidth{1.003750pt}%
\definecolor{currentstroke}{rgb}{0.000000,0.000000,0.000000}%
\pgfsetstrokecolor{currentstroke}%
\pgfsetdash{}{0pt}%
\pgfpathmoveto{\pgfqpoint{1.473318in}{1.850016in}}%
\pgfpathlineto{\pgfqpoint{1.473318in}{1.956228in}}%
\pgfusepath{stroke}%
\end{pgfscope}%
\begin{pgfscope}%
\pgfpathrectangle{\pgfqpoint{0.504568in}{0.240679in}}{\pgfqpoint{1.937500in}{1.886500in}}%
\pgfusepath{clip}%
\pgfsetrectcap%
\pgfsetroundjoin%
\pgfsetlinewidth{1.003750pt}%
\definecolor{currentstroke}{rgb}{0.000000,0.000000,0.000000}%
\pgfsetstrokecolor{currentstroke}%
\pgfsetdash{}{0pt}%
\pgfpathmoveto{\pgfqpoint{1.424881in}{1.523869in}}%
\pgfpathlineto{\pgfqpoint{1.521756in}{1.523869in}}%
\pgfusepath{stroke}%
\end{pgfscope}%
\begin{pgfscope}%
\pgfpathrectangle{\pgfqpoint{0.504568in}{0.240679in}}{\pgfqpoint{1.937500in}{1.886500in}}%
\pgfusepath{clip}%
\pgfsetrectcap%
\pgfsetroundjoin%
\pgfsetlinewidth{1.003750pt}%
\definecolor{currentstroke}{rgb}{0.000000,0.000000,0.000000}%
\pgfsetstrokecolor{currentstroke}%
\pgfsetdash{}{0pt}%
\pgfpathmoveto{\pgfqpoint{1.424881in}{1.956228in}}%
\pgfpathlineto{\pgfqpoint{1.521756in}{1.956228in}}%
\pgfusepath{stroke}%
\end{pgfscope}%
\begin{pgfscope}%
\pgfpathrectangle{\pgfqpoint{0.504568in}{0.240679in}}{\pgfqpoint{1.937500in}{1.886500in}}%
\pgfusepath{clip}%
\pgfsetbuttcap%
\pgfsetroundjoin%
\definecolor{currentfill}{rgb}{0.000000,0.000000,0.000000}%
\pgfsetfillcolor{currentfill}%
\pgfsetfillopacity{0.000000}%
\pgfsetlinewidth{1.003750pt}%
\definecolor{currentstroke}{rgb}{0.000000,0.000000,0.000000}%
\pgfsetstrokecolor{currentstroke}%
\pgfsetdash{}{0pt}%
\pgfsys@defobject{currentmarker}{\pgfqpoint{-0.041667in}{-0.041667in}}{\pgfqpoint{0.041667in}{0.041667in}}{%
\pgfpathmoveto{\pgfqpoint{0.000000in}{-0.041667in}}%
\pgfpathcurveto{\pgfqpoint{0.011050in}{-0.041667in}}{\pgfqpoint{0.021649in}{-0.037276in}}{\pgfqpoint{0.029463in}{-0.029463in}}%
\pgfpathcurveto{\pgfqpoint{0.037276in}{-0.021649in}}{\pgfqpoint{0.041667in}{-0.011050in}}{\pgfqpoint{0.041667in}{0.000000in}}%
\pgfpathcurveto{\pgfqpoint{0.041667in}{0.011050in}}{\pgfqpoint{0.037276in}{0.021649in}}{\pgfqpoint{0.029463in}{0.029463in}}%
\pgfpathcurveto{\pgfqpoint{0.021649in}{0.037276in}}{\pgfqpoint{0.011050in}{0.041667in}}{\pgfqpoint{0.000000in}{0.041667in}}%
\pgfpathcurveto{\pgfqpoint{-0.011050in}{0.041667in}}{\pgfqpoint{-0.021649in}{0.037276in}}{\pgfqpoint{-0.029463in}{0.029463in}}%
\pgfpathcurveto{\pgfqpoint{-0.037276in}{0.021649in}}{\pgfqpoint{-0.041667in}{0.011050in}}{\pgfqpoint{-0.041667in}{0.000000in}}%
\pgfpathcurveto{\pgfqpoint{-0.041667in}{-0.011050in}}{\pgfqpoint{-0.037276in}{-0.021649in}}{\pgfqpoint{-0.029463in}{-0.029463in}}%
\pgfpathcurveto{\pgfqpoint{-0.021649in}{-0.037276in}}{\pgfqpoint{-0.011050in}{-0.041667in}}{\pgfqpoint{0.000000in}{-0.041667in}}%
\pgfpathlineto{\pgfqpoint{0.000000in}{-0.041667in}}%
\pgfpathclose%
\pgfusepath{stroke,fill}%
}%
\begin{pgfscope}%
\pgfsys@transformshift{1.473318in}{1.098900in}%
\pgfsys@useobject{currentmarker}{}%
\end{pgfscope}%
\begin{pgfscope}%
\pgfsys@transformshift{1.473318in}{0.457900in}%
\pgfsys@useobject{currentmarker}{}%
\end{pgfscope}%
\begin{pgfscope}%
\pgfsys@transformshift{1.473318in}{-0.681069in}%
\pgfsys@useobject{currentmarker}{}%
\end{pgfscope}%
\end{pgfscope}%
\begin{pgfscope}%
\pgfpathrectangle{\pgfqpoint{0.504568in}{0.240679in}}{\pgfqpoint{1.937500in}{1.886500in}}%
\pgfusepath{clip}%
\pgfsetbuttcap%
\pgfsetmiterjoin%
\definecolor{currentfill}{rgb}{0.392157,0.560784,1.000000}%
\pgfsetfillcolor{currentfill}%
\pgfsetlinewidth{1.003750pt}%
\definecolor{currentstroke}{rgb}{0.000000,0.000000,0.000000}%
\pgfsetstrokecolor{currentstroke}%
\pgfsetdash{}{0pt}%
\pgfpathmoveto{\pgfqpoint{2.022277in}{1.833054in}}%
\pgfpathlineto{\pgfqpoint{2.216027in}{1.833054in}}%
\pgfpathlineto{\pgfqpoint{2.216027in}{1.990973in}}%
\pgfpathlineto{\pgfqpoint{2.022277in}{1.990973in}}%
\pgfpathlineto{\pgfqpoint{2.022277in}{1.833054in}}%
\pgfpathlineto{\pgfqpoint{2.022277in}{1.833054in}}%
\pgfpathclose%
\pgfusepath{stroke,fill}%
\end{pgfscope}%
\begin{pgfscope}%
\pgfpathrectangle{\pgfqpoint{0.504568in}{0.240679in}}{\pgfqpoint{1.937500in}{1.886500in}}%
\pgfusepath{clip}%
\pgfsetrectcap%
\pgfsetroundjoin%
\pgfsetlinewidth{1.003750pt}%
\definecolor{currentstroke}{rgb}{0.000000,0.000000,0.000000}%
\pgfsetstrokecolor{currentstroke}%
\pgfsetdash{}{0pt}%
\pgfpathmoveto{\pgfqpoint{2.119152in}{1.833054in}}%
\pgfpathlineto{\pgfqpoint{2.119152in}{1.761672in}}%
\pgfusepath{stroke}%
\end{pgfscope}%
\begin{pgfscope}%
\pgfpathrectangle{\pgfqpoint{0.504568in}{0.240679in}}{\pgfqpoint{1.937500in}{1.886500in}}%
\pgfusepath{clip}%
\pgfsetrectcap%
\pgfsetroundjoin%
\pgfsetlinewidth{1.003750pt}%
\definecolor{currentstroke}{rgb}{0.000000,0.000000,0.000000}%
\pgfsetstrokecolor{currentstroke}%
\pgfsetdash{}{0pt}%
\pgfpathmoveto{\pgfqpoint{2.119152in}{1.990973in}}%
\pgfpathlineto{\pgfqpoint{2.119152in}{2.019451in}}%
\pgfusepath{stroke}%
\end{pgfscope}%
\begin{pgfscope}%
\pgfpathrectangle{\pgfqpoint{0.504568in}{0.240679in}}{\pgfqpoint{1.937500in}{1.886500in}}%
\pgfusepath{clip}%
\pgfsetrectcap%
\pgfsetroundjoin%
\pgfsetlinewidth{1.003750pt}%
\definecolor{currentstroke}{rgb}{0.000000,0.000000,0.000000}%
\pgfsetstrokecolor{currentstroke}%
\pgfsetdash{}{0pt}%
\pgfpathmoveto{\pgfqpoint{2.070714in}{1.761672in}}%
\pgfpathlineto{\pgfqpoint{2.167589in}{1.761672in}}%
\pgfusepath{stroke}%
\end{pgfscope}%
\begin{pgfscope}%
\pgfpathrectangle{\pgfqpoint{0.504568in}{0.240679in}}{\pgfqpoint{1.937500in}{1.886500in}}%
\pgfusepath{clip}%
\pgfsetrectcap%
\pgfsetroundjoin%
\pgfsetlinewidth{1.003750pt}%
\definecolor{currentstroke}{rgb}{0.000000,0.000000,0.000000}%
\pgfsetstrokecolor{currentstroke}%
\pgfsetdash{}{0pt}%
\pgfpathmoveto{\pgfqpoint{2.070714in}{2.019451in}}%
\pgfpathlineto{\pgfqpoint{2.167589in}{2.019451in}}%
\pgfusepath{stroke}%
\end{pgfscope}%
\begin{pgfscope}%
\pgfpathrectangle{\pgfqpoint{0.504568in}{0.240679in}}{\pgfqpoint{1.937500in}{1.886500in}}%
\pgfusepath{clip}%
\pgfsetbuttcap%
\pgfsetroundjoin%
\definecolor{currentfill}{rgb}{0.000000,0.000000,0.000000}%
\pgfsetfillcolor{currentfill}%
\pgfsetfillopacity{0.000000}%
\pgfsetlinewidth{1.003750pt}%
\definecolor{currentstroke}{rgb}{0.000000,0.000000,0.000000}%
\pgfsetstrokecolor{currentstroke}%
\pgfsetdash{}{0pt}%
\pgfsys@defobject{currentmarker}{\pgfqpoint{-0.041667in}{-0.041667in}}{\pgfqpoint{0.041667in}{0.041667in}}{%
\pgfpathmoveto{\pgfqpoint{0.000000in}{-0.041667in}}%
\pgfpathcurveto{\pgfqpoint{0.011050in}{-0.041667in}}{\pgfqpoint{0.021649in}{-0.037276in}}{\pgfqpoint{0.029463in}{-0.029463in}}%
\pgfpathcurveto{\pgfqpoint{0.037276in}{-0.021649in}}{\pgfqpoint{0.041667in}{-0.011050in}}{\pgfqpoint{0.041667in}{0.000000in}}%
\pgfpathcurveto{\pgfqpoint{0.041667in}{0.011050in}}{\pgfqpoint{0.037276in}{0.021649in}}{\pgfqpoint{0.029463in}{0.029463in}}%
\pgfpathcurveto{\pgfqpoint{0.021649in}{0.037276in}}{\pgfqpoint{0.011050in}{0.041667in}}{\pgfqpoint{0.000000in}{0.041667in}}%
\pgfpathcurveto{\pgfqpoint{-0.011050in}{0.041667in}}{\pgfqpoint{-0.021649in}{0.037276in}}{\pgfqpoint{-0.029463in}{0.029463in}}%
\pgfpathcurveto{\pgfqpoint{-0.037276in}{0.021649in}}{\pgfqpoint{-0.041667in}{0.011050in}}{\pgfqpoint{-0.041667in}{0.000000in}}%
\pgfpathcurveto{\pgfqpoint{-0.041667in}{-0.011050in}}{\pgfqpoint{-0.037276in}{-0.021649in}}{\pgfqpoint{-0.029463in}{-0.029463in}}%
\pgfpathcurveto{\pgfqpoint{-0.021649in}{-0.037276in}}{\pgfqpoint{-0.011050in}{-0.041667in}}{\pgfqpoint{0.000000in}{-0.041667in}}%
\pgfpathlineto{\pgfqpoint{0.000000in}{-0.041667in}}%
\pgfpathclose%
\pgfusepath{stroke,fill}%
}%
\begin{pgfscope}%
\pgfsys@transformshift{2.119152in}{1.489310in}%
\pgfsys@useobject{currentmarker}{}%
\end{pgfscope}%
\begin{pgfscope}%
\pgfsys@transformshift{2.119152in}{0.470379in}%
\pgfsys@useobject{currentmarker}{}%
\end{pgfscope}%
\end{pgfscope}%
\begin{pgfscope}%
\pgfpathrectangle{\pgfqpoint{0.504568in}{0.240679in}}{\pgfqpoint{1.937500in}{1.886500in}}%
\pgfusepath{clip}%
\pgfsetbuttcap%
\pgfsetroundjoin%
\pgfsetlinewidth{1.003750pt}%
\definecolor{currentstroke}{rgb}{0.000000,0.000000,0.000000}%
\pgfsetstrokecolor{currentstroke}%
\pgfsetdash{}{0pt}%
\pgfpathmoveto{\pgfqpoint{0.730610in}{1.915064in}}%
\pgfpathlineto{\pgfqpoint{0.924360in}{1.915064in}}%
\pgfusepath{stroke}%
\end{pgfscope}%
\begin{pgfscope}%
\pgfpathrectangle{\pgfqpoint{0.504568in}{0.240679in}}{\pgfqpoint{1.937500in}{1.886500in}}%
\pgfusepath{clip}%
\pgfsetbuttcap%
\pgfsetroundjoin%
\pgfsetlinewidth{1.003750pt}%
\definecolor{currentstroke}{rgb}{0.000000,0.000000,0.000000}%
\pgfsetstrokecolor{currentstroke}%
\pgfsetdash{}{0pt}%
\pgfpathmoveto{\pgfqpoint{1.376443in}{1.773339in}}%
\pgfpathlineto{\pgfqpoint{1.570193in}{1.773339in}}%
\pgfusepath{stroke}%
\end{pgfscope}%
\begin{pgfscope}%
\pgfpathrectangle{\pgfqpoint{0.504568in}{0.240679in}}{\pgfqpoint{1.937500in}{1.886500in}}%
\pgfusepath{clip}%
\pgfsetbuttcap%
\pgfsetroundjoin%
\pgfsetlinewidth{1.003750pt}%
\definecolor{currentstroke}{rgb}{0.000000,0.000000,0.000000}%
\pgfsetstrokecolor{currentstroke}%
\pgfsetdash{}{0pt}%
\pgfpathmoveto{\pgfqpoint{2.022277in}{1.934544in}}%
\pgfpathlineto{\pgfqpoint{2.216027in}{1.934544in}}%
\pgfusepath{stroke}%
\end{pgfscope}%
\begin{pgfscope}%
\pgfsetrectcap%
\pgfsetmiterjoin%
\pgfsetlinewidth{0.803000pt}%
\definecolor{currentstroke}{rgb}{0.000000,0.000000,0.000000}%
\pgfsetstrokecolor{currentstroke}%
\pgfsetdash{}{0pt}%
\pgfpathmoveto{\pgfqpoint{0.504568in}{0.240679in}}%
\pgfpathlineto{\pgfqpoint{0.504568in}{2.127179in}}%
\pgfusepath{stroke}%
\end{pgfscope}%
\begin{pgfscope}%
\pgfsetrectcap%
\pgfsetmiterjoin%
\pgfsetlinewidth{0.803000pt}%
\definecolor{currentstroke}{rgb}{0.000000,0.000000,0.000000}%
\pgfsetstrokecolor{currentstroke}%
\pgfsetdash{}{0pt}%
\pgfpathmoveto{\pgfqpoint{2.442068in}{0.240679in}}%
\pgfpathlineto{\pgfqpoint{2.442068in}{2.127179in}}%
\pgfusepath{stroke}%
\end{pgfscope}%
\begin{pgfscope}%
\pgfsetrectcap%
\pgfsetmiterjoin%
\pgfsetlinewidth{0.803000pt}%
\definecolor{currentstroke}{rgb}{0.000000,0.000000,0.000000}%
\pgfsetstrokecolor{currentstroke}%
\pgfsetdash{}{0pt}%
\pgfpathmoveto{\pgfqpoint{0.504568in}{0.240679in}}%
\pgfpathlineto{\pgfqpoint{2.442068in}{0.240679in}}%
\pgfusepath{stroke}%
\end{pgfscope}%
\begin{pgfscope}%
\pgfsetrectcap%
\pgfsetmiterjoin%
\pgfsetlinewidth{0.803000pt}%
\definecolor{currentstroke}{rgb}{0.000000,0.000000,0.000000}%
\pgfsetstrokecolor{currentstroke}%
\pgfsetdash{}{0pt}%
\pgfpathmoveto{\pgfqpoint{0.504568in}{2.127179in}}%
\pgfpathlineto{\pgfqpoint{2.442068in}{2.127179in}}%
\pgfusepath{stroke}%
\end{pgfscope}%
\end{pgfpicture}%
\makeatother%
\endgroup%